%% file: main1.tex
\documentclass[a4paper,10pt]{article}
\usepackage[utf8]{inputenc}

\providecommand{\keywords}[1]{\textbf{\textit{keywords}} #1}

\usepackage{amsthm}
\usepackage{amsmath}
\usepackage{amsfonts}
\usepackage{amssymb}
\usepackage{epsfig,epstopdf,titling}
\usepackage{caption}
\usepackage{subcaption}
\usepackage{multirow}
\usepackage{url}
\usepackage{graphicx}
\usepackage{pgfplots}
\usepackage{array,tabularx}
\usepackage{multicol}  
\usepackage{algorithmic} 
\usepackage[boxruled,longend,linesnumbered,ruled]{algorithm2e}
\usepackage{subcaption}
\theoremstyle{definition}

\newtheorem{theorem}{Theorem}[section]

\newtheorem{lemma}[theorem]{Lemma}

\theoremstyle{definition}
\newcommand\M{\ensuremath{\textbf{M}}}
\newcommand\D{\ensuremath{\textbf{D}}}
\newcommand\K{\ensuremath{\textbf{K}}}
\newcommand\F{\ensuremath{\textbf{F}}}
\newcommand\x{\ensuremath{\textbf{x}}}
\newcommand\y{\ensuremath{\textbf{y}}}
\newcommand\tu{\ensuremath{\textbf{u}}}
\newcommand\G{\ensuremath{\textbf{G}}}
\newcommand\z{\ensuremath{\textbf{z}}}

\usepackage{mathrsfs}

\theoremstyle{remark}

\usepackage{authblk}

\title{Iterative Rational Krylov Algorithms for model reduction of a class of constrained structural dynamic system with Engineering applications}

\author[]{Xin Du\thanks{ School of Mechatronic Engineering and Automation, Shanghai University, Shanghai-200072, China
 and   
Key Laboratory of Modern Power System Simulation and Control \& Renewable Energy Technology, Ministry of Education(Northeast Electric Power University), Jilin-132012, China
, {duxin@shu.edu.cn}}}
\author[]{M. Monir Uddin\thanks{Department of Mathematics and Physics, North south University, Dhaka-1229, Bangladesh, {monir.uddin@northsouth.edu}}}
\author[]{A. Mostakim Fony\thanks{Department of Mathematics, Chittagong University,  Chittagong, Bangladesh, {asibmostakim1995@gmail.com}}}
\author[]{Md. Tanzim Hossain\thanks{Department of Electrical and Computer Engineering, North South University, Dhaka-1229, Bangladesh, {tanzim.hossain@northsouth.edu}}}
\author[]{Md. Nazmul Islam Shuzan\thanks{Department of Electrical and Computer Engineering, North South University, Dhaka-1229, Bangladesh, {nazmul.shuzan@northsouth.edu }}}
\affil[]{}
\date{}
\begin{document}
\maketitle
\begin{abstract}
This paper discusses model order reduction of large sparse second-order index-3 differential algebraic equations (DAEs) by applying Iterative Rational Krylov Algorithm (IRKA). In general, such DAEs arise in constraint mechanics, multibody dynamics, mechatronics and many other branches of sciences and technologies. By deflecting the algebraic equations the second-order index-3 system can be altered into an equivalent standard second-order system. This can be done by projecting the system onto the null space of the constraint matrix. However, creating the projector is computationally expensive and it yields huge bottleneck during the  implementation. This paper shows how to find a reduce order model without projecting the system onto the  null space of the constraint matrix explicitly. To show the efficiency of the theoretical works we apply them to several data of  second-order index-3 models and experimental  resultants are discussed in the paper.
\end{abstract}
\begin{center}
\end{center}
\section{Introduction} \label{sec:intro}
In mechanics or multibody dynamics linearized equation of motion with holonomically constraint has the following form \cite{EicF98,morUdd18}
\begin{subequations} \label{eq:INT:systemequation}
\begin{align}
     \M\ddot{\x}(t)+\D \dot{\x}(t)+ \K\x(t)+\G^T \z(t) & = \F \tu(t),
     \label{eq:INT:systemequation1} \\
     \G\x(t)& =0, \label{eq:INT:systemequation2}\\
     \y(t) & = \textbf{L} \x (t),\label{eq:INT:systemequation3}
\end{align}          
\end{subequations} 
where  $\M,\,\K,\,\D \in \mathbb{R}^{n_1 \times n_1}$ are sparse matrices known as 
mass, stiffness and damping matrices respectively. $\x(t) \in \mathbb{R}^{n_1}$,  $\tu(t) \in\mathbb{R}^m$ and  $\y(t) \in \mathbb{R}^q$ are respectively known as states, inputs and outputs vectors. The constraint matrix $\G \in \mathbb{R}^{n_{2} \times n_{1}}$ (with $n_1 < n_2$) is associated with the given algebraic constraints $\z(t)\in\mathbb{R}^{n_2}$. Furthermore, $\F \in \mathbb{R}^{n_1 \times m}$ is the input matrix corresponding to the input vector $\tu(t)$ and $\textbf{L} \in \mathbb{R}^{q \times n_1}$ is  the output matrix associated to the measurement output vector $\y(t)$. Such structured dynamical system also  appear in mechatronics where electrical and mechanical parts are coupled or in the electric circuits \cite{Ria08}. 

If we convert the system into first-order form then it becomes first-order index-3 system \cite{morUdd18}. Therefore the system in (\ref{eq:INT:systemequation}) is called second-order index-3  descriptor system. If the system becomes very large then it is very expensive to simulate, control and optimize. Therefore, we want to reduce the complexity of the model through model order reduction (MOR) \cite{morAnt05, morBen05, morUdd19}. Among different MOR methods \cite{morAnt05, morBen05, morUdd19} the two most frequently applied modern MOR methods are the balanced truncation (BT) \cite{morMoo81} and the rational interpolation of the transfer function by the iterative rational Krylov algorithm (IRKA)~\cite{morGugAB08}. Both approaches have been extended to first-order descriptor systems~\cite{morHeiSS08,morGugSW13}. The balancing based model order reduction of second-order index-3 system (\ref{eq:INT:systemequation}) has been investigated for second-order-to-first-order and second-order-to-second-order reductions in \cite{morUdd18} and \cite{morUdd20} respectively. On the other hand, the authors in \cite{morAhmB14} discussed IRKA for the model reduction of the underlying descriptor system. In order to follow the proposed algorithm one has to convert the system into a first-order form. Besides at each iteration one has to solve a linear system with dimension $2(n_1+n_2)$ which is computationally expensive tasks.  

In this paper we discuss second-order-to-second-order model reduction of second-order index-3 descriptor system via IRKA without converting the system into first-order form. In the literature second-order-to-second-order reduction is called structure preserving model order reduction (SPMOR). IRKA based SPMOR for the standard second-order system was developed by Wyatt in his P.hD., thesis \cite{morWya12}. This idea was generalized for the second-order index-1 system which is slightly different from (\ref{eq:INT:systemequation}) in \cite{morRahUAM20}. Like index-1 system the second-order index-3 system can be converted into standard second-order system. In this case instead of using Schur complement techniques as used in  \cite{morRahUAM20} we apply projection onto hidden manifold. This idea was already found in \cite{morGugSW13,morHeiSS08,morBenSU16} for the firs-order index-2 systems. On the other hand, for the second-order index-3 system the technique was implemented using a balancing based model order reduction. However, there was no investigation of this idea for this system using the IRKA. This paper contributes to close this gap. That is we mainly devote to second-order-to-second-order model reduction of second-order index-3 system using IRKA. Following the procedure in \cite{morGugSW13,morHeiSS08} first we show that the second-order index-3 descriptor system can be projected onto the null space of the constraint matrix which we call hidden manifold to obtain a standard second-order system. Then we can apply the technique as  in \cite{morWya12} to obtain a standard second-order reduced order system. It is shown in the paper that the explicit computation of hidden manifold projector is not required. This is important because creating a hidden manifold projector demands a lot of computational times. Moreover, the projected system is converted into a dense form which yields huge bottleneck in implementing the reduced order model. The proposed method is applied to several models coming from Engineering applications. The performances of the proposed algorithm seems to be promising and the results are better than balancing based techniques in both approximation accuracy and  computational time which appears in numerical results. 

Rest of the article is organized as follows. Section~\ref{sec:btforsecondorder} briefly discuss IRKA based SPMOR of second-order system and reformulation of second-order index-3 system from previous literature which are the main ingredient to obtain the new results of this paper. The main contribution of this paper will be discussed in Section~\ref{sec:hindex:reformofdynamicalsystem}. In this section we developed IRKA based SPMOR of second-order index-3 system.  The subsequent section  illustrates numerical results. At the end, Section~\ref{sec:conclusion} presents the conclusive remarks.
\section{Background}\label{sec:btforsecondorder}
In the following texts at first we briefly discuss IRKA for standard second-order system to obtain second-order reduced order model. Then we will show how to convert the second-order index-3 system into second-order standard system by projecting onto the null space of the constraint matrix. In fact we establish some definitions and notations based on the previous literature that will be in the upcoming sections. 
\subsection{IRKA for second-order system}\label{subsec:itmsecstd}
Structure preserving IRKA (SPIRKA) for a second-order standard system was proposed in \cite{morWya12}. The SPIRKA is mainly based on the IRKA of first-order system which was originally proposed in \cite{morGugAB08}. This prominent algorithm was developed by Gugercin et al., in \cite{morGugAB08} to achieve the $\mathscr{H}_2$-optimal model reduction via interpolatoy projection technique. To explain the SPIRKA let us consider a second-order linear time-invariant (LTI) continuous-time system 
\begin{equation}\label{eq:tisec:systemequation}
\begin{aligned}
     \mathcal{M}\ddot{\xi}(t)+ \mathcal{D} \dot{\xi}(t)+\mathcal{K}\xi(t) & = \mathcal{F} u(t),\\
     y(t) & = \mathcal{L} \xi(t),
\end{aligned}   
\end{equation}
where $\mathcal{M}, \mathcal{D}$ and $\mathcal{K}$ are non-singular, and $\xi(t)$ is the $n$ dimensional state vector. Consider that the system is MIMO and its transfer function is defined by
\begin{align}\label{eq:back:stsectf}
 \mathcal{T}(s)= \mathcal{L} (s^2\mathcal{M}+s\mathcal{D}+\mathcal{K})^{-1}\mathcal{F}; \quad s\in \mathbb{C}.
\end{align}
Our goal is to obtain an $r$ dimensional $(r\ll n)$  reduce order model
\begin{equation}\label{eq:back:reducedsystem}
\begin{aligned}
     \hat{\mathcal{M}}\ddot{\hat{\xi}}(t)+ \hat{\mathcal{D}} \dot{\hat{\xi}}(t)+\hat{\mathcal{K}}\hat{\xi}(t) & = \hat{\mathcal{F}} u(t),\\ 
     \hat{y}(t) & = \hat{\mathcal{L}} \hat{\xi}(t),
\end{aligned}
\end{equation}
where the reduced coefficient matrices are constructed as
\begin{equation}
\begin{aligned}
 \hat{\mathcal{M}} = \mathcal{W}^T \mathcal{M}\mathcal{V},\, 
 \hat{\mathcal{D}} = \mathcal{W}^T \mathcal{D} \mathcal{V},\, 
 \hat{\mathcal{K}} = \mathcal{W}^T \mathcal{K} \mathcal{V},\\
 \hat{\mathcal{F}} = \mathcal{W}^T\mathcal{F},\, 
 \hat{\mathcal{L}} = \mathcal{L}\mathcal{V}, 
\end{aligned}
\end{equation}
and the transfer function of the reduced order model can be defined as
\begin{align}\label{eq:back:rtf}
 \hat{\mathcal{T}}(s)= \hat{\mathcal{L}} (s^2\hat{\mathcal{M}}+s\hat{\mathcal{D}}+\hat{\mathcal{K}})^{-1}\hat{\mathcal{F}}; \quad s\in \mathbb{C}.
\end{align}
According to \cite{morWya12} the procedure of IRKA for second-order system is same as
the first-order system. We want to construct reduced order model (\ref{eq:back:reducedsystem}) in such way that the reduced transfer function (\ref{eq:back:rtf}) interpolate to the original transfer function (\ref{eq:back:stsectf}) at some interpolation points. Moreover the reduced order model satisfies the interpolation conditions mentioned below
\cite{morWya12}.

Given a set of interpolation points $\left\{\alpha_1, \alpha_2, \cdots, \alpha_r\right\}\subset \mathbb{C}$, and sets of left and right tangential directions  $\left\{c_1, c_2, \cdots, c_r\right\}\subset \mathbb{C}^m$ , $\left\{b_1, b_2, \cdots, b_r\right\}\subset \mathbb{C}^p$ are respectively defined by 
\begin{equation}\label{eq:back:left-rightproject}
\begin{aligned}
V & = \left[v_1\mathcal{F}b_1, v_2\mathcal{F}b_2, \cdots,v_r\mathcal{F}b_r\right],\\
W & = \left[w_1\mathcal{L}^Tc_1, w_2\mathcal{L}^Tc_2 \cdots, w_r\mathcal{L}^T c_r\right],
\end{aligned}
\end{equation}
Where $v_i=(\alpha_i^2 \mathcal{M}+\alpha_i\mathcal{D}+\mathcal{K})^{-1}$ and $w_i=(\alpha_i^2 \mathcal{M}^T+\alpha_i\mathcal{D}^T+\mathcal{K}^T)^{-1}$; $i= 1,2,\cdots r$.
If the reduced-order model (\ref{eq:back:reducedsystem}) is constructed by  $V$ and $W$,
the reduced transfer-function (\ref{eq:back:rtf}) tangentially interpolates (\ref{eq:back:stsectf}), satisfies the  interpolation conditions
\begin{equation}\label{eq:back:hermitcond}
\begin{aligned}
\mathcal{T}(\alpha_i)b_i & = \hat{\mathcal{T}}(\alpha_i)b_i,\\
c_i^T\mathcal{T}(\alpha_i)b_i & = c_i^T\hat{\mathcal{T}}(\alpha_i)b_i,\\
c_i^T\mathcal{T}'(\alpha_i)b_i &= c_i^T\hat{\mathcal{T}}'(\alpha_i)b_i,
\end{aligned}
\end{equation}
for $i = 1,2, \dots , r,$ which is known as Hermite bi-tangential interpolation conditions. One of the challenging parts of SPIRKA is to find a set of optimal interpolation points as well as tangential directions since they are not predefined. In \cite{morWya12} author shows several  remedies. Among them this paper consider the following strategy.
Construct 
{\small
\begin{align}\label{eq:back:2ndto1st}
(E,~ A,~ B,~ C):= \left(\begin{bmatrix}I & 0\\ 0 &                                                       \hat{\mathcal{M}}\end{bmatrix},\
                         \begin{bmatrix}0 & I\\                                     -\hat{\mathcal{K}} & -\hat{\mathcal{D}} \end{bmatrix},\
                         \begin{bmatrix} 0 \\  \hat{\mathcal{F}} \end{bmatrix},\
                         \begin{bmatrix} \hat{\mathcal{L}} & 0                          \end{bmatrix} 
                         \right).
\end{align} 
}
Then find $r$ dimensional reduced order model $(\hat{E}, \hat{A}, \hat{B}, \hat{C})$ from $(E, A, B, C)$. The interpolation points and tangential directions for the next iteration step are  constructed from the mirror image of the eigenvalues and the eigenvectors of $(A,E)$. The  reduced order model  $(\hat{E}, \hat{A}, \hat{B}, \hat{C})$ can be constructed again by the IRKA. Note that IRKA of first-order system is presented in \cite[Algorithm 1]{morUdd19}. The whole procedure for SPIRKA is summarized in Algorithm~\ref{alg:irka3}.
\begin{algorithm}[t]
\SetAlgoLined
\SetKwInOut{Input}{Input}
\SetKwInOut{Output}{Output}
\caption{SPIRKA for standard Second-Order Systems.}
\label{alg:irka3}
\Input {$\mathcal{M}, \mathcal{D}, \mathcal{K}, \mathcal{H}, \mathcal{L}$.}
\Output  {$\hat{\mathcal{M}}, \hat{\mathcal{D}}, \hat{\mathcal{K}}, \hat{\mathcal{H}}, \hat{\mathcal{L}}$.}
Consider: interpolation points $\{\alpha_i\}_{i=1}^r\subset \mathbb{C}$, left and right tangential directions $\{b_i\}_{i=1}^r\subset \mathbb{C}^p$ and $\{c_i\}_{i=1}^r\subset \mathbb{C}^m$.\\
Form
\begin{align*}
V & = \left[v_1\mathcal{F}b_1, v_2\mathcal{F}b_2, \cdots,v_r\mathcal{F}b_r\right],\\
W & = \left[w_1\mathcal{L}^Tc_1, w_2\mathcal{L}^Tc_2 \cdots, w_r\mathcal{L}^Tc_r\right],
\end{align*}
Where $v_i=(\alpha_i^2 \mathcal{M}+\alpha_i\mathcal{D}+\mathcal{K})^{-1}$ and $w_i=(\alpha_i^2 \mathcal{M}^T+\alpha_i\mathcal{D}^T+\mathcal{K}^T)^{-1}$; $i= 1,2,\cdots r$.

\While{(not converged)}{%
$
 \hat{\mathcal{M}} = \mathcal{W}^T \mathcal{M}\mathcal{V},\, 
 \hat{\mathcal{D}} = \mathcal{W}^T \mathcal{D} \mathcal{V},\, 
 \hat{\mathcal{K}} = \mathcal{W}^T \mathcal{K} \mathcal{V},\,
 \hat{\mathcal{F}} = \mathcal{W}^T\mathcal{F},\, 
 \hat{\mathcal{L}} = \mathcal{L}\mathcal{V}. 
$\\
Use the first-order representation ($E,A,B,C$) as in (\ref{eq:back:2ndto1st}) find the reduced-order matrices $\hat{E}$, $\hat{A}$, $\hat{B}$ and $\hat{C}$.\\
Compute $\hat{A}z_i=\lambda_i\hat{E}z_i$ and $y^*_i \hat{A}=\lambda_i y^*_i \hat{E}$ for
$\alpha_i \leftarrow -\lambda_i$, $b^*_i \leftarrow -y^*_i \hat{B}$ and $c^*_i \leftarrow \hat{C}z^*_i$ for all $i=1,\cdots,r$\\
Repeat Step~2.\\
$i=i+1$;
}
Construct the reduced matrices by repeating Step~4
\end{algorithm}
%
\subsection{Reformulation of second-order index-3 descriptor system}
\label{sec:hindex:reformofdynamicalsystem}
We already have mentioned in earlier section that projecting the index-3 system (\ref{eq:INT:systemequation}) onto the hidden manifold we can convert the system into an index-0 i.e., second-order standard system like (\ref{eq:tisec:systemequation}). However, such conversion for a large scale dynamical system is practically impossible due to additional complexities. The  idea of conversion is already developed in \cite{morUdd20}. For our convenience, we briefly introduce this in the following. 

Let us consider the projector onto the null-space of $\G$,
\begin{align}\label{eq:indexred:piprojetor}
\Pi:= I-\G^T(\G\M^{-1}\G^T)^{-1}\G\M^{-1},
\end{align}
which satisfies
$\Pi \M=\phantom{1} \M\Pi^T$, 
$\mathrm{Null}(\Pi)= \mathrm{Rang}(\G^T)$, 
$\mathrm{Rang}{(\Pi)}=\mathrm{Null}{(\G\M^{-1})}$
and the most importantly
\begin{align}\label{eq:propertisofPi}
\G\x=0 \quad \iff \quad \Pi^T\x = \x.
\end{align}
Readers are referred to e.g., \cite{morUdd20} to see the details 
of these properties with proofs. Now  applying these identities into (\ref{eq:tisec:systemequation}) we obtain
\begin{subequations}\label{eq:background:piproj}
 \begin{align}
 \Pi \M \Pi^T\ddot{\x}(t) +\Pi \D \Pi^T \dot{\x}(t)+\Pi \K \Pi^T \x(t)& =\Pi \F u(t), 
                                          \label{eq:secondordindex3piprojected2}\\
 \y(t)& = \textbf{L} \Pi^T \x(t) \label{eq:index3piprojected3}.
 \end{align}
\end{subequations}
The dynamical system (\ref{eq:background:piproj}) still has 
unnecessary equations due to the singularity of $\Pi$. Those equations can be avoided 
by splitting $\Pi= \Psi_l \Psi_r$, where $\Psi_l,\ \Psi_r\in \mathbb{R}^{n_1\times (n_1-n_2)}$ and they satisfies   
\begin{align}\label{eq:indexred:projectionproperty}
  \Psi_l^T\Psi_r=I_{n_1-n_2},
\end{align}
where $I_{n_1-n_2}$ is an $(n_1-n_2)\times (n_1-n_2)$ identity matrix.
Inserting the decomposition of $\Pi$ into
(\ref{eq:background:piproj}) and considering 
$\tilde{\x}(t)={\Psi_l}^T \x(t)$, the resulting dynamical system leads to 
\begin{subequations}\label{eq:backg:thetaproj}
 \begin{align}
 \Psi_r^T \M \Psi_r\ddot{\tilde{\x}}(t) +
 \Psi_r^T \D\Psi_r \dot{\tilde{\x}}(t) + \Psi_r^T \K\Psi_r 
 \tilde{\x}(t) & =\Psi_r^T \F u(t), \label{eq:secondordindex3piprojected2}\\
 \y(t) & =\textbf{L} \Psi_r \tilde{\x}(t) \label{eq:index3piprojected3}.
 \end{align}
\end{subequations}
This system is now a standard 
second-order system  as described by (\ref{eq:tisec:systemequation}). In fact system in (\ref{eq:backg:thetaproj}) can be seen as the system (\ref{eq:background:piproj}) with the redundant equations being removed through the $\Psi_r$ projection. Note that the coefficient matrices of (\ref{eq:backg:thetaproj}) are dense if compared to (\ref{eq:INT:systemequation}). Therefore, for the large-scale index-3 system explicit computation of (\ref{eq:backg:thetaproj}) is forbidden. In fact the dynamical systems (\ref{eq:INT:systemequation}), (\ref{eq:background:piproj}) and (\ref{eq:backg:thetaproj}) are equivalent in the sense that they are the different realizations of the same transfer function. Moreover, their finite spectra are the same which has been  proven in the sequel. Once the index-3 system (\ref{eq:INT:systemequation}) is converted into the index-0 system (\ref{eq:backg:thetaproj}) then the SPIRKA i.e., Algorithm~\ref{alg:irka3}  can be  applied  to the converted system. However, as the converted system is dense, the computational costs of Algorithm~\ref{alg:irka3} becomes high. For a very high dimensional index-3 descriptor system computing (\ref{eq:backg:thetaproj}) is not possible due to the restriction of computer memory. Therefore we are motivated to construct reduced order model without forming system (\ref{eq:backg:thetaproj}) explicitly. 

\section{IRKA for second-order index-3 descriptor systems}
\label{sec:hindex:reformofdynamicalsystem}
The SPIRKA introduced in Section~\ref{sec:btforsecondorder} can be applied
to the projected system~(\ref{eq:backg:thetaproj}).
As already mentioned, this is   infeasible for a large-scale system. Therefore, the  technique can be applied to the equivalent system (\ref{eq:background:piproj}) instead.
For this purpose, following the discussion in Section~\ref{sec:btforsecondorder}, 
we can create the right and left projectors as 
\begin{equation}\label{eq:irksecordindex3:left-rightproject}
\begin{aligned}
V & = \left[v_1, v_2, \cdots,v_r \right],\\
W & = \left[w_1, w_2 \cdots, w_r \right],
\end{aligned}
\end{equation}
where
\begin{align*}
v_i & =(\alpha_i^2 \tilde{\M}+\alpha_i\tilde{\D}+\tilde{\K})^{-1}\tilde{\F}b_i, \\
w_i & =(\alpha_1^2 \tilde{\M}^T+\alpha_1\tilde{\D}^T+\tilde{\K}^T)^{-1}\tilde{\textbf{L}}^Tc_i,
\end{align*}
for  $i= 1,2,\cdots r$ and in which 
%
$\tilde{\M} = \Pi \M \Pi^T,\, \tilde{\D} = \Pi \D \Pi^T, \tilde{\K} = \Pi \K \Pi^T, \tilde{\F} = \Pi \F \Pi^T$ and $\tilde{\textbf{L}} = \textbf{L} \Pi^T$. 
The main expensive task here is to compute each vector inside the projectors by solving a linear system. For example to construct $V$, at $i$-th iteration we find $(\alpha_i^2 \tilde{\M}+\alpha_i\tilde{\textbf{L}}+\tilde{\K})^{-1}\tilde{\F}b_i$  to solve the linear system
\begin{align*}
    (\alpha_i^2 \tilde{\M}+\alpha_i\tilde{\D}+\tilde{\K})v_i = \tilde{\F}b_i,
\end{align*}
which implies 
\begin{align}\label{eq:index3:projectedlinersys}
    \Pi(\alpha_i^2 \M+\alpha_i \D+ \K)\Pi^T v_i= \F b_i.
\end{align}
This linear system can be solved efficiently by applying the following Lemma~\label{lem:index3:linearsys}.
\begin{lemma}\label{lem:index3:linearsys}
The matrix $\varkappa$ satisfies $\varkappa=\Pi^T\varkappa$ and 
$\Pi(\alpha^2 \M+\alpha \D+ \K)\Pi^T\varkappa = \Pi \F b$ if and only if
  \begin{align}\label{eq:BC:linearsysinitialguess}
    \begin{bmatrix}
      \alpha^2 \M+\alpha \D+ \K & \G^T\\ \G & 0
    \end{bmatrix}
    \begin{bmatrix}
      \varkappa\\ \Lambda
    \end{bmatrix}=
    \begin{bmatrix}
      \F b \\ 0
    \end{bmatrix}.
  \end{align}  
  \end{lemma}
  \begin{proof}
  If $\varkappa=\Pi^T\varkappa$, then by using (\ref{eq:propertisofPi}) we have 
  \begin{align}\label{eq:index3:g0}
  G\varkappa= 0
  \end{align}
  which is the second block of equation (\ref{eq:BC:linearsysinitialguess}). Furthermore $\Pi(\alpha^2 \M+\alpha \D+ \K)\Pi^T\varkappa = \Pi \F b$ implies 
  \begin{align*}
  \Pi(\alpha^2 \M+\alpha \D+ \K)\varkappa & = \Pi \F b, \\
  \Pi((\alpha^2 \M+\alpha \D+ \K)\varkappa - \F b) & = 0. 
  \end{align*}
  This means $(\alpha^2 \M+\alpha \D+ \K)\varkappa - \F b$ is in the null space of $\Pi$. We know that  
  $\mathrm{Null}{(\Pi)}=\mathrm{Rang}{(G^T)}$. Therefore, there exists $\Lambda$ such that $(\alpha^2 \M+\alpha \D+ \K)\varkappa - \F b = -G^T\Lambda$ which implies
  \begin{align}\label{eqq:index3:g1}
  (\alpha^2 \M+\alpha \D+ \K)\varkappa +G^T\Lambda = \F b. 
  \end{align}
 Equations (\ref{eq:index3:g0}) and (\ref{eqq:index3:g1}) yield (\ref{eq:BC:linearsysinitialguess}). Conversely, we
  assume (\ref{eq:BC:linearsysinitialguess}) holds. From the second line of (\ref{eq:BC:linearsysinitialguess}) we obtain $G\varkappa = 0$ and thus $\varkappa=\Pi^T\varkappa$. Now from first equation we obtain
  \[
  (\alpha^2 \M+\alpha \D + \K)\varkappa +G^T\Lambda = \F b. 
  \]
  Applying (\ref{eq:propertisofPi}) this equation gives
  \[
   (\alpha^2 \M+\alpha \D+ \K)\Pi^T\varkappa +G^T\Lambda = \F b. 
  \]
 Multiplying both sides by $\Pi$ and since $\Pi G^T=0$ we have 
  \[
   \Pi(\alpha^2 \M+\alpha \D+ \K)\Pi^T\varkappa  = \Pi\F b. 
  \]
  This completes the proof.
\end{proof}
 Following Lemma~\ref{lem:index3:linearsys} instead of solving  (\ref{eq:index3:projectedlinersys}) we can solve  
  \begin{align}\label{eq:BC:linearsysinitialguess}
    \begin{bmatrix}
      \alpha_i^2 \M+\alpha_i \D+ \K & \G^T \\ \G & 0
    \end{bmatrix}
    \begin{bmatrix}
      v_i\\ \Lambda
    \end{bmatrix}=
    \begin{bmatrix}
      \F b_i \\ 0
    \end{bmatrix},
  \end{align}  
  for $\xi$. 
Although this system is larger than its projected system. Therefore we solve this linear system to avoid constructing the projector. Similarly to construct $W$ at each iteration we compute $w_i= (\alpha_i^2\tilde{\M}^T+\alpha_i\tilde{\D}^T+\tilde{\K}^{T})^{-1}\tilde{\textbf{L}}^Tc_i$ by solving the following linear system
\[
  \begin{bmatrix}
  \alpha_i^2\M^T+\alpha_i\D^T+\K^T & \G^T\\ \G & 0
 \end{bmatrix}
\begin{bmatrix}
 w_i\\ \Gamma
\end{bmatrix}=
\begin{bmatrix}
 \textbf{L}^Tc_i \\0
\end{bmatrix}.
\]
Once we have $V$ and $W$, apply them to (\ref{eq:background:piproj}) to find the reduce order model
\begin{subequations}
\label{eq:index3:rom}
 \begin{align}
 \hat{\M} \ddot{\hat{\x}}(t) + \hat{\D} \dot{\hat{\x}}(t)+ \hat{\K}  \hat{\x}(t)& = \hat{\F} u(t), 
                                          \label{eq:secondordindex3piprojected2}\\
 \hat{\y}(t)& = \hat{\textbf{L}}  \hat{\x}(t) \label{eq:index3piprojected3},
 \end{align}
\end{subequations}
in which the reduced matrices are constructed as follows
\begin{align*}
\hat{\M} & = W^T\Pi \M \Pi^T V,\quad \hat{\D} = W^T\Pi \D \Pi^T V,\\
\hat{\K} & = W^T\Pi \K \Pi^T V,\quad \hat{\F} = W^T\Pi \F,\quad \hat{\textbf{L}} =  \textbf{L} \Pi^T V.
\end{align*}
Due to the properties of the projector, as mentioned in (\ref{eq:propertisofPi}) we have
$\Pi^T V= V$ and $\Pi^T W= W$ or $(\Pi^TW)^T = W^T$. Therefore the reduced matrices can be constructed without using $\Pi$ as follows
\begin{align*}
\hat{\M}  = W^T \M  V,\quad \hat{\D} = W^T \D  V,\quad
\hat{\K}  = W^T \K  V,\\ \hat{\F} = W^T \F,\quad \hat{\textbf{L}} =  \textbf{L}  V.
\end{align*}
The whole procedure to construct reduced order model from second-order index-3 system (\ref{eq:INT:systemequation}) is summarize in Algorithm~\ref{alg:daes:irka2}.
\begin{algorithm}[t]
	\SetAlgoLined
	\SetKwInOut{Input}{Input}
	\SetKwInOut{Output}{Output}
	\caption{IRKA for Second-Order Index-3 Descriptor Systems.}
	\label{alg:daes:irka2}
	\Input {$\M, \D, \K, \F, \textbf{L}$.}
	\Output  {$\hat{\M}, \hat{\D}, \hat{\K}, \hat{\F}, \hat{\textbf{L}}$}
	Select randomly a set  of the interpolation points $\{\alpha_i\}_{i=1}^r$ and the tangential directions $\{b_i\}_{i=1}^r$ and $\{c_i\}_{i=1}^r$.\\
	Construct the  projection matrices
	$$V_s = \left[v_1, v_2, \cdots, v_r\right] \quad \text{and}
	\quad W_s = \left[w_1, w_2, \cdots, w_r\right],$$ 
	where $v_i$ \& $w_i$; $i=1,\cdots,r$ are the solutions of the linear systems
	\[
    \begin{bmatrix}
      \alpha_i^2 \M+\alpha_i \D+ \K & \G^T \\ \G & 0
    \end{bmatrix}
    \begin{bmatrix}
      v_i\\ \Lambda
    \end{bmatrix}=
    \begin{bmatrix}
      \F b_i \\ 0
    \end{bmatrix}
  \] 
 and  
\[
  \begin{bmatrix}
  \alpha_i^2\M^T+\alpha_i\D^T+\K^T & \G^T\\ \G & 0
 \end{bmatrix}
\begin{bmatrix}
 w_i \\ \Gamma
\end{bmatrix}=
\begin{bmatrix}
 \textbf{L}^Tc_i \\0
\end{bmatrix},
\] 
respectively.\\
\While{(not converged)}{%
Construct
\begin{align*}
\hat{\M}  = W^T \M  V,\quad \hat{\D} = W^T \D  V,\quad
\hat{\K}  = W^T \K  V,\quad
\hat{\F} = W^T \F,\quad \hat{\textbf{L}} =  \textbf{L}  V.
\end{align*}
	Compute $(E, A, B, C)$ as in (\ref{eq:daes:1storderreducedmodel}), then using 
	MATLAB function \texttt{balred} compute $r$ dimensions model $(\hat{E}, \hat{A}, \hat{B}, \hat{C})$. \\ 
	Compute $\hat{A}z_i = \lambda_i\hat{E}z_i$ and $y^*_i \hat{A} = \lambda_i y^*_i \hat{E}$ for $\alpha_i \leftarrow -\lambda_i$, $b^*_i \leftarrow -y^*_i \hat{B}$ and $c^*_i \leftarrow \hat{C}z^*_i$. \\
	Repeat Step~2. \\	
	$i=i+1$. \\}
Finally construct the reduced-order matrices
\begin{align*}
\hat{\M}  = W^T \M  V,\quad \hat{\D} = W^T \D  V,\quad
\hat{\K}  = W^T \K  V,\quad \hat{\F} = W^T \F,\quad \hat{\textbf{L}} =  \textbf{L}  V.
\end{align*}
\end{algorithm}
\paragraph{Update interpolation points and tangential directions.} In IRKA we need to update the interpolation points and tangential direction at each iteration steps which is often a challenging task. Usually, the interpolation points and tangential direction are updated  by using  mirror image of the eigenvalues and eigenvector of the reduced order model.  In SPIRKA this task is complicated because we need to solve the quadratic eigenvalue problems. If we solve a quadratic eigenvalue problem \cite{TisM01}  using ($\hat{\M}, \hat{\D}, \hat{\K}$) we obtain $2r$ number of eigenvalues and eigenvectors. Selecting $r$ number of optimal interpolation points and corresponding tangential direction is challenging task. To resolve this complexity \cite{morWya12} propose a techniques which is discussed in Algorithm~\ref{alg:daes:irka2}. To follow this idea, we need to apply standard IRKA onto the converted first order system from the reduced second order system. However, this is again an iterative method which is computationally expensive. In this paper to construct the interpolation points and tangential directions we construct 
\begin{equation}\label{eq:daes:1storderreducedmodel}
\begin{aligned}
E:= \begin{bmatrix}
   0 & I\\
   0 & \hat{\M}
 \end{bmatrix}, \quad 
A:= \begin{bmatrix}
      0 &  I\\
      -\hat{\K} &  -\hat{\D}
    \end{bmatrix}, \\ 
B:= \begin{bmatrix}
        0\\ \hat{\F} 
    \end{bmatrix}\quad \text{and}\quad
C:= \begin{bmatrix}
        \hat{\textbf{L}} & 0 
    \end{bmatrix}.
\end{aligned}
\end{equation}
Then we apply MATLAB function \texttt{balred} to compute $r$ dimensional reduced-order model from ($\hat{E}, \hat{A}, \hat{B}, \hat{C}$). The interpolation points are updated by choosing  the mirror images of the eigenvalues of the pair $(\hat{A},\hat{E})$ as the next interpolations points. This seems to be more efficient than the existing one.
\section{Numerical results}
 \label{sec:numerical-results}
\newlength\figwidth
\setlength{\figwidth}{.35\linewidth}
\newlength\figheight
\setlength{\figheight}{.5\linewidth}
\pgfplotsset{every axis plot/.append style={line width=2pt}}
\tikzset{mark options={solid,mark size=2,line width=.5pt,mark repeat=20}}
\tikzstyle{decision} = [diamond, draw, fill=blue!20, 
    text width=4.5em, text badly centered, node distance=3cm, inner sep=0pt]
To asses the efficiency of the proposed algorithm, i.,e., Algorithm~\ref{alg:daes:irka2} we have applied this to two sets of data. First one is coming from a damped spring-mass system (DSMS) with holonomic constraint which is taken from \cite{morMehS05}.
The second set of data set is a constrain triple chain oscillator model (TCOM). This data is originated in \cite{TruV07} but with the index-3 setup  described in \cite{morUdd18}. The details of the models is also available in \cite{morUdd20}. We intentionally have considered the same data for the both model examples as in  \cite{morUdd20} since we want to compare the results of this paper with that one.  The dimension of the models including the number of differential and algebraic equations and inputs/outputs are displayed in Table~\ref{tab:hindex:noofshift}.
\begin{table}[tb]
  \begin{center}
    \begin{tabular*}{\hsize}{@{\extracolsep{\fill}}lcrcr}
      \hline
      models   & dimension  & $n_1$ and $n_2$ & inputs/outputs\\
      \hline
     DSMS    & 2200   & 2000 and  200  & 1/3\\
      
      TCOM   & 11001   & 6001 and 5000  & 1/1 \\
    \hline
    \end{tabular*}
  \end{center}
  \caption{The dimension of the tested models including number of differential and algebraic variables and, inputs and outputs.} 
  \label{tab:hindex:noofshift}
\end{table}

This experiment was carried out with MATLAB\textsuperscript{\textregistered} R2015a (8.5.0.197613)  on a board with 4$\times$INTEL{}$\text{Core}^{\text{TM}}$i5-4460s 
CPU with a 2.90~GHz clock speed and 16~GB RAM.

We apply Algorithm~\ref{alg:daes:irka2} to the both models and find 30 dimensional reduced-order models. Figures~\ref{fig:numresult:smds} and \ref{fig:numresult:tcom} show the frequency domain analysis of the original and reduced models for the DSMS and TCOM, respectively. The frequency responses of the full and the reduced-order models and their absolute and relative errors for the DSMS  are shown  in Figure~\ref{fig:numresult:smds} over the frequency interval $[10^{-2},10^0]$. Sub-figure~\ref{fig:numresult:smds1} shows that the frequency responses of the reduced-order models are matching with the original model with good accuracy. The absolute and relative errors between full and the reduced-order models are shown in Sub-figures \ref{fig:numresult:smds2} and \ref{fig:numresult:smds3}, respectively. 
\begin{figure}[t]
 \setlength{\figheight}{.2\linewidth}
  \begin{subfigure}{\linewidth}
    \centering
    \setlength{\figwidth}{.75\linewidth}
      \input{figures/tf_smd_50}
    \caption{Frequency response.}
    \label{fig:numresult:smds1}
  \end{subfigure}
   \begin{subfigure}{\linewidth}
    \centering
    \setlength{\figwidth}{.75\linewidth}
      \input{figures/abserr_smd_50}
    \caption{Absolute error.}
    \label{fig:numresult:smds2}
  \end{subfigure}
   \begin{subfigure}{\linewidth}
    \centering
    \setlength{\figwidth}{.75\linewidth}
      \input{figures/relatverr_smd_50}
    \caption{Relative error.}
    \label{fig:numresult:smds3}
  \end{subfigure}
  \caption{Comparison of original and the 30  dimensional reduced models for the DSMS.}
  \label{fig:numresult:smds}
\end{figure}
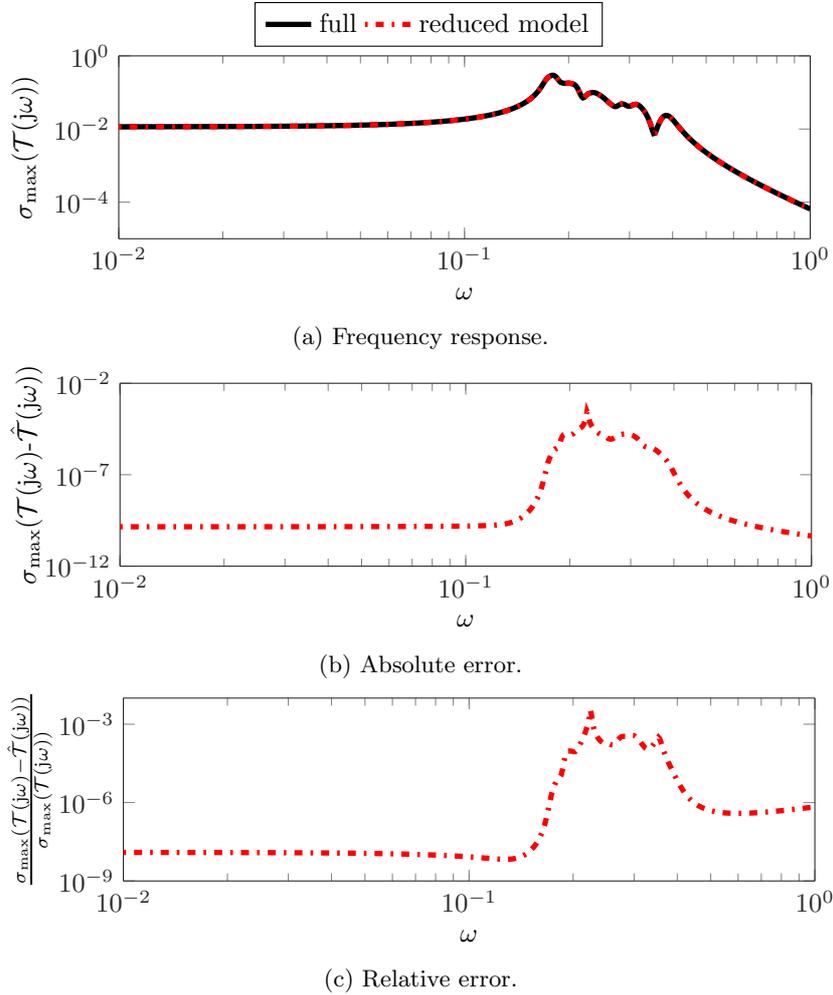
On the other hand, Figure~\ref{fig:numresult:tcom} depicts the frequency responses, absolute and relative errors of the full and the reduced-order models of the TCOM  over the frequency interval $[10^{-3},10^0]$. This figure also shows (in Sub-figure~\ref{fig:numresult:tcom1}) that the frequency responses of the reduced-order models are matching correctly with the original model. Sub-figures~\ref{fig:numresult:tcom2} and \ref{fig:numresult:tcom3} shows a good approximation between the original and reduced-order models using the absolute and relative errors, respectively. 
\begin{figure}[t]
 \setlength{\figheight}{.2\linewidth}
  \begin{subfigure}{\linewidth}
    \centering
    \setlength{\figwidth}{.75\linewidth}
    \input{figures/tf_tchain_30}
    \caption{Frequency response.}
    \label{fig:numresult:tcom1}
  \end{subfigure}
  \begin{subfigure}{\linewidth}
    \centering
    \setlength{\figwidth}{.75\linewidth}
   \input{figures/abserr_tchain_30}
    \caption{Absolute error.}
    \label{fig:numresult:tcom2}
  \end{subfigure}
  \begin{subfigure}{\linewidth}
    \centering
    \setlength{\figwidth}{.75\linewidth}
    \input{figures/relatv_err_tchain_30}
    \caption{Relative error.}
    \label{fig:numresult:tcom3}
  \end{subfigure}
  \caption{Comparison of the original and 30 dimensional reduced models for the TCOM.}
  \label{fig:numresult:tcom}
\end{figure}
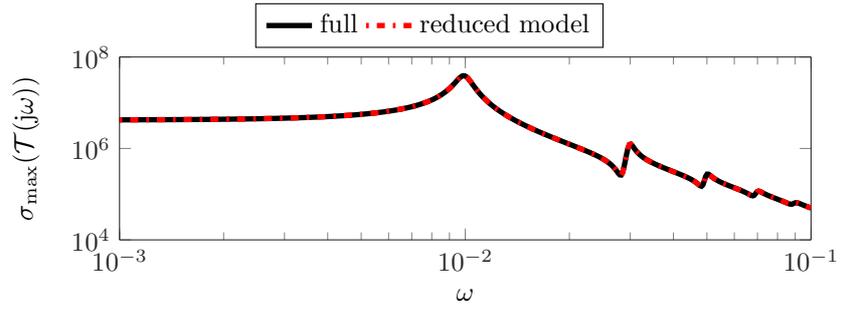
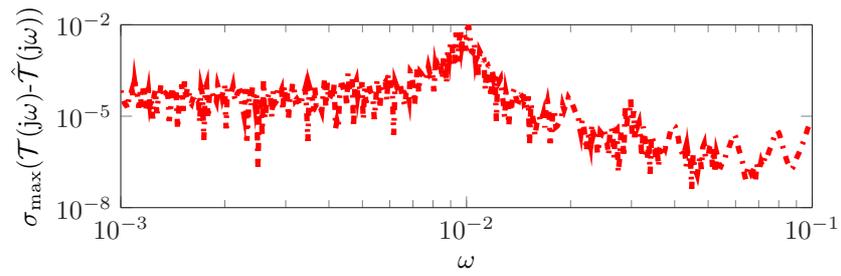
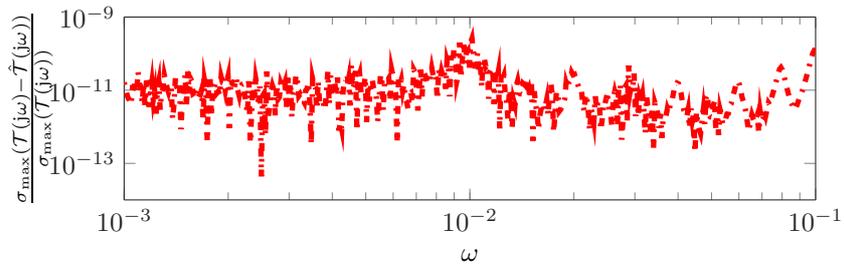
\paragraph{Comparisons of Balanced truncation and IRKA.} To compare the performance of IRKA and balanced truncation we compute $30$ dimensional reduced-order model applying \cite[Algorithm 2]{morUdd20} to the TCOM. This algorithm can compute several reduced-order models based on different balancing criterion. Here we consider the velocity-velocity balancing label which gives the best approximation. Figure \ref{fig:numresult:combtirka} shows the approximation errors  of 30 dimensional reduced-order models computed by the BT and IRKA. From Figure~\ref{fig:numresult:combtirka} it seems that the performance of IRKA is better than the BT.  Both the absolute error and relative errors as shown in Figures  \ref{fig:numresult:combtirka1} and \ref{fig:numresult:combtirka2}, respectively, IRKA depicts better accuracy than the balanced truncation. 
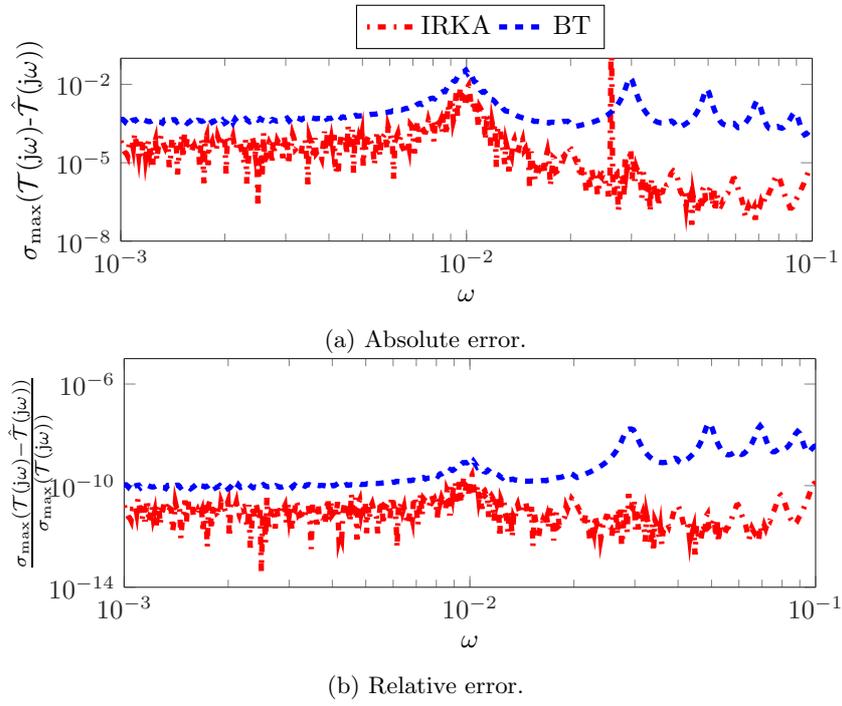
\begin{figure}[t]
\setlength{\figheight}{.2\linewidth}
  \begin{subfigure}{\linewidth}
    \centering
    \setlength{\figwidth}{.75\linewidth}
   \input{figures/abser_bt_irka30}
    \caption{Absolute error.}
    \label{fig:numresult:combtirka1}
  \end{subfigure}
  \setlength{\figheight}{.25\linewidth}
  \begin{subfigure}{\linewidth}
    \centering
    \setlength{\figwidth}{.75\linewidth}
    \input{figures/relatverr_bt_irka30}
    \caption{Relative error.}
    \label{fig:numresult:combtirka2}
  \end{subfigure}
  \caption{Comparison of the original and 30 dimensional reduced models computed by IRKA and balanced truncation for the TCOM.}
  \label{fig:numresult:combtirka}
\end{figure}
On the other hand, when we consider the computation time, again the performance of IRKA is far better than the BT which is reflected in Figure~\ref{fig:numresult:tcom}. We know that balanced truncation is expensive method since it requires to solve two continuous-time algebraic  Lyapunov equations. The solution of the Lyapunov equations involved the computation of shift parameters which is a computational. We have solved the Lyapunov equations  by \cite[Algorithm 3]{morUdd20} using adaptive shift parameters. See, e.g., \cite{morUdd20} for details. Note that the computational time of IRKA is increasing if the dimension of reduced order model and the number of iterations are increased gradually. 
\begin{figure}[t]
  \includegraphics[width=12cm,height= 8cm]{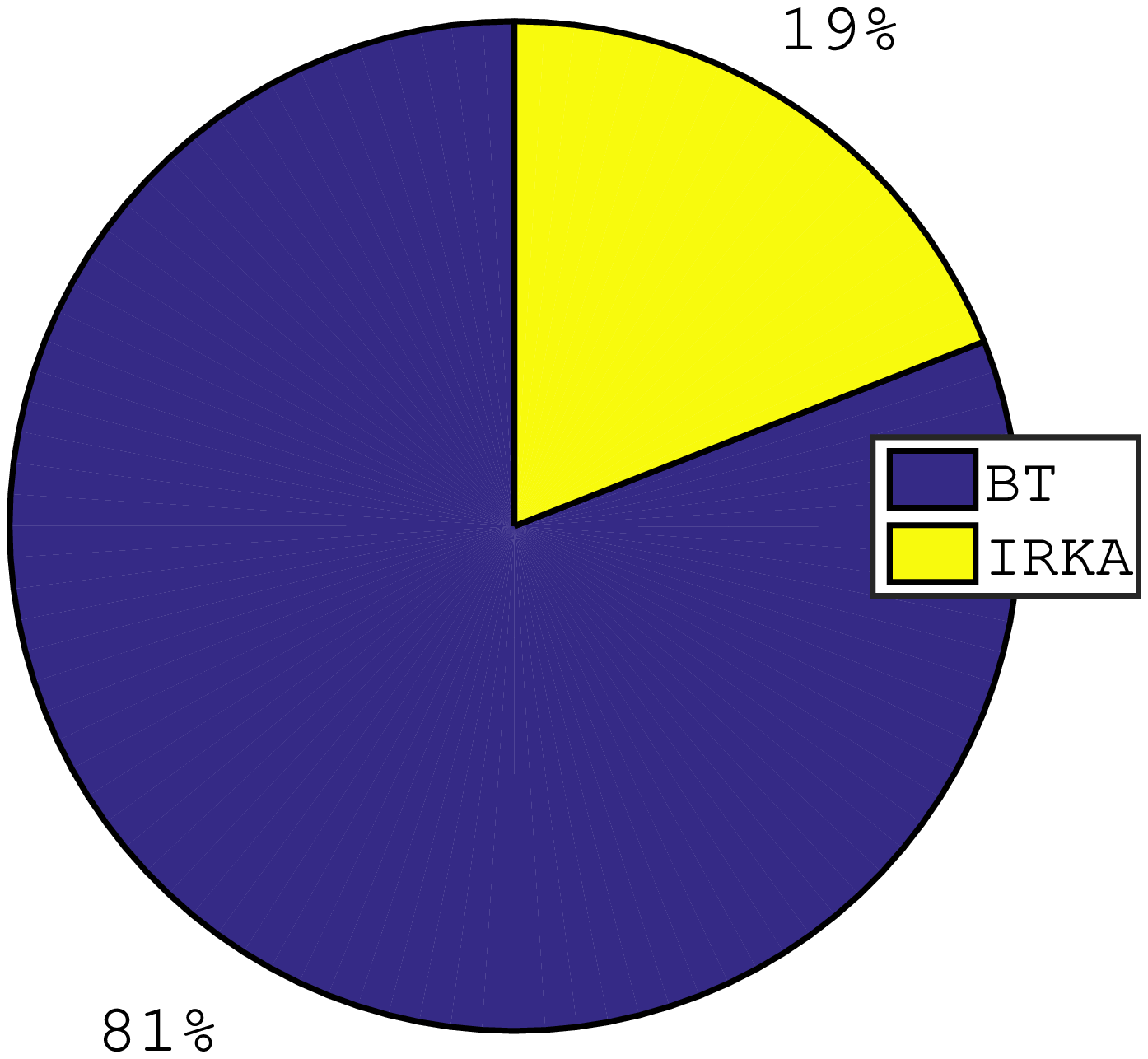}
  \caption{Time comparisons of both balanced truncation and IRKA for the TCOM.}
  \label{fig:numresult:tcom}
\end{figure}
\section{Conclusions}\label{sec:conclusion}
In this paper we have discussed a IRKA based technique to find a 
reduced second-order system from a large-scale sparse second-order index-3 system. In particular,  we have linearized equation of motion with holonomic constrains which arise  in constrained mechanics or multibody  dynamics. It has been shown that the index-3 system can be converted into index-0 by projecting onto the hidden manifold to apply the standard second-order IRKA. But creating projector is often computationally expensive task and it yields system matrices dense. Therefore we have modified the standard IRKA for the underlying index-3 descriptor system. We also have shown a clever techniques to compute the interpolation points and tangential directions. The proposed algorithm was applied to several data of second-order index-3 models.  Numerical results showed that the proposed algorithm can generated lower dimensional model with higher accuracy. The IRKA based method is better than Balanced truncation in terms of accuracy and computational complexity as well.
\section{Acknowledgment}
This research work was funded by NSU-CTRG research grant under the project No.: CTRG-19/SEPS/05. It was also supported by National Natural Science Foundation of China under Grant No. (61873336, 61873335), the Fundamental Research Funds for the Central Universities under Grant (FRF-BD-19-002A), and the High-end foreign expert program of Shanghai University, 
 
\clearpage
\label{sec:numeric}
\bibliographystyle{IEEEtran}      
\bibliography{morbook}
%

\end{document}

%% file: figures/tf_smd_50.tex
%
%
%
%
\begin{tikzpicture}

\begin{axis}[%
width=\figwidth,
height=\figheight,
scale only axis,
separate axis lines,
every outer x axis line/.append style={darkgray!60!black},
every x tick label/.append style={font=\color{darkgray!60!black}},
xmode=log,
xmin=0.01,
xmax=1,
xminorticks=true,
xlabel={$\omega$},
every outer y axis line/.append style={darkgray!60!black},
every y tick label/.append style={font=\color{darkgray!60!black}},
ymode=log,
ymin=1e-05,
ymax=1,
yminorticks=true,
ylabel={$\sigma{}_{\text{max}}\text{($\mathcal{T}$(j}\omega\text{))}$},
legend columns=2,
legend style={nodes=right,anchor=south east,at={(.7,1.05)}}
]
\addplot [
color=black,
solid,
]
table[row sep=crcr]{
0.01 0.0115424290734638\\
0.0100927151463057 0.0115433080840734\\
0.0101862899024469 0.0115442035880056\\
0.0102807322383087 0.0115451158969844\\
0.0103760501976691 0.0115460453287107\\
0.0104722518988843 0.0115469922069792\\
0.0105693455355799 0.0115479568617998\\
0.0106673393773486 0.0115489396295192\\
0.0107662417704549 0.0115499408529469\\
0.010866061138546 0.0115509608814829\\
0.0109668059833687 0.0115520000712482\\
0.0110684848854941 0.0115530587852179\\
0.0111711065050482 0.0115541373933578\\
0.0112746795824495 0.0115552362727629\\
0.0113792129391532 0.0115563558077993\\
0.0114847154784029 0.0115574963902493\\
0.0115911961859889 0.0115586584194591\\
0.0116986641310131 0.0115598423024898\\
0.0118071284666619 0.011561048454272\\
0.0119165984309856 0.0115622772977628\\
0.0120270833476851 0.0115635292641071\\
0.0121385926269063 0.0115648047928018\\
0.0122511357660412 0.0115661043318634\\
0.0123647223505372 0.0115674283379999\\
0.0124793620547131 0.0115687772767855\\
0.0125950646425836 0.0115701516228395\\
0.0127118399686903 0.0115715518600094\\
0.0128296979789415 0.0115729784815572\\
0.012948648711459 0.0115744319903503\\
0.0130687022974335 0.0115759128990565\\
0.0131898689619867 0.0115774217303431\\
0.0133121590250431 0.0115789590170802\\
0.0134355829022083 0.0115805253025489\\
0.0135601511056563 0.0115821211406536\\
0.0136858742450252 0.0115837470961396\\
0.0138127630283201 0.0115854037448145\\
0.0139408282628258 0.0115870916737755\\
0.0140700808560269 0.0115888114816414\\
0.0142005318165368 0.0115905637787894\\
0.0143321922550357 0.0115923491875978\\
0.0144650733852165 0.0115941683426936\\
0.0145991865247398 0.0115960218912058\\
0.0147345430961984 0.0115979104930251\\
0.0148711546280895 0.0115998348210677\\
0.0150090327557974 0.0116017955615474\\
0.0151481892225835 0.011603793414252\\
0.0152886358805873 0.0116058290928266\\
0.0154303846918356 0.0116079033250641\\
0.0155734477292613 0.0116100168532008\\
0.0157178371777316 0.0116121704342206\\
0.0158635653350859 0.0116143648401644\\
0.0160106446131832 0.0116166008584485\\
0.0161590875389592 0.0116188792921886\\
0.0163089067554933 0.0116212009605326\\
0.0164601150230855 0.0116235666990007\\
0.0166127252203429 0.0116259773598336\\
0.016766750345277 0.0116284338123485\\
0.0169222035164104 0.0116309369433041\\
0.0170790979738943 0.0116334876572738\\
0.0172374470806362 0.0116360868770279\\
0.017397264323438 0.0116387355439247\\
0.0175585633141447 0.0116414346183116\\
0.0177213577908036 0.011644185079935\\
0.0178856616188346 0.0116469879283605\\
0.0180514887922111 0.0116498441834033\\
0.0182188534346517 0.0116527548855688\\
0.0183877698008233 0.0116557210965041\\
0.0185582522775552 0.0116587438994607\\
0.0187303153850644 0.0116618243997681\\
0.0189039737781922 0.0116649637253193\\
0.0190792422476527 0.0116681630270685\\
0.019256135721292 0.0116714234795408\\
0.0194346692653602 0.0116747462813549\\
0.0196148580857943 0.0116781326557586\\
0.0197967175295133 0.011681583851178\\
0.0199802630857255 0.0116851011417802\\
0.0201655103872475 0.0116886858280508\\
0.0203524752118358 0.0116923392373856\\
0.0205411734835306 0.0116960627246975\\
0.0207316212740123 0.0116998576730391\\
0.0209238348039698 0.0117037254942412\\
0.0211178304444824 0.0117076676295679\\
0.0213136247184144 0.0117116855503887\\
0.0215112343018217 0.0117157807588679\\
0.0217106760253726 0.0117199547886725\\
0.0219119668757815 0.0117242092056983\\
0.0221151239972549 0.0117285456088152\\
0.0223201646929523 0.0117329656306329\\
0.0225271064264598 0.0117374709382859\\
0.0227359668232772 0.011742063234241\\
0.0229467636723194 0.0117467442571249\\
0.0231595149274315 0.0117515157825752\\
0.0233742387089181 0.0117563796241147\\
0.0235909533050863 0.0117613376340484\\
0.0238096771738036 0.0117663917043865\\
0.0240304289440697 0.0117715437677917\\
0.0242532274176035 0.0117767957985534\\
0.0244780915704444 0.0117821498135886\\
0.0247050405545683 0.0117876078734709\\
0.0249340936995188 0.0117931720834887\\
0.0251652705140539 0.0117988445947325\\
0.0253985906878073 0.0118046276052139\\
0.0256340740929651 0.0118105233610162\\
0.0258717407859592 0.0118165341574772\\
0.0261116110091746 0.0118226623404077\\
0.0263537051926739 0.0118289103073435\\
0.0265980439559376 0.0118352805088348\\
0.0268446481096196 0.0118417754497735\\
0.0270935386573205 0.0118483976907586\\
0.0273447367973758 0.0118551498495028\\
0.0275982639246618 0.0118620346022804\\
0.0278541416324177 0.0118690546854184\\
0.0281123917140846 0.0118762128968332\\
0.0283730361651621 0.0118835120976119\\
0.0286360971850812 0.011890955213644\\
0.0289015971790951 0.0118985452373004\\
0.029169558760188 0.0119062852291658\\
0.0294400047510004 0.0119141783198235\\
0.0297129581857733 0.011922227711696\\
0.0299884423123103 0.0119304366809427\\
0.0302664805939569 0.0119388085794176\\
0.0305470967115997 0.0119473468366879\\
0.0308303145656828 0.0119560549621176\\
0.0311161582782436 0.0119649365470169\\
0.0314046521949675 0.0119739952668608\\
0.0316958208872612 0.0119832348835793\\
0.0319896891543454 0.0119926592479216\\
0.0322862820253673 0.0120022723018979\\
0.0325856247615323 0.0120120780813013\\
0.0328877428582551 0.0120220807183125\\
0.0331926620473319 0.0120322844441918\\
0.0335004082991313 0.0120426935920603\\
0.0338110078248069 0.0120533125997751\\
0.0341244870785289 0.0120641460129016\\
0.0344408727597382 0.0120751984877865\\
0.0347601918154198 0.0120864747947369\\
0.0350824714423979 0.0120979798213079\\
0.0354077390896527 0.0121097185757049\\
0.0357360224606579 0.0121216961903035\\
0.0360673495157403 0.0121339179252944\\
0.0364017484744614 0.0121463891724551\\
0.0367392478180207 0.0121591154590572\\
0.0370798762916817 0.0121721024519123\\
0.0374236629072198 0.0121853559615635\\
0.0377706369453937 0.0121988819466294\\
0.0381208279584389 0.0122126865183053\\
0.0384742657725851 0.0122267759450298\\
0.0388309804905961 0.012241156657324\\
0.0391910024943341 0.0122558352528094\\
0.039554362447347 0.0122708185014144\\
0.0399210912974805 0.0122861133507759\\
0.0402912202795135 0.0123017269318459\\
0.0406647809178186 0.0123176665647111\\
0.0410418050290471 0.0123339397646371\\
0.041422324724839 0.0123505542483444\\
0.0418063724145576 0.0123675179405297\\
0.0421939808080503 0.0123848389806422\\
0.0425851829184341 0.0124025257299263\\
0.042980012064908 0.012420586778745\\
0.0433785018755899 0.0124390309541944\\
0.0437806862903817 0.0124578673280262\\
0.0441865995638594 0.0124771052248894\\
0.0445962762681909 0.0124967542309098\\
0.0450097512960805 0.0125168242026205\\
0.0454270598637405 0.012537325276263\\
0.045848237513891 0.012558267877475\\
0.0462733201187869 0.012579662731385\\
0.0467023438832734 0.0126015208731334\\
0.0471353453478691 0.0126238536588417\\
0.0475723613918789 0.0126466727770512\\
0.0480134292365345 0.0126699902606567\\
0.048458586448165 0.0126938184993586\\
0.0489078709413959 0.0127181702526609\\
0.0493613209823792 0.012743058663443\\
0.0498189751920516 0.0127684972721356\\
0.0502808725494248 0.0127945000315312\\
0.0507470523949047 0.0128210813222637\\
0.0512175544336424 0.0128482559689925\\
0.051692418738916 0.0128760392573278\\
0.0521716857555435 0.012904446951539\\
0.0526553963033276 0.0129334953130867\\
0.0531435915805324 0.0129632011200246\\
0.0536363131673924 0.0129935816873185\\
0.0541336030296538 0.0130246548881353\\
0.0546355035221488 0.0130564391761528\\
0.055142057392403 0.0130889536089518\\
0.0556533077842765 0.0131222178725496\\
0.0561692982416381 0.0131562523071409\\
0.0566900727120743 0.0131910779341172\\
0.0572156755506325 0.0132267164844369\\
0.0577461515235982 0.0132631904284276\\
0.0582815458123085 0.0133005230071051\\
0.0588219040169996 0.0133387382650981\\
0.0593672721606913 0.013377861085278\\
0.0599176966931062 0.0134179172251953\\
0.0604732244946265 0.0134589333554339\\
0.0610339028802862 0.0135009371000029\\
0.0615997796038017 0.0135439570788916\\
0.0621709028616383 0.013588022952924\\
0.0627473212971158 0.0136331654710598\\
0.0633290840045511 0.0136794165202969\\
0.0639162405334401 0.0137268091783454\\
0.0645088408926769 0.0137753777692517\\
0.0651069355548146 0.0138251579221681\\
0.0657105754603627 0.0138761866334766\\
0.0663198120221268 0.0139285023324914\\
0.0669346971295866 0.0139821449509818\\
0.0675552831533165 0.0140371559967773\\
0.0681816229494448 0.0140935786317354\\
0.0688137698641567 0.0141514577543763\\
0.069451777738237 0.0142108400875123\\
0.0700957009116562 0.0142717742712261\\
0.0707455942281988 0.0143343109615819\\
0.071401513040134 0.0143985029354845\\
0.0720635132129305 0.0144644052021348\\
0.0727316511300146 0.0145320751215704\\
0.0734059836975721 0.0146015725308203\\
0.0740865683493956 0.014672959878246\\
0.0747734630517759 0.0147463023666962\\
0.0754667263084389 0.0148216681061509\\
0.0761664171655289 0.0148991282765966\\
0.0768725952166373 0.0149787573019357\\
0.0775853206078784 0.0150606330358108\\
0.0783046540430119 0.0151448369603004\\
0.0790306567886135 0.0152314543985354\\
0.0797633906792928 0.015320574742381\\
0.0805029181229598 0.0154122916964401\\
0.0812493021061405 0.0155067035397514\\
0.0820026061993413 0.0156039134066909\\
0.0827628945624635 0.0157040295887347\\
0.0835302319502678 0.015807165858901\\
0.0843046837178897 0.0159134418208782\\
0.0850863158264058 0.0160229832850471\\
0.0858751948484517 0.0161359226738303\\
0.0866713879738923 0.0162523994590634\\
0.0874749630155442 0.016372560634359\\
0.0882859884149515 0.0164965612257598\\
0.0891045332482151 0.0166245648443297\\
0.0899306672318762 0.0167567442847369\\
0.0907644607288536 0.0168932821743331\\
0.0916059847544371 0.0170343716777411\\
0.0924553109823357 0.0171802172625403\\
0.0933125117507825 0.0173310355322856\\
0.0941776600686952 0.0174870561338323\\
0.0950508296218951 0.017648522746772\\
0.0959320947793824 0.017815694163733\\
0.0968215305996709 0.0179888454713732\\
0.09771921283718 0.0181682693431257\\
0.0986252179486878 0.0183542774561569\\
0.0995396230998423 0.0185472020466073\\
0.100462506171734 0.018747397619022\\
0.101393945767529 0.0189552428279979\\
0.102334021219164 0.0191711425525109\\
0.103282812594103 0.0193955301861989\\
0.104240400702156 0.019628870170128\\
0.105206867102362 0.0198716607983465\\
0.106182294109938 0.0201244373309075\\
0.107166764803286 0.0203877754541545\\
0.108160363031071 0.0206622951340264\\
0.109163173419361 0.020948664915124\\
0.110175281378839 0.0212476067264872\\
0.11119677311207 0.0215599012646897\\
0.112227735620851 0.0218863940362742\\
0.113268256713615 0.0222280021550693\\
0.114318425012915 0.022585722005998\\
0.115378329962966 0.0229606379061511\\
0.116448061837269 0.0233539319168216\\
0.117527711746295 0.0237668949877245\\
0.118617371645248 0.0242009396478002\\
0.119717134341897 0.0246576144971302\\
0.120827093504478 0.0251386208032402\\
0.121947343669674 0.0256458315645132\\
0.123077980250667 0.0261813134762618\\
0.124219099545262 0.0267473523246165\\
0.125370798744092 0.0273464824441736\\
0.126533175938894 0.0279815210129937\\
0.127706330130864 0.0286556081304475\\
0.128890361239089 0.0293722538392609\\
0.130085370109057 0.0301353935257018\\
0.131291458521247 0.0309494534781275\\
0.132508729199795 0.0318194288267654\\
0.13373728582125 0.0327509766571703\\
0.134977233023394 0.0337505278277379\\
0.136228676414165 0.0348254219847003\\
0.137491722580642 0.0359840715344293\\
0.138766479098131 0.0372361620116135\\
0.140053054539322 0.0385928985259785\\
0.14135155848354 0.0400673109967291\\
0.142662101526074 0.0416746350035341\\
0.143984795287601 0.0434327907449026\\
0.14531975242369 0.0453629904552273\\
0.146667086634397 0.0474905156575542\\
0.148026912673951 0.0498457212584174\\
0.149399346360526 0.0524653458785133\\
0.150784504586105 0.0553942402020607\\
0.152182505326439 0.0586876723992992\\
0.15359346765109 0.0624144391028313\\
0.155017511733577 0.0666611125172632\\
0.15645475886161 0.0715379031070852\\
0.157905331447418 0.0771868280131279\\
0.159369353038178 0.0837931501770962\\
0.160846948326536 0.0916013294578222\\
0.162338243161228 0.100936697870265\\
0.163843364557798 0.112232590845232\\
0.165362440709418 0.126055839923854\\
0.166895600997802 0.143099030597655\\
0.168442976004232 0.164032366618513\\
0.170004697520672 0.188935060902366\\
0.171580898561 0.215989236014902\\
0.173171713372335 0.240766738586459\\
0.174777277446468 0.260307630732998\\
0.176397727531404 0.276839795637122\\
0.178033201643011 0.290360715072757\\
0.179683839076772 0.296160869978743\\
0.18134978041965 0.291757451621504\\
0.183031167562061 0.277886156698596\\
0.184728143709963 0.254960553550297\\
0.186440853397049 0.226726810715541\\
0.188169442497056 0.202534530797475\\
0.189914058236193 0.188369215982712\\
0.191674849205682 0.182419799126293\\
0.193451965374405 0.179464800529891\\
0.195245558101686 0.17724654295628\\
0.19705578015018 0.177688401019943\\
0.198882785698881 0.179866531142724\\
0.200726730356257 0.180021740802732\\
0.202587771173502 0.17753566793577\\
0.204466066657912 0.173904805597558\\
0.206361776786386 0.169414661064711\\
0.208275063019052 0.162140794796466\\
0.210206088313016 0.149689091744367\\
0.212155017136245 0.131966466598269\\
0.214122015481573 0.111529606304452\\
0.216107250880838 0.0920309174950394\\
0.218110892419152 0.0774546665422446\\
0.220133110749303 0.0713304940030779\\
0.222174078106288 0.0736928688883623\\
0.224233968321985 0.0806058307387413\\
0.226312956839953 0.0878682545158204\\
0.228411220730382 0.0930746708259425\\
0.230528938705171 0.0961865385034137\\
0.232666291133146 0.098322231606094\\
0.234823460055428 0.0996665034282485\\
0.237000629200933 0.0995053446250427\\
0.239197984002024 0.0974622763443827\\
0.241415711610302 0.0939501195787817\\
0.243654000912547 0.08971496106209\\
0.245913042546805 0.0851758698380578\\
0.248193028918626 0.0804367773303254\\
0.25049415421745 0.0756092097247485\\
0.25281661443315 0.0707181340055984\\
0.255160607372719 0.0656293592496575\\
0.25752633267712 0.060456498291369\\
0.259913991838293 0.0557029752554999\\
0.262323788216312 0.0516950085940727\\
0.264755927056707 0.0482420390486675\\
0.267210615507944 0.0451131219115126\\
0.269688062639069 0.0425254443033696\\
0.272188479457518 0.0410856924952544\\
0.274712078927081 0.0412896344947036\\
0.277259075986048 0.0429603180578946\\
0.279829687565512 0.0453100976945821\\
0.282424132607843 0.0474224975843858\\
0.285042632085343 0.0485677113681408\\
0.287685409019059 0.0483683267166269\\
0.290352688497781 0.0469086751708752\\
0.293044697697214 0.0447447093381255\\
0.295761665899325 0.0427493233143318\\
0.298503824511873 0.0417563562892311\\
0.301271407088116 0.0421327116832534\\
0.304064649346707 0.0436070565309209\\
0.306883789191764 0.0454810654579287\\
0.309729066733141 0.046984411424869\\
0.31260072430687 0.0475458942880733\\
0.315499006495809 0.0469060338401416\\
0.318424160150465 0.0450925449444706\\
0.321376434410028 0.0423279013330887\\
0.324356080723581 0.0389282380157114\\
0.327363352871525 0.0352078623728203\\
0.330398506987186 0.0313916087871042\\
0.333461801578636 0.0275688742599965\\
0.336553497550709 0.0237327791869712\\
0.339673858227221 0.0198781746042184\\
0.342823149373397 0.0160677559952731\\
0.346001639218511 0.0124548966970436\\
0.349209598478727 0.00936968950511495\\
0.352447300380159 0.0075138044166985\\
0.355715020682139 0.00766969523087765\\
0.359013037700707 0.00948104731647427\\
0.362341632332315 0.0120015491852547\\
0.365701088077749 0.0147229573478506\\
0.369091691066278 0.017408769535842\\
0.372513730080021 0.0198487565606739\\
0.375967496578547 0.0218121639519782\\
0.379453284723694 0.0230994160750047\\
0.382971391404628 0.0236007121633864\\
0.386522116263126 0.023325332513773\\
0.390105761719099 0.0223945889285576\\
0.393722632996348 0.0210008256273688\\
0.397373038148561 0.0193503346188489\\
0.40105728808555 0.0176172026704067\\
0.404775696599732 0.0159230499714057\\
0.408528580392856 0.0143389559036704\\
0.412316259102975 0.0128978708950847\\
0.416139055331671 0.0116082470452073\\
0.419997294671531 0.010464750128947\\
0.423891305733878 0.009455373684917\\
0.427821420176762 0.008565670277393\\
0.431787972733202 0.00778104721325759\\
0.435791301239703 0.00708788722765928\\
0.439831746665023 0.00647400449956292\\
0.443909653139217 0.005928747855272\\
0.448025367982949 0.00544293075154283\\
0.45217924173707 0.00500868683004915\\
0.456371628192476 0.00461930295513728\\
0.46060288442024 0.00426905540829954\\
0.464873370802026 0.00395306067671012\\
0.46918345106078 0.00366714481289792\\
0.473533492291712 0.00340773162107843\\
0.47792386499356 0.00317174822147588\\
0.482354943100147 0.0029565458985672\\
0.486827104012228 0.00275983402689148\\
0.491340728629636 0.00257962500546798\\
0.495896201383721 0.00241418836962222\\
0.500493910270095 0.00226201251123381\\
0.505134246881676 0.00212177268885252\\
0.509817606442042 0.0019923042325489\\
0.514544387839093 0.00187258004031237\\
0.519314993659021 0.00176169162405817\\
0.524129830220606 0.00165883309690432\\
0.528989307609815 0.00156328760313537\\
0.533893839714735 0.0014744157819794\\
0.538843844260822 0.00139164592944981\\
0.54383974284648 0.00131446558204219\\
0.548881960978967 0.00124241429455352\\
0.55397092811064 0.00117507742379539\\
0.559107077675529 0.00111208076221048\\
0.564290847126254 0.00105308589175896\\
0.569522677971283 0.000997786150037783\\
0.574803015812535 0.000945903118333658\\
0.580132310383338 0.000897183555916913\\
0.585511015586724 0.000851396716942642\\
0.590939589534097 0.000808331996309799\\
0.596418494584246 0.000767796859117605\\
0.601948197382727 0.000729615015259185\\
0.607529168901607 0.000693624806453288\\
0.613161884479578 0.00065967777783764\\
0.618846823862439 0.000627637410296053\\
0.624584471243962 0.000597377993099186\\
0.630375315307128 0.00056878361931492\\
0.636219849265749 0.000541747288878211\\
0.642118570906476 0.000516170106274967\\
0.648071982631197 0.000491960561550777\\
0.654080591499826 0.000469033884852647\\
0.660144909273489 0.000447311465991766\\
0.666265452458115 0.000426720331611787\\
0.672442742348425 0.00040719267348867\\
0.67867730507233 0.00038866542229849\\
0.684969671635745 0.000371079861888612\\
0.691320377967813 0.000354381279691824\\
0.697729964966554 0.000338518649446449\\
0.704198978544929 0.000323444342839738\\
0.710727969677342 0.000309113867087006\\
0.717317494446562 0.000295485625803335\\
0.723968114091088 0.000282520700825288\\
0.73068039505295 0.000270182652903119\\
0.737454909025955 0.00025843733941443\\
0.744292233004376 0.000247252747452509\\
0.751192949332097 0.000236598840820593\\
0.758157645752211 0.000226447419619962\\
0.765186915457083 0.000216771991258111\\
0.772281357138865 0.000207547651825431\\
0.779441575040495 0.00019875097689703\\
0.786668179007158 0.000190359920912228\\
0.793961784538228 0.000182353724369379\\
0.801323012839689 0.000174712828149396\\
0.808752490877045 0.000167418794348761\\
0.816250851428723 0.000160454233062872\\
0.823818733139961 0.000153802734614285\\
0.831456780577206 0.000147448806768331\\
0.839165644283016 0.000141377816521577\\
0.846945980831459 0.000135575936087062\\
0.854798452884041 0.000130030092734797\\
0.862723729246145 0.000124727922177067\\
0.870722484923992 0.000119657725215973\\
0.878795401182132 0.000114808427395822\\
0.886943165601471 0.000110169541425648\\
0.89516647213783 0.000105731132157597\\
0.903466021181053 0.000101483783925447\\
0.911842519614657 9.74185700642057e-05\\
0.920296680876041 9.35270244469573e-05\\
0.92882922501725 8.98011148887954e-05\\
0.9374408787663 8.6233218280193e-05\\
0.946132375589077 8.28160973234569e-05\\
0.954904455751808 7.95428787562131e-05\\
0.963757866384109 7.64070329552386e-05\\
0.972693361542617 7.34023548224838e-05\\
0.981711702275219 7.05229458629066e-05\\
0.990813656685867 6.77631973708422e-05\\
1 6.51177746481065e-05\\
};
\addlegendentry{full};
\addplot [
color=red,
dash pattern=on 1pt off 3pt on 3pt off 3pt,
]
table[row sep=crcr]{
0.01 0.0115424291595802\\
0.0100927151463057 0.0115433081701934\\
0.0101862899024469 0.0115442036741292\\
0.0102807322383087 0.0115451159831118\\
0.0103760501976691 0.0115460454148418\\
0.0104722518988843 0.0115469922931143\\
0.0105693455355799 0.0115479569479388\\
0.0106673393773486 0.0115489397156622\\
0.0107662417704549 0.0115499409390941\\
0.010866061138546 0.0115509609676343\\
0.0109668059833687 0.0115520001574038\\
0.0110684848854941 0.0115530588713778\\
0.0111711065050482 0.0115541374795222\\
0.0112746795824495 0.0115552363589317\\
0.0113792129391532 0.0115563558939727\\
0.0114847154784029 0.0115574964764273\\
0.0115911961859889 0.0115586585056419\\
0.0116986641310131 0.0115598423886775\\
0.0118071284666619 0.0115610485404646\\
0.0119165984309856 0.0115622773839604\\
0.0120270833476851 0.0115635293503099\\
0.0121385926269063 0.0115648048790097\\
0.0122511357660412 0.0115661044180767\\
0.0123647223505372 0.0115674284242187\\
0.0124793620547131 0.0115687773630097\\
0.0125950646425836 0.0115701517090694\\
0.0127118399686903 0.011571551946245\\
0.0128296979789415 0.0115729785677986\\
0.012948648711459 0.0115744320765977\\
0.0130687022974335 0.01157591298531\\
0.0131898689619867 0.0115774218166027\\
0.0133121590250431 0.0115789591033462\\
0.0134355829022083 0.0115805253888212\\
0.0135601511056563 0.0115821212269325\\
0.0136858742450252 0.0115837471824251\\
0.0138127630283201 0.0115854038311068\\
0.0139408282628258 0.0115870917600747\\
0.0140700808560269 0.0115888115679476\\
0.0142005318165368 0.0115905638651029\\
0.0143321922550357 0.0115923492739186\\
0.0144650733852165 0.0115941684290218\\
0.0145991865247398 0.0115960219775417\\
0.0147345430961984 0.0115979105793686\\
0.0148711546280895 0.0115998349074191\\
0.0150090327557974 0.0116017956479069\\
0.0151481892225835 0.0116037935006196\\
0.0152886358805873 0.0116058291792026\\
0.0154303846918356 0.0116079034114485\\
0.0155734477292613 0.011610016939594\\
0.0157178371777316 0.0116121705206225\\
0.0158635653350859 0.0116143649265754\\
0.0160106446131832 0.0116166009448686\\
0.0161590875389592 0.011618879378618\\
0.0163089067554933 0.0116212010469715\\
0.0164601150230855 0.0116235667854493\\
0.0166127252203429 0.0116259774462921\\
0.016766750345277 0.0116284338988171\\
0.0169222035164104 0.011630937029783\\
0.0170790979738943 0.0116334877437631\\
0.0172374470806362 0.0116360869635278\\
0.017397264323438 0.0116387356304355\\
0.0175585633141447 0.0116414347048335\\
0.0177213577908036 0.0116441851664682\\
0.0178856616188346 0.0116469880149051\\
0.0180514887922111 0.0116498442699596\\
0.0182188534346517 0.011652754972137\\
0.0183877698008233 0.0116557211830845\\
0.0185582522775552 0.0116587439860535\\
0.0187303153850644 0.0116618244863735\\
0.0189039737781922 0.0116649638119375\\
0.0190792422476527 0.0116681631136999\\
0.019256135721292 0.0116714235661856\\
0.0194346692653602 0.0116747463680133\\
0.0196148580857943 0.0116781327424309\\
0.0197967175295133 0.0116815839378644\\
0.0199802630857255 0.011685101228481\\
0.0201655103872475 0.0116886859147663\\
0.0203524752118358 0.0116923393241161\\
0.0205411734835306 0.0116960628114432\\
0.0207316212740123 0.0116998577598004\\
0.0209238348039698 0.0117037255810184\\
0.0211178304444824 0.0117076677163613\\
0.0213136247184144 0.0117116856371985\\
0.0215112343018217 0.0117157808456945\\
0.0217106760253726 0.0117199548755163\\
0.0219119668757815 0.0117242092925595\\
0.0221151239972549 0.0117285456956942\\
0.0223201646929523 0.01173296571753\\
0.0225271064264598 0.0117374710252015\\
0.0227359668232772 0.0117420633211755\\
0.0229467636723194 0.0117467443440786\\
0.0231595149274315 0.0117515158695484\\
0.0233742387089181 0.0117563797111078\\
0.0235909533050863 0.0117613377210619\\
0.0238096771738036 0.0117663917914207\\
0.0240304289440697 0.0117715438548471\\
0.0242532274176035 0.0117767958856303\\
0.0244780915704444 0.0117821499006875\\
0.0247050405545683 0.0117876079605923\\
0.0249340936995188 0.0117931721706329\\
0.0251652705140539 0.0117988446819\\
0.0253985906878073 0.0118046276924052\\
0.0256340740929651 0.0118105234482316\\
0.0258717407859592 0.0118165342447174\\
0.0261116110091746 0.011822662427673\\
0.0263537051926739 0.0118289103946344\\
0.0265980439559376 0.011835280596152\\
0.0268446481096196 0.0118417755371173\\
0.0270935386573205 0.0118483977781296\\
0.0273447367973758 0.0118551499369016\\
0.0275982639246618 0.0118620346897074\\
0.0278541416324177 0.0118690547728743\\
0.0281123917140846 0.0118762129843185\\
0.0283730361651621 0.0118835121851272\\
0.0286360971850812 0.0118909553011899\\
0.0289015971790951 0.0118985453248775\\
0.029169558760188 0.0119062853167747\\
0.0294400047510004 0.0119141784074649\\
0.0297129581857733 0.0119222277993705\\
0.0299884423123103 0.011930436768651\\
0.0302664805939569 0.0119388086671602\\
0.0305470967115997 0.0119473469244656\\
0.0308303145656828 0.0119560550499312\\
0.0311161582782436 0.0119649366348671\\
0.0314046521949675 0.0119739953547483\\
0.0316958208872612 0.0119832349715048\\
0.0319896891543454 0.0119926593358858\\
0.0322862820253673 0.0120022723899017\\
0.0325856247615323 0.0120120781693454\\
0.0328877428582551 0.0120220808063979\\
0.0331926620473319 0.0120322845323192\\
0.0335004082991313 0.0120426936802305\\
0.0338110078248069 0.0120533126879891\\
0.0341244870785289 0.0120641461011602\\
0.0344408727597382 0.0120751985760906\\
0.0347601918154198 0.0120864748830875\\
0.0350824714423979 0.012097979909706\\
0.0354077390896527 0.0121097186641513\\
0.0357360224606579 0.0121216962787994\\
0.0360673495157403 0.0121339180138406\\
0.0364017484744614 0.0121463892610528\\
0.0367392478180207 0.0121591155477074\\
0.0370798762916817 0.0121721025406161\\
0.0374236629072198 0.012185356050322\\
0.0377706369453937 0.0121988820354437\\
0.0381208279584389 0.0122126866071766\\
0.0384742657725851 0.0122267760339593\\
0.0388309804905961 0.0122411567463129\\
0.0391910024943341 0.0122558353418589\\
0.039554362447347 0.0122708185905258\\
0.0399210912974805 0.0122861134399506\\
0.0402912202795135 0.0123017270210851\\
0.0406647809178186 0.0123176666540163\\
0.0410418050290471 0.0123339398540096\\
0.041422324724839 0.0123505543377856\\
0.0418063724145576 0.0123675180300411\\
0.0421939808080503 0.0123848390702253\\
0.0425851829184341 0.0124025258195826\\
0.042980012064908 0.0124205868684761\\
0.0433785018755899 0.0124390310440019\\
0.0437806862903817 0.0124578674179117\\
0.0441865995638594 0.0124771053148547\\
0.0445962762681909 0.0124967543209566\\
0.0450097512960805 0.0125168242927505\\
0.0454270598637405 0.0125373253664781\\
0.045848237513891 0.012558267967777\\
0.0462733201187869 0.0125796628217758\\
0.0467023438832734 0.012601520963615\\
0.0471353453478691 0.0126238537494161\\
0.0475723613918789 0.0126466728677205\\
0.0480134292365345 0.0126699903514229\\
0.048458586448165 0.0126938185902239\\
0.0489078709413959 0.0127181703436275\\
0.0493613209823792 0.0127430587545132\\
0.0498189751920516 0.0127684973633118\\
0.0502808725494248 0.0127945001228157\\
0.0507470523949047 0.0128210814136591\\
0.0512175544336424 0.0128482560605011\\
0.051692418738916 0.0128760393489523\\
0.0521716857555435 0.0129044470432822\\
0.0526553963033276 0.0129334954049512\\
0.0531435915805324 0.0129632012120131\\
0.0536363131673924 0.0129935817794341\\
0.0541336030296538 0.0130246549803809\\
0.0546355035221488 0.0130564392685314\\
0.055142057392403 0.0130889537014666\\
0.0556533077842765 0.0131222179652037\\
0.0561692982416381 0.0131562523999378\\
0.0566900727120743 0.0131910780270602\\
0.0572156755506325 0.0132267165775295\\
0.0577461515235982 0.0132631905216735\\
0.0582815458123085 0.0133005231005078\\
0.0588219040169996 0.0133387383586617\\
0.0593672721606913 0.0133778611790063\\
0.0599176966931062 0.0134179173190923\\
0.0604732244946265 0.0134589334495039\\
0.0610339028802862 0.0135009371942502\\
0.0615997796038017 0.0135439571733206\\
0.0621709028616383 0.0135880230475393\\
0.0627473212971158 0.013633165565866\\
0.0633290840045511 0.013679416615299\\
0.0639162405334401 0.0137268092735484\\
0.0645088408926769 0.0137753778646607\\
0.0651069355548146 0.0138251580177886\\
0.0657105754603627 0.013876186729314\\
0.0663198120221268 0.0139285024285515\\
0.0669346971295866 0.0139821450472704\\
0.0675552831533165 0.0140371560933006\\
0.0681816229494448 0.0140935787284997\\
0.0688137698641567 0.0141514578513882\\
0.069451777738237 0.0142108401847784\\
0.0700957009116562 0.0142717743687534\\
0.0707455942281988 0.0143343110593778\\
0.071401513040134 0.0143985030335564\\
0.0720635132129305 0.0144644053004904\\
0.0727316511300146 0.0145320752202178\\
0.0734059836975721 0.0146015726297678\\
0.0740865683493956 0.0146729599775024\\
0.0747734630517759 0.0147463024662704\\
0.0754667263084389 0.0148216682060522\\
0.0761664171655289 0.0148991283768347\\
0.0768725952166373 0.0149787574025207\\
0.0775853206078784 0.0150606331367533\\
0.0783046540430119 0.0151448370616112\\
0.0790306567886135 0.0152314545002258\\
0.0797633906792928 0.0153205748444629\\
0.0805029181229598 0.0154122917989258\\
0.0812493021061405 0.0155067036426537\\
0.0820026061993413 0.0156039135100233\\
0.0827628945624635 0.0157040296925111\\
0.0835302319502678 0.0158071659631359\\
0.0843046837178897 0.0159134419255871\\
0.0850863158264058 0.0160229833902457\\
0.0858751948484517 0.0161359227795354\\
0.0866713879738923 0.0162523995652925\\
0.0874749630155442 0.0163725607411304\\
0.0882859884149515 0.0164965613330927\\
0.0891045332482151 0.0166245649522442\\
0.0899306672318762 0.0167567443932542\\
0.0907644607288536 0.0168932822834753\\
0.0916059847544371 0.0170343717875317\\
0.0924553109823357 0.0171802173730039\\
0.0933125117507825 0.0173310356434482\\
0.0941776600686952 0.017487056245721\\
0.0950508296218951 0.0176485228594159\\
0.0959320947793824 0.0178156942771626\\
0.0968215305996709 0.0179888455856208\\
0.09771921283718 0.0181682694582252\\
0.0986252179486878 0.0183542775721447\\
0.0995396230998423 0.0185472021635219\\
0.100462506171734 0.018747397736904\\
0.101393945767529 0.0189552429468908\\
0.102334021219164 0.019171142672461\\
0.103282812594103 0.0193955303072555\\
0.104240400702156 0.0196288702923437\\
0.105206867102362 0.0198716609217777\\
0.106182294109938 0.0201244374556144\\
0.107166764803286 0.0203877755802018\\
0.108160363031071 0.0206622952614836\\
0.109163173419361 0.0209486650440659\\
0.110175281378839 0.0212476068569944\\
0.11119677311207 0.0215599013968493\\
0.112227735620851 0.0218863941701805\\
0.113268256713615 0.0222280022908244\\
0.114318425012915 0.022585722143713\\
0.115378329962966 0.0229606380459468\\
0.116448061837269 0.02335393205883\\
0.117527711746295 0.02376689513209\\
0.118617371645248 0.0242009397946811\\
0.119717134341897 0.0246576146467008\\
0.120827093504478 0.0251386209556927\\
0.121947343669674 0.0256458317200602\\
0.123077980250667 0.0261813136351391\\
0.124219099545262 0.0267473524870869\\
0.125370798744092 0.0273464826105307\\
0.126533175938894 0.027981521183567\\
0.127706330130864 0.0286556083056076\\
0.128890361239089 0.0293722540194273\\
0.130085370109057 0.030135393711351\\
0.131291458521247 0.0309494536698031\\
0.132508729199795 0.0318194290250912\\
0.13373728582125 0.0327509768628653\\
0.134977233023394 0.0337505280416368\\
0.136228676414165 0.034825422207777\\
0.137491722580642 0.0359840717678278\\
0.138766479098131 0.0372361622566876\\
0.140053054539322 0.0385928987843415\\
0.14135155848354 0.0400673112703202\\
0.142662101526074 0.0416746352947026\\
0.143984795287601 0.0434327910565216\\
0.14531975242369 0.0453629907908431\\
0.146667086634397 0.0474905160215872\\
0.148026912673951 0.049845721656433\\
0.149399346360526 0.0524653463175853\\
0.150784504586105 0.055394240691258\\
0.152182505326439 0.0586876729503211\\
0.15359346765109 0.0624144397307834\\
0.155017511733577 0.066661113241452\\
0.15645475886161 0.0715379039513353\\
0.157905331447418 0.0771868290039998\\
0.159369353038178 0.0837931513350637\\
0.160846948326536 0.0916013307665783\\
0.162338243161228 0.100936699180136\\
0.163843364557798 0.112232591584607\\
0.165362440709418 0.126055838307828\\
0.166895600997802 0.14309902202828\\
0.168442976004232 0.164032344890013\\
0.170004697520672 0.188935049516172\\
0.171580898561 0.215989331824323\\
0.173171713372335 0.240766627922995\\
0.174777277446468 0.260307726841836\\
0.176397727531404 0.27683968292346\\
0.178033201643011 0.29036075367315\\
0.179683839076772 0.296161134465991\\
0.18134978041965 0.29175663305603\\
0.183031167562061 0.277887687823516\\
0.184728143709963 0.25495848348804\\
0.186440853397049 0.226728960298316\\
0.188169442497056 0.202532877721054\\
0.189914058236193 0.188370005966056\\
0.191674849205682 0.182419331990008\\
0.193451965374405 0.179466271158737\\
0.195245558101686 0.177243009157103\\
0.19705578015018 0.177694077519561\\
0.198882785698881 0.179859915433738\\
0.200726730356257 0.180027088733239\\
0.202587771173502 0.177533447618176\\
0.204466066657912 0.173903512277406\\
0.206361776786386 0.169417722665003\\
0.208275063019052 0.162139580652105\\
0.210206088313016 0.149686319445224\\
0.212155017136245 0.131970104410358\\
0.214122015481573 0.111532123900542\\
0.216107250880838 0.0920183933922418\\
0.218110892419152 0.0774685529901764\\
0.220133110749303 0.0713371213581075\\
0.222174078106288 0.0736542999921195\\
0.224233968321985 0.0806856189353729\\
0.226312956839953 0.0878279306348092\\
0.228411220730382 0.0930823107030364\\
0.230528938705171 0.0961915232877398\\
0.232666291133146 0.0983186076741571\\
0.234823460055428 0.0996655945198788\\
0.237000629200933 0.0995061209081522\\
0.239197984002024 0.097464345853011\\
0.241415711610302 0.0939481453065199\\
0.243654000912547 0.0897132676337225\\
0.245913042546805 0.0851806058755983\\
0.248193028918626 0.0804335246899527\\
0.25049415421745 0.0756070649252632\\
0.25281661443315 0.0707237759668647\\
0.255160607372719 0.0656258150699374\\
0.25752633267712 0.0604552047126204\\
0.259913991838293 0.0557078742894566\\
0.262323788216312 0.0516897233761041\\
0.264755927056707 0.0482445281228479\\
0.267210615507944 0.0451150378176466\\
0.269688062639069 0.0425190759601119\\
0.272188479457518 0.0410935795926002\\
0.274712078927081 0.0412873411929499\\
0.277259075986048 0.0429528479972749\\
0.279829687565512 0.0453201014124316\\
0.282424132607843 0.0474217819636876\\
0.285042632085343 0.0485583475803553\\
0.287685409019059 0.0483762162609615\\
0.290352688497781 0.046911566655337\\
0.293044697697214 0.0447352588832496\\
0.295761665899325 0.0427538036032676\\
0.298503824511873 0.0417608773814098\\
0.301271407088116 0.0421258279113936\\
0.304064649346707 0.0436089330513964\\
0.306883789191764 0.0454847899836674\\
0.309729066733141 0.0469801687677731\\
0.31260072430687 0.0475457075600186\\
0.315499006495809 0.0469098445251256\\
0.318424160150465 0.0450908358432034\\
0.321376434410028 0.0423251507776432\\
0.324356080723581 0.0389304502083241\\
0.327363352871525 0.035209893750585\\
0.330398506987186 0.0313895246489102\\
0.333461801578636 0.0275672717006362\\
0.336553497550709 0.0237347301684697\\
0.339673858227221 0.0198792719734522\\
0.342823149373397 0.0160658588141874\\
0.346001639218511 0.0124544149943915\\
0.349209598478727 0.00937133871829491\\
0.352447300380159 0.00751366342029757\\
0.355715020682139 0.00766926100313106\\
0.359013037700707 0.00948124418484876\\
0.362341632332315 0.0120013026395544\\
0.365701088077749 0.0147228561368235\\
0.369091691066278 0.0174090502103869\\
0.372513730080021 0.0198488550258322\\
0.375967496578547 0.0218119575139844\\
0.379453284723694 0.0230993070908734\\
0.382971391404628 0.0236008201343244\\
0.386522116263126 0.0233254391942799\\
0.390105761719099 0.0223945820687994\\
0.393722632996348 0.0210007668232699\\
0.397373038148561 0.019350293184908\\
0.40105728808555 0.0176171929635186\\
0.404775696599732 0.0159230588783777\\
0.408528580392856 0.0143389699379419\\
0.412316259102975 0.0128978832748788\\
0.416139055331671 0.0116082558645807\\
0.419997294671531 0.0104647556467761\\
0.423891305733878 0.00945537675688216\\
0.427821420176762 0.00856567172414684\\
0.431787972733202 0.00778104765143277\\
0.435791301239703 0.00708788707398565\\
0.439831746665023 0.00647400401979253\\
0.443909653139217 0.00592874721260218\\
0.448025367982949 0.00544293004319138\\
0.45217924173707 0.00500868611210248\\
0.456371628192476 0.00461930225856983\\
0.46060288442024 0.00426905474900967\\
0.464873370802026 0.00395306006170239\\
0.46918345106078 0.00366714424404954\\
0.473533492291712 0.0034077310974137\\
0.47792386499356 0.00317174774052676\\
0.482354943100147 0.00295654545717685\\
0.486827104012228 0.00275983362167984\\
0.491340728629636 0.00257962463309307\\
0.495896201383721 0.00241418802691895\\
0.500493910270095 0.00226201219527995\\
0.505134246881676 0.00212177239699239\\
0.509817606442042 0.00199230396239283\\
0.514544387839093 0.00187257978972343\\
0.519314993659021 0.00176169139113248\\
0.524129830220606 0.00165883287994907\\
0.528989307609815 0.00156328740064665\\
0.533893839714735 0.00147441559262101\\
0.538843844260822 0.00139164575203382\\
0.54383974284648 0.0013144654155113\\
0.548881960978967 0.00124241413796538\\
0.55397092811064 0.00117507727630868\\
0.559107077675529 0.00111208062307277\\
0.564290847126254 0.00105308576029599\\
0.569522677971283 0.000997786025644154\\
0.574803015812535 0.000945903000464713\\
0.580132310383338 0.000897183444081633\\
0.585511015586724 0.000851396610697452\\
0.590939589534097 0.000808331895253191\\
0.596418494584246 0.000767796762885399\\
0.601948197382727 0.000729614923520423\\
0.607529168901607 0.000693624718906602\\
0.613161884479578 0.000659677694208099\\
0.618846823862439 0.000627637330332374\\
0.624584471243962 0.000597377916571281\\
0.630375315307128 0.000568783546011736\\
0.636219849265749 0.000541747218605821\\
0.642118570906476 0.000516170038854881\\
0.648071982631197 0.00049196049681845\\
0.654080591499826 0.000469033822656141\\
0.660144909273489 0.000447311406190572\\
0.666265452458115 0.000426720274075772\\
0.672442742348425 0.000407192618097129\\
0.67867730507233 0.000388665368939312\\
0.684969671635745 0.000371079810457519\\
0.691320377967813 0.000354381230091699\\
0.697729964966554 0.000338518601586725\\
0.704198978544929 0.00032344429663585\\
0.710727969677342 0.000309113822459897\\
0.717317494446562 0.00029548558267901\\
0.723968114091088 0.000282520659134409\\
0.73068039505295 0.000270182612580642\\
0.737454909025955 0.000258437300399267\\
0.744292233004376 0.000247252709687232\\
0.751192949332097 0.000236598804251159\\
0.758157645752211 0.00022644738419546\\
0.765186915457083 0.000216771956930537\\
0.772281357138865 0.000207547618549475\\
0.779441575040495 0.000198750944629891\\
0.786668179007158 0.000190359889613436\\
0.793961784538228 0.000182353694000632\\
0.801323012839689 0.000174712798674421\\
0.808752490877045 0.000167418765733171\\
0.816250851428723 0.000160454205274047\\
0.823818733139961 0.000153802707621251\\
0.831456780577206 0.000147448780541663\\
0.839165644283016 0.000141377791033294\\
0.846945980831459 0.00013557591131054\\
0.854798452884041 0.000130030068644685\\
0.862723729246145 0.000124727898749209\\
0.870722484923992 0.000119657702427339\\
0.878795401182132 0.000114808405224438\\
0.886943165601471 0.000110169519850535\\
0.89516647213783 0.000105731111158718\\
0.903466021181053 0.000101483763483645\\
0.911842519614657 9.74185501611625e-05\\
0.920296680876041 9.35270050651414e-05\\
0.92882922501725 8.98010960114221e-05\\
0.9374408787663 8.62331998911829e-05\\
0.946132375589077 8.28160794073967e-05\\
0.954904455751808 7.95428612983225e-05\\
0.963757866384109 7.64070159413363e-05\\
0.972693361542617 7.34023382389549e-05\\
0.981711702275219 7.05229296966778e-05\\
0.990813656685867 6.776318160935e-05\\
1 6.51177592792756e-05\\
};
\addlegendentry{reduced model};
\end{axis}
\end{tikzpicture}%

%% file: figures/abserr_smd_50.tex
%
%
%
%
\begin{tikzpicture}

\begin{axis}[%
width=\figwidth,
height=\figheight,
scale only axis,
separate axis lines,
every outer x axis line/.append style={darkgray!60!black},
every x tick label/.append style={font=\color{darkgray!60!black}},
xmode=log,
xmin=0.01,
xmax=1,
xminorticks=true,
xlabel={$\omega$},
every outer y axis line/.append style={darkgray!60!black},
every y tick label/.append style={font=\color{darkgray!60!black}},
ymode=log,
ymin=1e-12,
ymax=0.01,
yminorticks=true,
ylabel={$\sigma{}_{\text{max}}\text{($\mathcal{T}$(j}\omega\text{)-$\hat{\mathcal{T}}$}\text{(j}\omega\text{))}$},
]
\addplot [
color=red,
dash pattern=on 1pt off 3pt on 3pt off 3pt,
forget plot
]
table[row sep=crcr]{
0.01 1.4276920313705e-10\\
0.0100927151463057 1.42771263419152e-10\\
0.0101862899024469 1.42773341608268e-10\\
0.0102807322383087 1.42775480293254e-10\\
0.0103760501976691 1.42777649193391e-10\\
0.0104722518988843 1.42779856594842e-10\\
0.0105693455355799 1.42782112784142e-10\\
0.0106673393773486 1.42784398879764e-10\\
0.0107662417704549 1.42786743506319e-10\\
0.010866061138546 1.42789119373951e-10\\
0.0109668059833687 1.42791537788966e-10\\
0.0110684848854941 1.42794019785743e-10\\
0.0111711065050482 1.42796531202815e-10\\
0.0112746795824495 1.42799100498724e-10\\
0.0113792129391532 1.42801703049348e-10\\
0.0114847154784029 1.42804358730718e-10\\
0.0115911961859889 1.42807075992643e-10\\
0.0116986641310131 1.4280983573653e-10\\
0.0118071284666619 1.42812644550105e-10\\
0.0119165984309856 1.4281550848305e-10\\
0.0120270833476851 1.42818419768575e-10\\
0.0121385926269063 1.42821388378822e-10\\
0.0122511357660412 1.42824427832865e-10\\
0.0123647223505372 1.42827504466284e-10\\
0.0124793620547131 1.42830650970266e-10\\
0.0125950646425836 1.42833843915433e-10\\
0.0127118399686903 1.42837106115606e-10\\
0.0128296979789415 1.42840433443305e-10\\
0.012948648711459 1.42843809810989e-10\\
0.0130687022974335 1.42847249419292e-10\\
0.0131898689619867 1.42850762741976e-10\\
0.0133121590250431 1.42854339830177e-10\\
0.0134355829022083 1.42857979871204e-10\\
0.0135601511056563 1.42861693473957e-10\\
0.0136858742450252 1.42865467490641e-10\\
0.0138127630283201 1.42869322542378e-10\\
0.0139408282628258 1.42873245094789e-10\\
0.0140700808560269 1.428772347081e-10\\
0.0142005318165368 1.42881309943414e-10\\
0.0143321922550357 1.42885455852498e-10\\
0.0144650733852165 1.42889677053449e-10\\
0.0145991865247398 1.42893982725431e-10\\
0.0147345430961984 1.42898367335897e-10\\
0.0148711546280895 1.42902824986741e-10\\
0.0150090327557974 1.42907381735611e-10\\
0.0151481892225835 1.42912009762366e-10\\
0.0152886358805873 1.42916736719521e-10\\
0.0154303846918356 1.42921544187051e-10\\
0.0155734477292613 1.42926444800301e-10\\
0.0157178371777316 1.42931426442016e-10\\
0.0158635653350859 1.42936519156883e-10\\
0.0160106446131832 1.4294168838271e-10\\
0.0161590875389592 1.42946977503347e-10\\
0.0163089067554933 1.42952349099474e-10\\
0.0164601150230855 1.42957817471218e-10\\
0.0166127252203429 1.42963398131916e-10\\
0.016766750345277 1.42969068192578e-10\\
0.0169222035164104 1.42974860927207e-10\\
0.0170790979738943 1.42980763271703e-10\\
0.0172374470806362 1.42986777201622e-10\\
0.017397264323438 1.42992880711853e-10\\
0.0175585633141447 1.42999112210714e-10\\
0.0177213577908036 1.4300547477757e-10\\
0.0178856616188346 1.43011939617214e-10\\
0.0180514887922111 1.43018528742209e-10\\
0.0182188534346517 1.43025236072178e-10\\
0.0183877698008233 1.43032069127367e-10\\
0.0185582522775552 1.43039050841148e-10\\
0.0187303153850644 1.43046149525926e-10\\
0.0189039737781922 1.43053380921344e-10\\
0.0190792422476527 1.43060754251215e-10\\
0.019256135721292 1.43068253632963e-10\\
0.0194346692653602 1.43075889616461e-10\\
0.0196148580857943 1.43083682340436e-10\\
0.0197967175295133 1.43091608434486e-10\\
0.0199802630857255 1.43099694305447e-10\\
0.0201655103872475 1.43107933391614e-10\\
0.0203524752118358 1.43116326991426e-10\\
0.0205411734835306 1.43124869357494e-10\\
0.0207316212740123 1.43133579926618e-10\\
0.0209238348039698 1.43142453835511e-10\\
0.0211178304444824 1.43151483344977e-10\\
0.0213136247184144 1.43160692885715e-10\\
0.0215112343018217 1.43170059032641e-10\\
0.0217106760253726 1.43179624516587e-10\\
0.0219119668757815 1.43189354820524e-10\\
0.0221151239972549 1.43199279892348e-10\\
0.0223201646929523 1.43209374127261e-10\\
0.0225271064264598 1.43219675133909e-10\\
0.0227359668232772 1.43230155266831e-10\\
0.0229467636723194 1.4324084782676e-10\\
0.0231595149274315 1.4325172228632e-10\\
0.0233742387089181 1.43262806354289e-10\\
0.0235909533050863 1.43274110956354e-10\\
0.0238096771738036 1.43285626643192e-10\\
0.0240304289440697 1.43297346081424e-10\\
0.0242532274176035 1.43309292639752e-10\\
0.0244780915704444 1.4332145889967e-10\\
0.0247050405545683 1.43333861087792e-10\\
0.0249340936995188 1.43346501579506e-10\\
0.0251652705140539 1.4335937019843e-10\\
0.0253985906878073 1.4337248527106e-10\\
0.0256340740929651 1.43385844093505e-10\\
0.0258717407859592 1.43399470732813e-10\\
0.0261116110091746 1.43413329247187e-10\\
0.0263537051926739 1.43427456186431e-10\\
0.0265980439559376 1.43441855557081e-10\\
0.0268446481096196 1.43456523639534e-10\\
0.0270935386573205 1.43471466124457e-10\\
0.0273447367973758 1.43486696051152e-10\\
0.0275982639246618 1.43502198441702e-10\\
0.0278541416324177 1.43518011403912e-10\\
0.0281123917140846 1.43534114689867e-10\\
0.0283730361651621 1.43550531640563e-10\\
0.0286360971850812 1.43567245535348e-10\\
0.0289015971790951 1.43584275635402e-10\\
0.029169558760188 1.43601643259617e-10\\
0.0294400047510004 1.43619320991254e-10\\
0.0297129581857733 1.43637345367049e-10\\
0.0299884423123103 1.436557020892e-10\\
0.0302664805939569 1.43674390806298e-10\\
0.0305470967115997 1.43693455304458e-10\\
0.0308303145656828 1.43712877241844e-10\\
0.0311161582782436 1.43732665130419e-10\\
0.0314046521949675 1.43752835598902e-10\\
0.0316958208872612 1.43773388958283e-10\\
0.0319896891543454 1.43794312279188e-10\\
0.0322862820253673 1.43815645647957e-10\\
0.0325856247615323 1.43837376760219e-10\\
0.0328877428582551 1.43859531980668e-10\\
0.0331926620473319 1.4388209773763e-10\\
0.0335004082991313 1.43905095504533e-10\\
0.0338110078248069 1.43928532534678e-10\\
0.0341244870785289 1.43952412908454e-10\\
0.0344408727597382 1.43976741326477e-10\\
0.0347601918154198 1.44001552295168e-10\\
0.0350824714423979 1.44026802346982e-10\\
0.0354077390896527 1.44052553538765e-10\\
0.0357360224606579 1.44078785788162e-10\\
0.0360673495157403 1.44105524851232e-10\\
0.0364017484744614 1.44132772140042e-10\\
0.0367392478180207 1.44160545524407e-10\\
0.0370798762916817 1.44188844496856e-10\\
0.0374236629072198 1.44217677741611e-10\\
0.0377706369453937 1.44247072247635e-10\\
0.0381208279584389 1.44277014747243e-10\\
0.0384742657725851 1.44307538355793e-10\\
0.0388309804905961 1.44338655650136e-10\\
0.0391910024943341 1.44370341000579e-10\\
0.039554362447347 1.44402663568565e-10\\
0.0399210912974805 1.44435582276192e-10\\
0.0402912202795135 1.44469135579181e-10\\
0.0406647809178186 1.44503355566664e-10\\
0.0410418050290471 1.44538209813689e-10\\
0.041422324724839 1.44573735708508e-10\\
0.0418063724145576 1.44609952928871e-10\\
0.0421939808080503 1.44646859669169e-10\\
0.0425851829184341 1.44684487254265e-10\\
0.042980012064908 1.44722831930609e-10\\
0.0433785018755899 1.44761913491914e-10\\
0.0437806862903817 1.44801747570701e-10\\
0.0441865995638594 1.44842367711857e-10\\
0.0445962762681909 1.44883754170755e-10\\
0.0450097512960805 1.44925952538049e-10\\
0.0454270598637405 1.44968968998363e-10\\
0.045848237513891 1.45012806963002e-10\\
0.0462733201187869 1.45057506946722e-10\\
0.0467023438832734 1.45103074573583e-10\\
0.0471353453478691 1.45149525792339e-10\\
0.0475723613918789 1.45196878039232e-10\\
0.0480134292365345 1.45245162539119e-10\\
0.048458586448165 1.45294376029162e-10\\
0.0489078709413959 1.45344555349982e-10\\
0.0493613209823792 1.45395720174119e-10\\
0.0498189751920516 1.45447874970209e-10\\
0.0502808725494248 1.45501062791921e-10\\
0.0507470523949047 1.45555299871181e-10\\
0.0512175544336424 1.45610598643957e-10\\
0.051692418738916 1.45666976459363e-10\\
0.0521716857555435 1.4572448352069e-10\\
0.0526553963033276 1.45783119455452e-10\\
0.0531435915805324 1.45842899409214e-10\\
0.0536363131673924 1.45903887203256e-10\\
0.0541336030296538 1.45966070290695e-10\\
0.0546355035221488 1.46029498490904e-10\\
0.055142057392403 1.46094184222719e-10\\
0.0556533077842765 1.46160161840847e-10\\
0.0561692982416381 1.46227457991141e-10\\
0.0566900727120743 1.46296116323914e-10\\
0.0572156755506325 1.46366149305882e-10\\
0.0577461515235982 1.464375906956e-10\\
0.0582815458123085 1.46510464599922e-10\\
0.0588219040169996 1.46584826257249e-10\\
0.0593672721606913 1.46660693264948e-10\\
0.0599176966931062 1.46738104004702e-10\\
0.0604732244946265 1.46817087901129e-10\\
0.0610339028802862 1.46897699730173e-10\\
0.0615997796038017 1.46979968141425e-10\\
0.0621709028616383 1.47063918392861e-10\\
0.0627473212971158 1.47149623767032e-10\\
0.0633290840045511 1.47237096022912e-10\\
0.0639162405334401 1.47326407295021e-10\\
0.0645088408926769 1.47417568065368e-10\\
0.0651069355548146 1.47510665700127e-10\\
0.0657105754603627 1.4760570971291e-10\\
0.0663198120221268 1.4770278540148e-10\\
0.0669346971295866 1.47801937690822e-10\\
0.0675552831533165 1.47903198133182e-10\\
0.0681816229494448 1.48006646726731e-10\\
0.0688137698641567 1.48112350762757e-10\\
0.069451777738237 1.48220365213758e-10\\
0.0700957009116562 1.48330741911756e-10\\
0.0707455942281988 1.48443569916274e-10\\
0.071401513040134 1.4855890246023e-10\\
0.0720635132129305 1.48676839244725e-10\\
0.0727316511300146 1.48797442918809e-10\\
0.0734059836975721 1.4892080508439e-10\\
0.0740865683493956 1.49047010551909e-10\\
0.0747734630517759 1.49176175091811e-10\\
0.0754667263084389 1.49308368086007e-10\\
0.0761664171655289 1.49443701656348e-10\\
0.0768725952166373 1.49582304309206e-10\\
0.0775853206078784 1.49724291960026e-10\\
0.0783046540430119 1.49869787091247e-10\\
0.0790306567886135 1.50018913786627e-10\\
0.0797633906792928 1.50171835959577e-10\\
0.0805029181229598 1.50328710044066e-10\\
0.0812493021061405 1.50489702956172e-10\\
0.0820026061993413 1.50655002798608e-10\\
0.0827628945624635 1.50824793130378e-10\\
0.0835302319502678 1.50999278032001e-10\\
0.0843046837178897 1.51178713040473e-10\\
0.0850863158264058 1.51363313050695e-10\\
0.0858751948484517 1.51553404814942e-10\\
0.0866713879738923 1.51749218590374e-10\\
0.0874749630155442 1.51951136990394e-10\\
0.0882859884149515 1.52159468432585e-10\\
0.0891045332482151 1.52374644156253e-10\\
0.0899306672318762 1.52597078383084e-10\\
0.0907644607288536 1.52827231835315e-10\\
0.0916059847544371 1.53065616588391e-10\\
0.0924553109823357 1.53312856300714e-10\\
0.0933125117507825 1.53569586569426e-10\\
0.0941776600686952 1.53836471542591e-10\\
0.0950508296218951 1.54114395424388e-10\\
0.0959320947793824 1.54404214580816e-10\\
0.0968215305996709 1.5470696496851e-10\\
0.09771921283718 1.55023703200511e-10\\
0.0986252179486878 1.55355804974993e-10\\
0.0995396230998423 1.55704626697719e-10\\
0.100462506171734 1.56071832956989e-10\\
0.101393945767529 1.56459275470444e-10\\
0.102334021219164 1.56869045308255e-10\\
0.103282812594103 1.57303528029146e-10\\
0.104240400702156 1.57765515697776e-10\\
0.105206867102362 1.58258059418281e-10\\
0.106182294109938 1.58784815176473e-10\\
0.107166764803286 1.59349913921624e-10\\
0.108160363031071 1.59958115628854e-10\\
0.109163173419361 1.60614979611194e-10\\
0.110175281378839 1.61326873243176e-10\\
0.11119677311207 1.62101281355966e-10\\
0.112227735620851 1.62946888113218e-10\\
0.113268256713615 1.63873799760739e-10\\
0.114318425012915 1.64893907111467e-10\\
0.115378329962966 1.66021131269801e-10\\
0.116448061837269 1.67271954298637e-10\\
0.117527711746295 1.6866562776512e-10\\
0.118617371645248 1.70225121494669e-10\\
0.119717134341897 1.71977559772679e-10\\
0.120827093504478 1.73955185724809e-10\\
0.121947343669674 1.76196428834553e-10\\
0.123077980250667 1.78747098544433e-10\\
0.124219099545262 1.81662148407637e-10\\
0.125370798744092 1.85007163446835e-10\\
0.126533175938894 1.888611345659e-10\\
0.127706330130864 1.9331900059017e-10\\
0.128890361239089 1.98495219709597e-10\\
0.130085370109057 2.04527988764445e-10\\
0.131291458521247 2.11584566563033e-10\\
0.132508729199795 2.19867810682039e-10\\
0.13373728582125 2.29624541808148e-10\\
0.134977233023394 2.41156220273363e-10\\
0.136228676414165 2.54832376888013e-10\\
0.137491722580642 2.71108574086449e-10\\
0.138766479098131 2.90551125093523e-10\\
0.140053054539322 3.13868787516675e-10\\
0.14135155848354 3.41958944129638e-10\\
0.142662101526074 3.7597047517337e-10\\
0.143984795287601 4.17395175480448e-10\\
0.14531975242369 4.68198586879647e-10\\
0.146667086634397 5.31015207920105e-10\\
0.148026912673951 6.09439887322826e-10\\
0.149399346360526 7.08475918371963e-10\\
0.150784504586105 8.35239244677648e-10\\
0.152182505326439 1.00009191245906e-09\\
0.15359346765109 1.21852195285585e-09\\
0.155017511733577 1.5143610701143e-09\\
0.15645475886161 1.92548999018181e-09\\
0.157905331447418 2.51435895654783e-09\\
0.159369353038178 3.3882526708748e-09\\
0.160846948326536 4.74023863611693e-09\\
0.162338243161228 6.93635116836766e-09\\
0.163843364557798 1.07117493840857e-08\\
0.165362440709418 1.76365939158078e-08\\
0.166895600997802 3.12700724955797e-08\\
0.168442976004232 6.00192964904608e-08\\
0.170004697520672 1.2333020934188e-07\\
0.171580898561 2.57922543065513e-07\\
0.173171713372335 4.93415847574722e-07\\
0.174777277446468 7.902249497506e-07\\
0.176397727531404 1.16343914724694e-06\\
0.178033201643011 1.70237077098195e-06\\
0.179683839076772 2.2344245979206e-06\\
0.18134978041965 2.55414474501239e-06\\
0.183031167562061 2.64923931238201e-06\\
0.184728143709963 2.89573750409482e-06\\
0.186440853397049 4.27401696137577e-06\\
0.188169442497056 7.4030857400672e-06\\
0.189914058236193 1.17263898334993e-05\\
0.191674849205682 1.52867118493227e-05\\
0.193451965374405 1.66796913669223e-05\\
0.195245558101686 1.67084179187781e-05\\
0.19705578015018 1.65880052431897e-05\\
0.198882785698881 1.64602528240586e-05\\
0.200726730356257 1.59745158635407e-05\\
0.202587771173502 1.57601488320206e-05\\
0.204466066657912 1.71155594412644e-05\\
0.206361776786386 2.02395483551665e-05\\
0.208275063019052 2.42629798248942e-05\\
0.210206088313016 2.83219318699031e-05\\
0.212155017136245 3.18675343025949e-05\\
0.214122015481573 3.51377336036874e-05\\
0.216107250880838 3.94427631263676e-05\\
0.218110892419152 4.74201450707142e-05\\
0.220133110749303 6.42708254904633e-05\\
0.222174078106288 0.000106516793998835\\
0.224233968321985 0.000282740039191797\\
0.226312956839953 0.000179051119243816\\
0.228411220730382 7.99659591011485e-05\\
0.230528938705171 5.13765676434559e-05\\
0.232666291133146 3.90406924528734e-05\\
0.234823460055428 3.2244147199734e-05\\
0.237000629200933 2.74318493622099e-05\\
0.239197984002024 2.33763757133099e-05\\
0.241415711610302 1.99082278840984e-05\\
0.243654000912547 1.73070072272612e-05\\
0.245913042546805 1.56987363108671e-05\\
0.248193028918626 1.47563661634362e-05\\
0.25049415421745 1.39167416422972e-05\\
0.25281661443315 1.27749551304202e-05\\
0.255160607372719 1.13173262570349e-05\\
0.25752633267712 9.92610065176129e-06\\
0.259913991838293 9.02602026548909e-06\\
0.262323788216312 8.7436432805755e-06\\
0.264755927056707 9.00873357982341e-06\\
0.267210615507944 9.70242725146198e-06\\
0.269688062639069 1.07094401383113e-05\\
0.272188479457518 1.19410449068174e-05\\
0.274712078927081 1.32710225725031e-05\\
0.277259075986048 1.45068646297593e-05\\
0.279829687565512 1.54751769032751e-05\\
0.282424132607843 1.61204216211726e-05\\
0.285042632085343 1.64946572012242e-05\\
0.287685409019059 1.66566106477898e-05\\
0.290352688497781 1.66202395037862e-05\\
0.293044697697214 1.64023266644368e-05\\
0.295761665899325 1.60411501466444e-05\\
0.298503824511873 1.55388460615992e-05\\
0.301271407088116 1.48495041347116e-05\\
0.304064649346707 1.39430213001869e-05\\
0.306883789191764 1.28321679961154e-05\\
0.309729066733141 1.15377819874056e-05\\
0.31260072430687 1.00834480733462e-05\\
0.315499006495809 8.53192801556245e-06\\
0.318424160150465 7.00082241952531e-06\\
0.321376434410028 5.63699734316208e-06\\
0.324356080723581 4.56918298345805e-06\\
0.327363352871525 3.8550567041558e-06\\
0.330398506987186 3.45283521014981e-06\\
0.333461801578636 3.25605014386477e-06\\
0.336553497550709 3.15685243633482e-06\\
0.339673858227221 3.07887501651747e-06\\
0.342823149373397 2.97982730530477e-06\\
0.346001639218511 2.83983738418359e-06\\
0.349209598478727 2.65062961509087e-06\\
0.352447300380159 2.41483730074572e-06\\
0.355715020682139 2.1475598953441e-06\\
0.359013037700707 1.87098732211816e-06\\
0.362341632332315 1.6057819113291e-06\\
0.365701088077749 1.36586281974508e-06\\
0.369091691066278 1.15657824278345e-06\\
0.372513730080021 9.75029501585169e-07\\
0.375967496578547 8.13593587579397e-07\\
0.379453284723694 6.65693276318582e-07\\
0.382971391404628 5.29734882534502e-07\\
0.386522116263126 4.08560059448396e-07\\
0.390105761719099 3.06014297152347e-07\\
0.393722632996348 2.23991873320916e-07\\
0.397373038148561 1.615947720891e-07\\
0.40105728808555 1.15922908439948e-07\\
0.404775696599732 8.33337510780686e-08\\
0.408528580392856 6.03886970255493e-08\\
0.412316259102975 4.42913756645378e-08\\
0.416139055331671 3.29570552628462e-08\\
0.419997294671531 2.49082274019363e-08\\
0.423891305733878 1.91261313008767e-08\\
0.427821420176762 1.49170076429437e-08\\
0.431787972733202 1.18096777860561e-08\\
0.435791301239703 9.48297664254882e-09\\
0.439831746665023 7.71633305366048e-09\\
0.443909653139217 6.35675024999055e-09\\
0.448025367982949 5.29689972653773e-09\\
0.45217924173707 4.46058485059318e-09\\
0.456371628192476 3.79304890915305e-09\\
0.46060288442024 3.25446413735023e-09\\
0.464873370802026 2.8155171303294e-09\\
0.46918345106078 2.45438493670302e-09\\
0.473533492291712 2.15464148944667e-09\\
0.47792386499356 1.90379232592903e-09\\
0.482354943100147 1.69223790475795e-09\\
0.486827104012228 1.51253228575731e-09\\
0.491340728629636 1.3588475508193e-09\\
0.495896201383721 1.22658290943641e-09\\
0.500493910270095 1.11207668877348e-09\\
0.505134246881676 1.01239208261523e-09\\
0.509817606442042 9.25156399751853e-10\\
0.514544387839093 8.48439385514885e-10\\
0.519314993659021 7.80660369392481e-10\\
0.524129830220606 7.20516859999711e-10\\
0.528989307609815 6.6692922100282e-10\\
0.533893839714735 6.18997494310676e-10\\
0.538843844260822 5.75967469488941e-10\\
0.54383974284648 5.37203868067634e-10\\
0.548881960978967 5.02168984343938e-10\\
0.55397092811064 4.70405598936919e-10\\
0.559107077675529 4.41523212466211e-10\\
0.564290847126254 4.15186913117697e-10\\
0.569522677971283 3.91108306227323e-10\\
0.574803015812535 3.69038102435611e-10\\
0.580132310383338 3.48760023694402e-10\\
0.585511015586724 3.30085777049495e-10\\
0.590939589534097 3.1285087514633e-10\\
0.596418494584246 2.9691117419535e-10\\
0.601948197382727 2.82139955284267e-10\\
0.607529168901607 2.684255040853e-10\\
0.613161884479578 2.55669038893569e-10\\
0.618846823862439 2.43782983145066e-10\\
0.624584471243962 2.3268948736872e-10\\
0.630375315307128 2.22319178073214e-10\\
0.636219849265749 2.12610078400601e-10\\
0.642118570906476 2.03506697165142e-10\\
0.648071982631197 1.94959228788013e-10\\
0.654080591499826 1.8692288502303e-10\\
0.660144909273489 1.79357293201139e-10\\
0.666265452458115 1.72225996530438e-10\\
0.672442742348425 1.65496004471267e-10\\
0.67867730507233 1.5913741014083e-10\\
0.684969671635745 1.53123048556632e-10\\
0.691320377967813 1.4742820601679e-10\\
0.697729964966554 1.42030357133932e-10\\
0.704198978544929 1.36908938700742e-10\\
0.710727969677342 1.32045149761653e-10\\
0.717317494446562 1.27421770420426e-10\\
0.723968114091088 1.23023008047958e-10\\
0.73068039505295 1.18834358507001e-10\\
0.737454909025955 1.14842478276159e-10\\
0.744292233004376 1.11035078919075e-10\\
0.751192949332097 1.07400825708609e-10\\
0.758157645752211 1.03929251264055e-10\\
0.765186915457083 1.00610677738442e-10\\
0.772281357138865 9.74361450543668e-11\\
0.779441575040495 9.43973467204917e-11\\
0.786668179007158 9.14865767792051e-11\\
0.793961784538228 8.86966756377605e-11\\
0.801323012839689 8.60209839474839e-11\\
0.808752490877045 8.34533020602355e-11\\
0.816250851428723 8.09878525168193e-11\\
0.823818733139961 7.86192449113454e-11\\
0.831456780577206 7.63424453392551e-11\\
0.839165644283016 7.41527480685287e-11\\
0.846945980831459 7.20457508504596e-11\\
0.854798452884041 7.00173311025269e-11\\
0.862723729246145 6.8063624690856e-11\\
0.870722484923992 6.61810048855938e-11\\
0.878795401182132 6.43660688324664e-11\\
0.886943165601471 6.2615617728034e-11\\
0.89516647213783 6.0926642163929e-11\\
0.903466021181053 5.92963122065625e-11\\
0.911842519614657 5.77219588375322e-11\\
0.920296680876041 5.62010692559048e-11\\
0.92882922501725 5.47312710647816e-11\\
0.9374408787663 5.33103244956531e-11\\
0.946132375589077 5.19361151978122e-11\\
0.954904455751808 5.06066427678627e-11\\
0.963757866384109 4.93200153212877e-11\\
0.972693361542617 4.80744431990757e-11\\
0.981711702275219 4.68682282201945e-11\\
0.990813656685867 4.56997642213549e-11\\
1 4.45675248779159e-11\\
};
\end{axis}
\end{tikzpicture}%

%% file: figures/relatverr_smd_50.tex
%
%
%
%
\begin{tikzpicture}

\begin{axis}[%
width=\figwidth,
height=\figheight,
scale only axis,
separate axis lines,
every outer x axis line/.append style={darkgray!60!black},
every x tick label/.append style={font=\color{darkgray!60!black}},
xmode=log,
xmin=0.01,
xmax=1,
xminorticks=true,
xlabel={$\omega$},
every outer y axis line/.append style={darkgray!60!black},
every y tick label/.append style={font=\color{darkgray!60!black}},
ymode=log,
ymin=1e-09,
ymax=0.01,
yminorticks=true,
ylabel={$\frac{\sigma{}_{\text{max}}\text{($\mathcal{T}$(j}\omega\text{)}-\hat{\text{$\mathcal{T}$}}\text{(j}\omega\text{))}}{\sigma{}_{\text{max}}\text{($\mathcal{T}$(j}\omega\text{))}}$},
]
\addplot [
color=red,
dash pattern=on 1pt off 3pt on 3pt off 3pt,
forget plot
]
table[row sep=crcr]{
0.01 1.2369077793623e-08\\
0.0100927151463057 1.23683143843434e-08\\
0.0101862899024469 1.23675349728421e-08\\
0.0102807322383087 1.23667429211817e-08\\
0.0103760501976691 1.23659352729507e-08\\
0.0104722518988843 1.23651124063757e-08\\
0.0105693455355799 1.23642748663584e-08\\
0.0106673393773486 1.23634206654614e-08\\
0.0107662417704549 1.23625519233622e-08\\
0.010866061138546 1.23616659115219e-08\\
0.0109668059833687 1.23607632365204e-08\\
0.0110684848854941 1.23598453396988e-08\\
0.0111711065050482 1.23589088775165e-08\\
0.0112746795824495 1.23579559195443e-08\\
0.0113792129391532 1.23569839337218e-08\\
0.0114847154784029 1.23559942316917e-08\\
0.0115911961859889 1.23549871282836e-08\\
0.0116986641310131 1.23539605471756e-08\\
0.0118071284666619 1.23529146266431e-08\\
0.0119165984309856 1.2351849450167e-08\\
0.0120270833476851 1.23507638979978e-08\\
0.0121385926269063 1.23496583762242e-08\\
0.0122511357660412 1.23485335887381e-08\\
0.0123647223505372 1.23473861512575e-08\\
0.0124793620547131 1.23462184077895e-08\\
0.0125950646425836 1.23450278415953e-08\\
0.0127118399686903 1.23438159240545e-08\\
0.0128296979789415 1.23425817883388e-08\\
0.012948648711459 1.23413235249971e-08\\
0.0130687022974335 1.23400418321163e-08\\
0.0131898689619867 1.23387370754216e-08\\
0.0133121590250431 1.23374078463747e-08\\
0.0134355829022083 1.23360535156174e-08\\
0.0135601511056563 1.23346744295834e-08\\
0.0136858742450252 1.23332688727512e-08\\
0.0138127630283201 1.23318380342442e-08\\
0.0139408282628258 1.23303801434615e-08\\
0.0140700808560269 1.23288945492333e-08\\
0.0142005318165368 1.23273822283679e-08\\
0.0143321922550357 1.23258412544513e-08\\
0.0144650733852165 1.23242713776443e-08\\
0.0145991865247398 1.23226727291536e-08\\
0.0147345430961984 1.23210441589315e-08\\
0.0148711546280895 1.23193844732341e-08\\
0.0150090327557974 1.23176952203207e-08\\
0.0151481892225835 1.23159732908412e-08\\
0.0152886358805873 1.2314220343625e-08\\
0.0154303846918356 1.23124340533102e-08\\
0.0155734477292613 1.23106147568508e-08\\
0.0157178371777316 1.23087606448493e-08\\
0.0158635653350859 1.2306873524636e-08\\
0.0160106446131832 1.23049496254966e-08\\
0.0161590875389592 1.23029918728436e-08\\
0.0163089067554933 1.23009962210414e-08\\
0.0164601150230855 1.229896306127e-08\\
0.0166127252203429 1.22968928724941e-08\\
0.016766750345277 1.2294782814239e-08\\
0.0169222035164104 1.22926348602996e-08\\
0.0170790979738943 1.22904469823634e-08\\
0.0172374470806362 1.2288218428818e-08\\
0.017397264323438 1.22859463703936e-08\\
0.0175585633141447 1.22836331516891e-08\\
0.0177213577908036 1.2281278062472e-08\\
0.0178856616188346 1.22788776374516e-08\\
0.0180514887922111 1.22764327565821e-08\\
0.0182188534346517 1.22739418683994e-08\\
0.0183877698008233 1.22714045697496e-08\\
0.0185582522775552 1.22688217594148e-08\\
0.0187303153850644 1.22661896305668e-08\\
0.0189039737781922 1.22635084248776e-08\\
0.0190792422476527 1.22607778036127e-08\\
0.019256135721292 1.22579952551419e-08\\
0.0194346692653602 1.22551605121355e-08\\
0.0196148580857943 1.22522741056448e-08\\
0.0197967175295133 1.22493328180028e-08\\
0.0199802630857255 1.22463376712926e-08\\
0.0201655103872475 1.22432868413813e-08\\
0.0203524752118358 1.22401791536991e-08\\
0.0205411734835306 1.2237012807333e-08\\
0.0207316212740123 1.22337881302995e-08\\
0.0209238348039698 1.22305033474977e-08\\
0.0211178304444824 1.22271564135837e-08\\
0.0213136247184144 1.22237480053384e-08\\
0.0215112343018217 1.22202746858567e-08\\
0.0217106760253726 1.22167386392115e-08\\
0.0219119668757815 1.22131354284372e-08\\
0.0221151239972549 1.22094660897016e-08\\
0.0223201646929523 1.22057268925569e-08\\
0.0225271064264598 1.22019194668885e-08\\
0.0227359668232772 1.21980398512211e-08\\
0.0229467636723194 1.21940892464632e-08\\
0.0231595149274315 1.21900633872891e-08\\
0.0233742387089181 1.21859629354285e-08\\
0.0235909533050863 1.21817870904058e-08\\
0.0238096771738036 1.21775332865874e-08\\
0.0240304289440697 1.21731991069431e-08\\
0.0242532274176035 1.21687847094501e-08\\
0.0244780915704444 1.21642876017732e-08\\
0.0247050405545683 1.21597072643024e-08\\
0.0249340936995188 1.21550419653591e-08\\
0.0251652705140539 1.21502888734065e-08\\
0.0253985906878073 1.214544753684e-08\\
0.0256340740929651 1.21405156833937e-08\\
0.0258717407859592 1.21354932691557e-08\\
0.0261116110091746 1.21303751319215e-08\\
0.0263537051926739 1.21251622051263e-08\\
0.0265980439559376 1.21198526262224e-08\\
0.0268446481096196 1.21144438389328e-08\\
0.0270935386573205 1.21089340406222e-08\\
0.0273447367973758 1.21033220054296e-08\\
0.0275982639246618 1.20976041002371e-08\\
0.0278541416324177 1.20917811239196e-08\\
0.0281123917140846 1.2085848909642e-08\\
0.0283730361651621 1.20798069174693e-08\\
0.0286360971850812 1.20736511874685e-08\\
0.0289015971790951 1.20673807404021e-08\\
0.029169558760188 1.20609947179703e-08\\
0.0294400047510004 1.20544881179336e-08\\
0.0297129581857733 1.204786125886e-08\\
0.0299884423123103 1.20411101396372e-08\\
0.0302664805939569 1.20342318792171e-08\\
0.0305470967115997 1.20272272387041e-08\\
0.0308303145656828 1.2020091718982e-08\\
0.0311161582782436 1.20128230154513e-08\\
0.0314046521949675 1.20054194439805e-08\\
0.0316958208872612 1.19978778981705e-08\\
0.0319896891543454 1.19901941101269e-08\\
0.0322862820253673 1.19823681741678e-08\\
0.0325856247615323 1.19743957528985e-08\\
0.0328877428582551 1.1966275668199e-08\\
0.0331926620473319 1.19580033538091e-08\\
0.0335004082991313 1.19495770945637e-08\\
0.0338110078248069 1.19409939253847e-08\\
0.0341244870785289 1.19322505508893e-08\\
0.0344408727597382 1.19233436594937e-08\\
0.0347601918154198 1.1914272336701e-08\\
0.0350824714423979 1.19050291432385e-08\\
0.0354077390896527 1.18956152975983e-08\\
0.0357360224606579 1.18860251507883e-08\\
0.0360673495157403 1.1876256765412e-08\\
0.0364017484744614 1.18663061172861e-08\\
0.0367392478180207 1.18561704599181e-08\\
0.0370798762916817 1.18458454541026e-08\\
0.0374236629072198 1.18353274370088e-08\\
0.0377706369453937 1.18246141637178e-08\\
0.0381208279584389 1.1813700002124e-08\\
0.0384742657725851 1.18025830361645e-08\\
0.0388309804905961 1.17912595754403e-08\\
0.0391910024943341 1.17797227216713e-08\\
0.039554362447347 1.17679732245995e-08\\
0.0399210912974805 1.17560027449259e-08\\
0.0402912202795135 1.17438093350284e-08\\
0.0406647809178186 1.17313904226512e-08\\
0.0410418050290471 1.17187380976269e-08\\
0.041422324724839 1.17058500210942e-08\\
0.0418063724145576 1.16927223088934e-08\\
0.0421939808080503 1.16793492345969e-08\\
0.0425851829184341 1.16657276432938e-08\\
0.042980012064908 1.1651851438957e-08\\
0.0433785018755899 1.1637716316085e-08\\
0.0437806862903817 1.16233175196001e-08\\
0.0441865995638594 1.16086516143925e-08\\
0.0445962762681909 1.15937107743061e-08\\
0.0450097512960805 1.15784922910164e-08\\
0.0454270598637405 1.15629901756504e-08\\
0.045848237513891 1.15471980991186e-08\\
0.0462733201187869 1.15311125619305e-08\\
0.0467023438832734 1.15147271535251e-08\\
0.0471353453478691 1.14980361555979e-08\\
0.0475723613918789 1.1481033833872e-08\\
0.0480134292365345 1.14637154055389e-08\\
0.048458586448165 1.14460732234751e-08\\
0.0489078709413959 1.14281026643415e-08\\
0.0493613209823792 1.14097976015152e-08\\
0.0498189751920516 1.13911505692699e-08\\
0.0502808725494248 1.13721569763057e-08\\
0.0507470523949047 1.13528099707491e-08\\
0.0512175544336424 1.13331022510268e-08\\
0.051692418738916 1.131302674279e-08\\
0.0521716857555435 1.12925787573802e-08\\
0.0526553963033276 1.12717495098129e-08\\
0.0531435915805324 1.12505312583577e-08\\
0.0536363131673924 1.12289198401435e-08\\
0.0541336030296538 1.12069050231543e-08\\
0.0546355035221488 1.11844811989492e-08\\
0.055142057392403 1.11616396992042e-08\\
0.0556533077842765 1.11383733497216e-08\\
0.0561692982416381 1.11146741927237e-08\\
0.0566900727120743 1.10905353644781e-08\\
0.0572156755506325 1.10659474313298e-08\\
0.0577461515235982 1.104090237457e-08\\
0.0582815458123085 1.10153912384992e-08\\
0.0588219040169996 1.0989407194592e-08\\
0.0593672721606913 1.09629403631903e-08\\
0.0599176966931062 1.09359822051344e-08\\
0.0604732244946265 1.09085232851572e-08\\
0.0610339028802862 1.08805558193543e-08\\
0.0615997796038017 1.08520698408366e-08\\
0.0621709028616383 1.08230548993306e-08\\
0.0627473212971158 1.07935038329432e-08\\
0.0633290840045511 1.0763404696717e-08\\
0.0639162405334401 1.07327497148743e-08\\
0.0645088408926769 1.07015263417619e-08\\
0.0651069355548146 1.0669727357226e-08\\
0.0657105754603627 1.06373396100632e-08\\
0.0663198120221268 1.06043551471381e-08\\
0.0669346971295866 1.05707627984821e-08\\
0.0675552831533165 1.05365501506957e-08\\
0.0681816229494448 1.05017079475794e-08\\
0.0688137698641567 1.04662256944486e-08\\
0.069451777738237 1.04300916976756e-08\\
0.0700957009116562 1.03932937203759e-08\\
0.0707455942281988 1.03558217980707e-08\\
0.071401513040134 1.03176631019126e-08\\
0.0720635132129305 1.02788076776764e-08\\
0.0727316511300146 1.02392426184162e-08\\
0.0734059836975721 1.01989566377221e-08\\
0.0740865683493956 1.01579375796484e-08\\
0.0747734630517759 1.01161749828702e-08\\
0.0754667263084389 1.00736548016512e-08\\
0.0761664171655289 1.00303654604473e-08\\
0.0768725952166373 9.98629601201135e-09\\
0.0775853206078784 9.9414341750453e-09\\
0.0783046540430119 9.89576761269235e-09\\
0.0790306567886135 9.84928358522691e-09\\
0.0797633906792928 9.8019714328445e-09\\
0.0805029181229598 9.75381941926186e-09\\
0.0812493021061405 9.70481589271323e-09\\
0.0820026061993413 9.65494993928943e-09\\
0.0827628945624635 9.60420968886686e-09\\
0.0835302319502678 9.55258389643415e-09\\
0.0843046837178897 9.50006382919178e-09\\
0.0850863158264058 9.446637393172e-09\\
0.0858751948484517 9.3922986542775e-09\\
0.0866713879738923 9.33703475432047e-09\\
0.0874749630155442 9.28084130417045e-09\\
0.0882859884149515 9.22370828382002e-09\\
0.0891045332482151 9.16563203807557e-09\\
0.0899306672318762 9.1066066170192e-09\\
0.0907644607288536 9.04662754449897e-09\\
0.0916059847544371 8.9856919576554e-09\\
0.0924553109823357 8.92380194952463e-09\\
0.0933125117507825 8.86095849745071e-09\\
0.0941776600686952 8.79716233339944e-09\\
0.0950508296218951 8.73242467008049e-09\\
0.0959320947793824 8.66675264863568e-09\\
0.0968215305996709 8.60016087273107e-09\\
0.09771921283718 8.5326620974588e-09\\
0.0986252179486878 8.46428334463692e-09\\
0.0995396230998423 8.39504666560746e-09\\
0.100462506171734 8.32498654632635e-09\\
0.101393945767529 8.25414250242921e-09\\
0.102334021219164 8.18256110080981e-09\\
0.103282812594103 8.11029791498443e-09\\
0.104240400702156 8.0374221404689e-09\\
0.105206867102362 7.96400769036121e-09\\
0.106182294109938 7.89014930283832e-09\\
0.107166764803286 7.8159539416132e-09\\
0.108160363031071 7.74154635732779e-09\\
0.109163173419361 7.66707474017769e-09\\
0.110175281378839 7.59270798447464e-09\\
0.11119677311207 7.51864673988335e-09\\
0.112227735620851 7.44512265671322e-09\\
0.113268256713615 7.37240344937463e-09\\
0.114318425012915 7.30080300588471e-09\\
0.115378329962966 7.23068461548816e-09\\
0.116448061837269 7.16247503394291e-09\\
0.117527711746295 7.09666230495128e-09\\
0.118617371645248 7.03382281729472e-09\\
0.119717134341897 6.97462278002173e-09\\
0.120827093504478 6.91983808842796e-09\\
0.121947343669674 6.87037300355509e-09\\
0.123077980250667 6.82727773404113e-09\\
0.124219099545262 6.79178059207181e-09\\
0.125370798744092 6.76530021089615e-09\\
0.126533175938894 6.74949494268732e-09\\
0.127706330130864 6.74628853487016e-09\\
0.128890361239089 6.75791584792432e-09\\
0.130085370109057 6.78696923569317e-09\\
0.131291458521247 6.83645566512388e-09\\
0.132508729199795 6.90986038370034e-09\\
0.13373728582125 7.01122730512145e-09\\
0.134977233023394 7.14525774246319e-09\\
0.136228676414165 7.31742394966434e-09\\
0.137491722580642 7.53412725480647e-09\\
0.138766479098131 7.80292891095768e-09\\
0.140053054539322 8.13281198107981e-09\\
0.14135155848354 8.53461177261317e-09\\
0.142662101526074 9.02156611909106e-09\\
0.143984795287601 9.61013944353632e-09\\
0.14531975242369 1.03211578906323e-08\\
0.146667086634397 1.11815001494017e-08\\
0.148026912673951 1.22265235999551e-08\\
0.149399346360526 1.35036928949727e-08\\
0.150784504586105 1.50780882927712e-08\\
0.152182505326439 1.7040919695275e-08\\
0.15359346765109 1.95230778385794e-08\\
0.155017511733577 2.27173086816114e-08\\
0.15645475886161 2.69156615801212e-08\\
0.157905331447418 3.25749745295944e-08\\
0.159369353038178 4.04359146745737e-08\\
0.160846948326536 5.17485790236218e-08\\
0.162338243161228 6.87198146434613e-08\\
0.163843364557798 9.54424138604906e-08\\
0.165362440709418 1.39910962684962e-07\\
0.166895600997802 2.18520505449826e-07\\
0.168442976004232 3.65899107156375e-07\\
0.170004697520672 6.52765075750613e-07\\
0.171580898561 1.1941453556867e-06\\
0.173171713372335 2.04935220899517e-06\\
0.174777277446468 3.03573486311335e-06\\
0.176397727531404 4.20257190469812e-06\\
0.178033201643011 5.86295143458157e-06\\
0.179683839076772 7.54463139604156e-06\\
0.18134978041965 8.75434279678955e-06\\
0.183031167562061 9.53354187864585e-06\\
0.184728143709963 1.1357590277288e-05\\
0.186440853397049 1.88509552438335e-05\\
0.188169442497056 3.65522151255775e-05\\
0.189914058236193 6.22521560772197e-05\\
0.191674849205682 8.3799630975031e-05\\
0.193451965374405 9.29412972219259e-05\\
0.195245558101686 9.42665376717641e-05\\
0.19705578015018 9.33544629135805e-05\\
0.198882785698881 9.15137058544672e-05\\
0.200726730356257 8.87365925488168e-05\\
0.202587771173502 8.87717325496663e-05\\
0.204466066657912 9.84191286862559e-05\\
0.206361776786386 0.000119467513779316\\
0.208275063019052 0.00014964142648586\\
0.210206088313016 0.000189205048543351\\
0.212155017136245 0.000241482060739003\\
0.214122015481573 0.000315052969054412\\
0.216107250880838 0.000428581657120756\\
0.218110892419152 0.000612230962802619\\
0.220133110749303 0.000901028744980933\\
0.222174078106288 0.00144541521595798\\
0.224233968321985 0.00350768718094614\\
0.226312956839953 0.00203772249978607\\
0.228411220730382 0.000859159193274952\\
0.230528938705171 0.000534134697462187\\
0.232666291133146 0.000397068819687507\\
0.234823460055428 0.000323520401445076\\
0.237000629200933 0.000275682170295263\\
0.239197984002024 0.000239850500009968\\
0.241415711610302 0.000211902102662088\\
0.243654000912547 0.000192911048752319\\
0.245913042546805 0.000184309668227804\\
0.248193028918626 0.000183452975780432\\
0.25049415421745 0.000184061461466935\\
0.25281661443315 0.00018064610032568\\
0.255160607372719 0.000172443040529822\\
0.25752633267712 0.000164185834977121\\
0.259913991838293 0.000162038387071576\\
0.262323788216312 0.00016913902363831\\
0.264755927056707 0.000186740315241137\\
0.267210615507944 0.00021506885004529\\
0.269688062639069 0.000251836055184089\\
0.272188479457518 0.000290637547564682\\
0.274712078927081 0.000321412934139814\\
0.277259075986048 0.000337680568617054\\
0.279829687565512 0.000341539252631659\\
0.282424132607843 0.000339931940372544\\
0.285042632085343 0.000339621875039479\\
0.287685409019059 0.000344370206258634\\
0.290352688497781 0.000354310571408024\\
0.293044697697214 0.000366575778613023\\
0.295761665899325 0.000375237521976555\\
0.298503824511873 0.000372131273954251\\
0.301271407088116 0.000352445962803146\\
0.304064649346707 0.000319742317170622\\
0.306883789191764 0.000282143082333581\\
0.309729066733141 0.000245566170512856\\
0.31260072430687 0.000212078208314942\\
0.315499006495809 0.000181894040426435\\
0.318424160150465 0.000155254542145415\\
0.321376434410028 0.000133174505837253\\
0.324356080723581 0.000117374513113435\\
0.327363352871525 0.000109494199430063\\
0.330398506987186 0.000109992298692517\\
0.333461801578636 0.000118106024684128\\
0.336553497550709 0.000133016551136491\\
0.339673858227221 0.000154887210612593\\
0.342823149373397 0.00018545385591998\\
0.346001639218511 0.000228009710016919\\
0.349209598478727 0.000282894071745268\\
0.352447300380159 0.000321386765854465\\
0.355715020682139 0.000280005897326686\\
0.359013037700707 0.000197339730481793\\
0.362341632332315 0.000133797886134732\\
0.365701088077749 9.27709554184427e-05\\
0.369091691066278 6.64365301868252e-05\\
0.372513730080021 4.91229512843631e-05\\
0.375967496578547 3.72999941395365e-05\\
0.379453284723694 2.88186192307654e-05\\
0.382971391404628 2.24457159964995e-05\\
0.386522116263126 1.7515722839413e-05\\
0.390105761719099 1.36646534628781e-05\\
0.393722632996348 1.06658603473667e-05\\
0.397373038148561 8.35100659870203e-06\\
0.40105728808555 6.58009734057693e-06\\
0.404775696599732 5.23352945746685e-06\\
0.408528580392856 4.21151284872083e-06\\
0.412316259102975 3.434006746138e-06\\
0.416139055331671 2.83910698441357e-06\\
0.419997294671531 2.3802027850657e-06\\
0.423891305733878 2.02277899723691e-06\\
0.427821420176762 1.74148749133077e-06\\
0.431787972733202 1.51774914897501e-06\\
0.435791301239703 1.33791302513154e-06\\
0.439831746665023 1.1918949166905e-06\\
0.443909653139217 1.07219102670018e-06\\
0.448025367982949 9.73170515725612e-07\\
0.45217924173707 8.90569724549821e-07\\
0.456371628192476 8.21130145823122e-07\\
0.46060288442024 7.62338228504407e-07\\
0.464873370802026 7.12237266409118e-07\\
0.46918345106078 6.6929043218325e-07\\
0.473533492291712 6.32280275864211e-07\\
0.47792386499356 6.00234379588667e-07\\
0.482354943100147 5.72369908269663e-07\\
0.486827104012228 5.48051901317031e-07\\
0.491340728629636 5.26761660295188e-07\\
0.495896201383721 5.08072578291954e-07\\
0.500493910270095 4.91631537513866e-07\\
0.505134246881676 4.77144459410847e-07\\
0.509817606442042 4.64365022488676e-07\\
0.514544387839093 4.53085778578178e-07\\
0.519314993659021 4.43131112580407e-07\\
0.524129830220606 4.34351630278131e-07\\
0.528989307609815 4.2661965697496e-07\\
0.533893839714735 4.19825602707313e-07\\
0.538843844260822 4.13875007500399e-07\\
0.54383974284648 4.08686142419203e-07\\
0.548881960978967 4.04188028538741e-07\\
0.55397092811064 4.003188125405e-07\\
0.559107077675529 3.970244135764e-07\\
0.564290847126254 3.94257407080265e-07\\
0.569522677971283 3.91976082462673e-07\\
0.574803015812535 3.90143657720173e-07\\
0.580132310383338 3.88727614761031e-07\\
0.585511015586724 3.8769914245715e-07\\
0.590939589534097 3.87032650661558e-07\\
0.596418494584246 3.86705377430922e-07\\
0.601948197382727 3.86697024298549e-07\\
0.607529168901607 3.86989481327578e-07\\
0.613161884479578 3.8756654761303e-07\\
0.618846823862439 3.88413722869188e-07\\
0.624584471243962 3.89518010466926e-07\\
0.630375315307128 3.90867758007852e-07\\
0.636219849265749 3.9245250094532e-07\\
0.642118570906476 3.94262849961972e-07\\
0.648071982631197 3.96290361514864e-07\\
0.654080591499826 3.98527464773156e-07\\
0.660144909273489 4.00967350129229e-07\\
0.666265452458115 4.03603915191747e-07\\
0.672442742348425 4.06431685161132e-07\\
0.67867730507233 4.09445762372487e-07\\
0.684969671635745 4.12641763358732e-07\\
0.691320377967813 4.16015784312862e-07\\
0.697729964966554 4.19564350047426e-07\\
0.704198978544929 4.23284381784902e-07\\
0.710727969677342 4.27173167629152e-07\\
0.717317494446562 4.31228321425128e-07\\
0.723968114091088 4.35447766087895e-07\\
0.73068039505295 4.39829712345049e-07\\
0.737454909025955 4.44372622533453e-07\\
0.744292233004376 4.49075207709882e-07\\
0.751192949332097 4.53936398572841e-07\\
0.758157645752211 4.58955334701871e-07\\
0.765186915457083 4.64131353661113e-07\\
0.772281357138865 4.69463972236702e-07\\
0.779441575040495 4.74952869134311e-07\\
0.786668179007158 4.80597892354599e-07\\
0.793961784538228 4.86399035415888e-07\\
0.801323012839689 4.92356427737108e-07\\
0.808752490877045 4.98470332347446e-07\\
0.816250851428723 5.04741140017696e-07\\
0.823818733139961 5.11169356699092e-07\\
0.831456780577206 5.17755599468518e-07\\
0.839165644283016 5.24500589222294e-07\\
0.846945980831459 5.31405151458401e-07\\
0.854798452884041 5.3847020816428e-07\\
0.862723729246145 5.45696773447656e-07\\
0.870722484923992 5.53085935455833e-07\\
0.878795401182132 5.60638885946525e-07\\
0.886943165601471 5.68356888099536e-07\\
0.89516647213783 5.76241272751293e-07\\
0.903466021181053 5.84293469487928e-07\\
0.911842519614657 5.92514946580404e-07\\
0.920296680876041 6.00907273467023e-07\\
0.92882922501725 6.09472066494471e-07\\
0.9374408787663 6.18211004516087e-07\\
0.946132375589077 6.27125847224653e-07\\
0.954904455751808 6.36218396406855e-07\\
0.963757866384109 6.45490518525706e-07\\
0.972693361542617 6.54944154248715e-07\\
0.981711702275219 6.6458125999592e-07\\
0.990813656685867 6.74403894657708e-07\\
1 6.84414126845654e-07\\
};
\end{axis}
\end{tikzpicture}%

%% file: figures/tf_tchain_30.tex
%
%
%
%
\begin{tikzpicture}

\begin{axis}[%
width=\figwidth,
height=\figheight,
scale only axis,
separate axis lines,
every outer x axis line/.append style={darkgray!60!black},
every x tick label/.append style={font=\color{darkgray!60!black}},
xmode=log,
xmin=0.001,
xmax=0.1,
xminorticks=true,
xlabel={$\omega$},
every outer y axis line/.append style={darkgray!60!black},
every y tick label/.append style={font=\color{darkgray!60!black}},
ymode=log,
ymin=10000,
ymax=100000000,
yminorticks=true,
ylabel={$\sigma{}_{\text{max}}\text{($\mathcal{T}$(j}\omega\text{))}$},
legend columns=2,
legend style={nodes=right,anchor=south east,at={(.7,1.05)}}
]
\addplot [
color=black,
solid,
]
table[row sep=crcr]{
0.001 4224398.33891604\\
0.00100927151463057 4225190.95760389\\
0.00101862899024469 4225998.64985668\\
0.00102807322383086 4226821.70789528\\
0.00103760501976691 4227660.4303751\\
0.00104722518988843 4228515.12176271\\
0.00105693455355799 4229386.09256175\\
0.00106673393773486 4230273.65998867\\
0.00107662417704549 4231178.14717691\\
0.0010866061138546 4232099.88409268\\
0.00109668059833687 4233039.2075908\\
0.00110684848854941 4233996.46076803\\
0.00111711065050482 4234971.99467354\\
0.00112746795824495 4235966.16686676\\
0.00113792129391532 4236979.34217323\\
0.00114847154784029 4238011.8934382\\
0.00115911961859889 4239064.20052162\\
0.00116986641310131 4240136.65168449\\
0.00118071284666619 4241229.6429445\\
0.00119165984309856 4242343.57828446\\
0.00120270833476851 4243478.87041984\\
0.00121385926269063 4244635.94053518\\
0.00122511357660412 4245815.21844049\\
0.00123647223505371 4247017.14284492\\
0.00124793620547131 4248242.16188255\\
0.00125950646425836 4249490.73265634\\
0.00127118399686903 4250763.32241578\\
0.00128296979789415 4252060.40759745\\
0.0012948648711459 4253382.47502983\\
0.00130687022974335 4254730.02162044\\
0.00131898689619867 4256103.5549979\\
0.00133121590250431 4257503.59326279\\
0.00134355829022083 4258930.66568243\\
0.00135601511056563 4260385.31256722\\
0.00136858742450252 4261868.08592251\\
0.00138127630283201 4263379.54946708\\
0.00139408282628258 4264920.27909266\\
0.00140700808560269 4266490.86295834\\
0.00142005318165368 4268091.90168073\\
0.00143321922550357 4269724.00914823\\
0.00144650733852165 4271387.81223229\\
0.00145991865247398 4273083.95178802\\
0.00147345430961984 4274813.082086\\
0.00148711546280895 4276575.8718502\\
0.00150090327557974 4278373.00468408\\
0.00151481892225835 4280205.17844963\\
0.00152886358805873 4282073.10722407\\
0.00154303846918356 4283977.51982586\\
0.00155734477292614 4285919.16208308\\
0.00157178371777316 4287898.7958634\\
0.00158635653350859 4289917.19996188\\
0.00160106446131832 4291975.17051559\\
0.00161590875389592 4294073.52133953\\
0.00163089067554933 4296213.0847696\\
0.00164601150230855 4298394.71163416\\
0.00166127252203429 4300619.27198736\\
0.0016766750345277 4302887.65532635\\
0.00169222035164104 4305200.77178578\\
0.00170790979738943 4307559.55166022\\
0.00172374470806362 4309964.94700775\\
0.0017397264323438 4312417.93131942\\
0.00175585633141447 4314919.50041459\\
0.00177213577908036 4317470.67360314\\
0.00178856616188346 4320072.49283183\\
0.00180514887922111 4322726.02512957\\
0.00182188534346517 4325432.36216385\\
0.00183877698008233 4328192.62063864\\
0.00185582522775552 4331007.9442826\\
0.00187303153850644 4333879.50292636\\
0.00189039737781922 4336808.49479791\\
0.00190792422476527 4339796.14588268\\
0.0019256135721292 4342843.71224382\\
0.00194346692653602 4345952.47886925\\
0.00196148580857943 4349123.76270769\\
0.00197967175295133 4352358.91176623\\
0.00199802630857255 4355659.30697301\\
0.00201655103872475 4359026.36292511\\
0.00203524752118358 4362461.52844923\\
0.00205411734835306 4365966.28799841\\
0.00207316212740123 4369542.16265766\\
0.00209238348039698 4373190.71139249\\
0.00211178304444824 4376913.53149075\\
0.00213136247184144 4380712.26049172\\
0.00215112343018217 4384588.57638554\\
0.00217106760253726 4388544.20010628\\
0.00219119668757815 4392580.89595404\\
0.00221151239972549 4396700.47261348\\
0.00223201646929523 4400904.78579329\\
0.00225271064264598 4405195.73824304\\
0.00227359668232772 4409575.28201334\\
0.00229467636723194 4414045.419842\\
0.00231595149274315 4418608.2063353\\
0.00233742387089181 4423265.75041887\\
0.00235909533050864 4428020.21611909\\
0.00238096771738036 4432873.82520324\\
0.00240304289440697 4437828.85806763\\
0.00242532274176035 4442887.65615399\\
0.00244780915704444 4448052.62447504\\
0.00247050405545683 4453326.23255205\\
0.00249340936995188 4458711.01704247\\
0.00251652705140539 4464209.58439351\\
0.00253985906878073 4469824.6124166\\
0.00256340740929651 4475558.8531219\\
0.00258717407859592 4481415.13534226\\
0.00261116110091746 4487396.36640958\\
0.00263537051926739 4493505.53647337\\
0.00265980439559376 4499745.7198848\\
0.00268446481096197 4506120.07848715\\
0.00270935386573205 4512631.86534032\\
0.00273447367973758 4519284.42704013\\
0.00275982639246618 4526081.20751259\\
0.00278541416324177 4533025.7511607\\
0.00281123917140846 4540121.70715349\\
0.00283730361651621 4547372.8323721\\
0.00286360971850812 4554782.99595188\\
0.00289015971790951 4562356.18300889\\
0.0029169558760188 4570096.49932964\\
0.00294400047510003 4578008.17518065\\
0.00297129581857733 4586095.57093009\\
0.00299884423123103 4594363.18161283\\
0.00302664805939569 4602815.64157607\\
0.00305470967115997 4611457.73019056\\
0.00308303145656828 4620294.37794229\\
0.00311161582782436 4629330.67183165\\
0.00314046521949675 4638571.86142721\\
0.00316958208872612 4648023.36538715\\
0.00319896891543454 4657690.7793271\\
0.00322862820253673 4667579.88097562\\
0.00325856247615323 4677696.63917312\\
0.00328877428582551 4688047.22139161\\
0.00331926620473319 4698638.00138239\\
0.00335004082991313 4709475.56859495\\
0.00338110078248069 4720566.73688571\\
0.00341244870785289 4731918.55364497\\
0.00344408727597382 4743538.31104792\\
0.00347601918154198 4755433.55522888\\
0.00350824714423979 4767612.09817424\\
0.00354077390896527 4780082.02977226\\
0.00357360224606579 4792851.72886548\\
0.00360673495157403 4805929.8776522\\
0.00364017484744614 4819325.47484334\\
0.00367392478180207 4833047.84967952\\
0.00370798762916817 4847106.67823967\\
0.00374236629072198 4861511.99860413\\
0.00377706369453937 4876274.22838976\\
0.00381208279584389 4891404.18269731\\
0.00384742657725851 4906913.09309763\\
0.00388309804905961 4922812.62795841\\
0.00391910024943341 4939114.91326843\\
0.0039554362447347 4955832.55665682\\
0.00399210912974805 4972978.66919828\\
0.00402912202795135 4990566.89313602\\
0.00406647809178186 5008611.42734102\\
0.00410418050290471 5027127.05617119\\
0.0041422324724839 5046129.18078018\\
0.00418063724145576 5065633.85032206\\
0.00421939808080502 5085657.7975808\\
0.00425851829184341 5106218.47461814\\
0.0042980012064908 5127334.09240587\\
0.00433785018755899 5149023.66198445\\
0.00437806862903817 5171307.03993112\\
0.00441865995638594 5194204.97367535\\
0.0044596276268191 5217739.15461787\\
0.00450097512960805 5241932.27050934\\
0.00454270598637404 5266808.06317627\\
0.0045848237513891 5292391.39246771\\
0.00462733201187869 5318708.29965314\\
0.00467023438832734 5345786.08084017\\
0.00471353453478691 5373653.36277294\\
0.00475723613918789 5402340.18380576\\
0.00480134292365346 5431878.08271557\\
0.0048458586448165 5462300.19444695\\
0.00489078709413959 5493641.35001835\\
0.00493613209823792 5525938.18743613\\
0.00498189751920516 5559229.27104457\\
0.00502808725494248 5593555.21841651\\
0.00507470523949047 5628958.83845283\\
0.00512175544336424 5665485.28234445\\
0.0051692418738916 5703182.20406215\\
0.00521716857555435 5742099.93687984\\
0.00526553963033276 5782291.68357633\\
0.00531435915805324 5823813.72363201\\
0.00536363131673924 5866725.63797863\\
0.00541336030296538 5911090.55460036\\
0.00546355035221488 5956975.41503441\\
0.0055142057392403 6004451.26644201\\
0.00556533077842765 6053593.58072044\\
0.00561692982416381 6104482.6013549\\
0.00566900727120743 6157203.72777832\\
0.00572156755506325 6211847.93343478\\
0.00577461515235982 6268512.22527809\\
0.00582815458123085 6327300.15274622\\
0.00588219040169996 6388322.36426535\\
0.00593672721606913 6451697.2238738\\
0.00599176966931062 6517551.49228649\\
0.00604732244946265 6586021.08183322\\
0.00610339028802862 6657251.89189448\\
0.00615997796038017 6731400.73816178\\
0.00621709028616383 6808636.38681356\\
0.00627473212971158 6889140.70771903\\
0.00633290840045512 6973109.96165339\\
0.00639162405334401 7060756.24281809\\
0.00645088408926769 7152309.09518982\\
0.00651069355548146 7248017.32970293\\
0.00657105754603627 7348151.07222394\\
0.00663198120221268 7453004.07348228\\
0.00669346971295866 7562896.32465245\\
0.00675552831533165 7678177.02263242\\
0.00681816229494448 7799227.94319432\\
0.00688137698641567 7926467.28547196\\
0.0069451777738237 8060354.06757645\\
0.00700957009116562 8201393.15998825\\
0.00707455942281988 8350141.08080301\\
0.0071401513040134 8507212.66532753\\
0.00720635132129305 8673288.78680276\\
0.00727316511300145 8849125.304917\\
0.00734059836975721 9035563.4877375\\
0.00740865683493957 9233542.17016056\\
0.00747734630517759 9444112.00304943\\
0.00754667263084389 9668452.2041877\\
0.00761664171655289 9907890.3177821\\
0.00768725952166374 10163925.6152432\\
0.00775853206078784 10438256.8957895\\
0.00783046540430118 10732815.6464353\\
0.00790306567886135 11049805.7243646\\
0.00797633906792928 11391751.00879\\
0.00805029181229598 11761552.808824\\
0.00812493021061405 12162559.1893784\\
0.00820026061993413 12598648.893894\\
0.00827628945624635 13074333.0154055\\
0.00835302319502678 13594878.1409338\\
0.00843046837178897 14166455.0988584\\
0.00850863158264058 14796317.5074671\\
0.00858751948484518 15493013.5315584\\
0.00866713879738923 16266631.6065477\\
0.00874749630155442 17129074.2954735\\
0.00882859884149515 18094340.2745748\\
0.00891045332482151 19178764.9741052\\
0.00899306672318762 20401112.0177098\\
0.00907644607288536 21782294.5369569\\
0.00916059847544371 23344292.7151753\\
0.00924553109823357 25107452.3336295\\
0.00933125117507825 27084717.6458116\\
0.00941776600686952 29270482.9464814\\
0.00950508296218951 31621116.1347057\\
0.00959320947793824 34025689.41786\\
0.00968215305996709 36273541.9786678\\
0.009771921283718 38044787.3390772\\
0.00986252179486878 38970152.7765157\\
0.00995396230998423 38780441.542597\\
0.0100462506171734 37462161.676186\\
0.0101393945767529 35273421.8921926\\
0.0102334021219164 32600061.537698\\
0.0103282812594103 29790000.0159675\\
0.0104240400702156 27075296.7228099\\
0.0105206867102362 24576034.3165332\\
0.0106182294109938 22335583.3042033\\
0.0107166764803286 20354373.7213014\\
0.0108160363031071 18612629.2820967\\
0.0109163173419361 17083255.5361926\\
0.0110175281378839 15738331.218831\\
0.011119677311207 14551997.4917701\\
0.0112227735620851 13501489.341234\\
0.0113268256713615 12567278.8953299\\
0.0114318425012915 11732834.8215371\\
0.0115378329962966 10984246.795478\\
0.0116448061837269 10309831.1791655\\
0.0117527711746295 9699767.26230952\\
0.0118617371645248 9145781.09851239\\
0.0119717134341897 8640879.10281132\\
0.0120827093504478 8179127.27277498\\
0.0121947343669674 7755469.82389724\\
0.0123077980250667 7365580.79524766\\
0.0124219099545262 7005742.74989423\\
0.0125370798744092 6672747.55149584\\
0.0126533175938894 6363815.01708872\\
0.0127706330130864 6076526.04834779\\
0.0128890361239089 5808767.46620897\\
0.0130085370109057 5558686.34147833\\
0.0131291458521247 5324652.03677903\\
0.0132508729199795 5105224.53348269\\
0.013373728582125 4899127.89892962\\
0.0134977233023394 4705227.97133657\\
0.0136228676414165 4522513.51917326\\
0.0137491722580642 4350080.27285198\\
0.0138766479098131 4187117.34221784\\
0.0140053054539322 4032895.61801069\\
0.014135155848354 3886757.83783001\\
0.0142662101526074 3748110.04326116\\
0.0143984795287601 3616414.2140074\\
0.014531975242369 3491181.89282088\\
0.0146667086634397 3371968.65298538\\
0.0148026912673951 3258369.28348222\\
0.0149399346360526 3150013.58326724\\
0.0150784504586105 3046562.68179077\\
0.0152182505326439 2947705.80771401\\
0.015359346765109 2853157.44562563\\
0.0155017511733577 2762654.82765742\\
0.015645475886161 2675955.71500391\\
0.0157905331447418 2592836.43072275\\
0.0159369353038178 2513090.11241094\\
0.0160846948326536 2436525.1558062\\
0.0162338243161228 2362963.82589142\\
0.0163843364557798 2292241.0146544\\
0.0165362440709418 2224203.12806983\\
0.0166895600997802 2158707.08618184\\
0.0168442976004231 2095619.42386887\\
0.0170004697520672 2034815.47965176\\
0.0171580898561 1976178.66337745\\
0.0173171713372335 1919599.79285195\\
0.0174777277446468 1864976.49294548\\
0.0176397727531404 1812212.6489937\\
0.0178033201643011 1761217.90934254\\
0.0179683839076772 1711907.23152131\\
0.018134978041965 1664200.46713249\\
0.0183031167562061 1618021.98144066\\
0.0184728143709963 1573300.30396773\\
0.0186440853397049 1529967.80650971\\
0.0188169442497056 1487960.4056964\\
0.0189914058236193 1447217.28731837\\
0.0191674849205681 1407680.65010587\\
0.0193451965374405 1369295.46604655\\
0.0195245558101686 1332009.25604328\\
0.019705578015018 1295771.87807912\\
0.0198882785698881 1260535.32626584\\
0.0200726730356257 1226253.53920625\\
0.0202587771173502 1192882.21536807\\
0.0204466066657912 1160378.63409819\\
0.0206361776786386 1128701.48004005\\
0.0208275063019052 1097810.66936014\\
0.0210206088313016 1067667.17540438\\
0.0212155017136245 1038232.85184713\\
0.0214122015481573 1009470.25057017\\
0.0216107250880838 981342.431659382\\
0.0218110892419152 953812.762165259\\
0.0220133110749303 926844.69983739\\
0.0222174078106288 900401.557488467\\
0.0224233968321985 874446.242676467\\
0.0226312956839953 848940.96624127\\
0.0228411220730382 823846.912220175\\
0.0230528938705171 799123.859657164\\
0.0232666291133146 774729.745209117\\
0.0234823460055428 750620.152638006\\
0.0237000629200933 726747.712876714\\
0.0239197984002024 703061.394735044\\
0.0241415711610302 679505.663120929\\
0.0243654000912547 656019.478901093\\
0.0245913042546805 632535.113211511\\
0.0248193028918626 608976.752928952\\
0.025049415421745 585258.88915749\\
0.025281661443315 561284.519405542\\
0.0255160607372719 536943.286191069\\
0.025752633267712 512109.878139914\\
0.0259913991838293 486643.466548759\\
0.0262323788216312 460389.935062102\\
0.0264755927056707 433190.866359034\\
0.0267210615507944 404908.318738145\\
0.0269688062639069 375486.417678632\\
0.0272188479457518 345099.912461935\\
0.0274712078927081 314510.894073563\\
0.0277259075986048 285915.478474273\\
0.0279829687565512 264797.808914884\\
0.0282424132607843 262844.768904159\\
0.0285042632085343 298333.787863894\\
0.0287685409019059 389015.237526715\\
0.0290352688497781 546897.421696171\\
0.0293044697697214 775423.511108828\\
0.0295761665899325 1038395.80563239\\
0.0298503824511873 1221010.45508847\\
0.0301271407088116 1239247.63447613\\
0.0304064649346707 1154066.87762261\\
0.0306883789191764 1047056.77825873\\
0.0309729066733141 950743.763425975\\
0.031260072430687 870827.105074738\\
0.0315499006495809 805363.325729347\\
0.0318424160150465 751260.320920956\\
0.0321376434410028 705853.678149258\\
0.0324356080723581 667113.689489908\\
0.0327363352871525 633543.060858999\\
0.0330398506987186 604040.217812669\\
0.0333461801578636 577787.829330989\\
0.0336553497550709 554171.685829322\\
0.0339673858227221 532723.654598334\\
0.0342823149373397 513081.827774084\\
0.0346001639218511 494962.542607397\\
0.0349209598478727 478140.539443589\\
0.0352447300380159 462434.718255199\\
0.0355715020682139 447697.781291061\\
0.0359013037700707 433808.603090133\\
0.0362341632332315 420666.53781678\\
0.0365701088077749 408187.119187839\\
0.0369091691066278 396298.773154392\\
0.0372513730080021 384940.275104365\\
0.0375967496578547 374058.759787107\\
0.0379453284723694 363608.145046241\\
0.0382971391404628 353547.86741471\\
0.0386522116263126 343841.853612618\\
0.0390105761719099 334457.670482132\\
0.0393722632996348 325365.808902274\\
0.0397373038148561 316539.066432094\\
0.040105728808555 307951.999736201\\
0.0404775696599732 299580.422110149\\
0.0408528580392856 291400.92319259\\
0.0412316259102975 283390.388672975\\
0.0416139055331671 275525.496957493\\
0.0419997294671531 267782.166880487\\
0.0423891305733878 260134.926105352\\
0.0427821420176762 252556.163840263\\
0.0431787972733202 245015.224405316\\
0.0435791301239703 237477.292463883\\
0.0439831746665023 229902.023409971\\
0.0443909653139217 222241.902464891\\
0.0448025367982949 214440.424594724\\
0.045217924173707 206430.510639448\\
0.0456371628192476 198134.486250467\\
0.046060288442024 189469.501561654\\
0.0464873370802026 180369.541097839\\
0.046918345106078 170856.462201458\\
0.0473533492291712 161255.775946964\\
0.047792386499356 152831.036508636\\
0.0482354943100147 149453.283412624\\
0.0486827104012228 160243.156795488\\
0.0491340728629636 195751.989232626\\
0.0495896201383721 245315.348422098\\
0.0500493910270095 273966.913575412\\
0.0505134246881676 273533.004315669\\
0.0509817606442042 260752.523987096\\
0.0514544387839093 246208.5419742\\
0.0519314993659021 233011.923207845\\
0.0524129830220606 221597.763003637\\
0.0528989307609815 211740.701071939\\
0.0533893839714735 203126.910596146\\
0.0538843844260822 195486.851515906\\
0.054383974284648 188612.132588354\\
0.0548881960978967 182345.975971665\\
0.055397092811064 176570.551457245\\
0.0559107077675529 171196.443933507\\
0.0564290847126254 166154.816552303\\
0.0569522677971283 161391.773746197\\
0.0574803015812535 156864.329613559\\
0.0580132310383338 152537.502624537\\
0.0585511015586724 148382.188822152\\
0.0590939589534097 144373.569021907\\
0.0596418494584246 140489.878520616\\
0.0601948197382727 136711.417348175\\
0.0607529168901607 133019.712201015\\
0.0613161884479577 129396.763643802\\
0.0618846823862439 125824.331184019\\
0.0624584471243962 122283.233292731\\
0.0630375315307128 118752.688638069\\
0.0636219849265749 115209.847429805\\
0.0642118570906476 111629.984586238\\
0.0648071982631197 107988.697733491\\
0.0654080591499826 104269.859391003\\
0.0660144909273489 100489.851609954\\
0.0666265452458115 96767.6205049836\\
0.0672442742348425 93518.2420685302\\
0.067867730507233 91913.3825286951\\
0.0684969671635745 94411.2936729694\\
0.0691320377967813 102879.210751101\\
0.0697729964966554 112657.924021924\\
0.0704198978544929 116619.632298482\\
0.0710727969677342 115143.287287458\\
0.0717317494446562 111505.08434261\\
0.0723968114091088 107477.119757276\\
0.073068039505295 103638.737504179\\
0.0737454909025955 100114.046321105\\
0.0744292233004376 96890.2399822304\\
0.0751192949332097 93922.6868844274\\
0.0758157645752211 91165.891391509\\
0.0765186915457083 88581.0605276674\\
0.0772281357138865 86136.7055925632\\
0.0779441575040495 83807.4362324659\\
0.0786668179007158 81572.5092266743\\
0.0793961784538228 79414.5552015976\\
0.0801323012839689 77318.5599458445\\
0.0808752490877045 75271.0954536142\\
0.0816250851428723 73259.8201493526\\
0.082381873313996 71273.3783610543\\
0.0831456780577206 69302.101367869\\
0.0839165644283016 67340.5815249843\\
0.0846945980831458 65394.8699055781\\
0.0854798452884041 63501.17308433\\
0.0862723729246145 61771.7250027374\\
0.0870722484923992 60491.7038976334\\
0.0878795401182132 60228.9285220996\\
0.0886943165601471 61552.3128982601\\
0.089516647213783 63826.2161907425\\
0.0903466021181053 65167.4731721276\\
0.0911842519614657 64860.4593029704\\
0.0920296680876042 63540.5742856266\\
0.092882922501725 61853.1582687073\\
0.09374408787663 60107.9142232708\\
0.0946132375589077 58415.2451401969\\
0.0954904455751808 56803.8020799786\\
0.0963757866384109 55273.2914758512\\
0.0972693361542617 53814.5651352308\\
0.0981711702275219 52416.7356099369\\
0.0990813656685867 51069.5155810226\\
0.1 49763.9203173544\\
};
\addlegendentry{full};
\addplot [
color=red,
dash pattern=on 1pt off 3pt on 3pt off 3pt,
]
table[row sep=crcr]{
0.001 4224398.338847\\
0.00100927151463057 4225190.95765133\\
0.00101862899024469 4225998.64982421\\
0.00102807322383086 4226821.70792016\\
0.00103760501976691 4227660.43039791\\
0.00104722518988843 4228515.12174517\\
0.00105693455355799 4229386.09261137\\
0.00106673393773486 4230273.65993948\\
0.00107662417704549 4231178.14710558\\
0.0010866061138546 4232099.88405669\\
0.00109668059833687 4233039.20745595\\
0.00110684848854941 4233996.46083236\\
0.00111711065050482 4234971.99472564\\
0.00112746795824495 4235966.16684764\\
0.00113792129391532 4236979.34223676\\
0.00114847154784029 4238011.89342182\\
0.00115911961859889 4239064.20058901\\
0.00116986641310131 4240136.65175264\\
0.00118071284666619 4241229.6429296\\
0.00119165984309856 4242343.57831939\\
0.00120270833476851 4243478.87048794\\
0.00121385926269063 4244635.9405584\\
0.00122511357660412 4245815.2184013\\
0.00123647223505371 4247017.14283646\\
0.00124793620547131 4248242.16183538\\
0.00125950646425836 4249490.73272986\\
0.00127118399686903 4250763.32242826\\
0.00128296979789415 4252060.40763506\\
0.0012948648711459 4253382.47507648\\
0.00130687022974335 4254730.02173115\\
0.00131898689619867 4256103.5550735\\
0.00133121590250431 4257503.59331246\\
0.00134355829022083 4258930.66564466\\
0.00135601511056563 4260385.31251335\\
0.00136858742450252 4261868.08587304\\
0.00138127630283201 4263379.54946018\\
0.00139408282628258 4264920.27907516\\
0.00140700808560269 4266490.8628666\\
0.00142005318165368 4268091.90162861\\
0.00143321922550357 4269724.00910381\\
0.00144650733852165 4271387.81229435\\
0.00145991865247398 4273083.95178436\\
0.00147345430961984 4274813.08206577\\
0.00148711546280895 4276575.87188163\\
0.00150090327557974 4278373.00457264\\
0.00151481892225835 4280205.17843449\\
0.00152886358805873 4282073.10708861\\
0.00154303846918356 4283977.51986066\\
0.00155734477292614 4285919.16217163\\
0.00157178371777316 4287898.79593876\\
0.00158635653350859 4289917.19999055\\
0.00160106446131832 4291975.17049037\\
0.00161590875389592 4294073.5213776\\
0.00163089067554933 4296213.08481801\\
0.00164601150230855 4298394.7116679\\
0.00166127252203429 4300619.27195514\\
0.0016766750345277 4302887.6553708\\
0.00169222035164104 4305200.77177964\\
0.00170790979738943 4307559.55174258\\
0.00172374470806362 4309964.94705918\\
0.0017397264323438 4312417.93132111\\
0.00175585633141447 4314919.50049175\\
0.00177213577908036 4317470.67349403\\
0.00178856616188346 4320072.4928236\\
0.00180514887922111 4322726.02517971\\
0.00182188534346517 4325432.36211458\\
0.00183877698008233 4328192.62070484\\
0.00185582522775552 4331007.94424615\\
0.00187303153850644 4333879.502967\\
0.00189039737781922 4336808.49476711\\
0.00190792422476527 4339796.14598285\\
0.0019256135721292 4342843.71217386\\
0.00194346692653602 4345952.47893722\\
0.00196148580857943 4349123.76274995\\
0.00197967175295133 4352358.9118385\\
0.00199802630857255 4355659.30707756\\
0.00201655103872475 4359026.36292111\\
0.00203524752118358 4362461.52836283\\
0.00205411734835306 4365966.28793209\\
0.00207316212740123 4369542.16272187\\
0.00209238348039698 4373190.71145548\\
0.00211178304444824 4376913.53158794\\
0.00213136247184144 4380712.26044748\\
0.00215112343018217 4384588.57641672\\
0.00217106760253726 4388544.20016098\\
0.00219119668757815 4392580.8958906\\
0.00221151239972549 4396700.47267913\\
0.00223201646929523 4400904.78582527\\
0.00225271064264598 4405195.73826152\\
0.00227359668232772 4409575.28202348\\
0.00229467636723194 4414045.41976454\\
0.00231595149274315 4418608.20633371\\
0.00233742387089181 4423265.75040893\\
0.00235909533050864 4428020.21619373\\
0.00238096771738036 4432873.82517995\\
0.00240304289440697 4437828.85797727\\
0.00242532274176035 4442887.65621019\\
0.00244780915704444 4448052.62449656\\
0.00247050405545683 4453326.23249593\\
0.00249340936995188 4458711.01704227\\
0.00251652705140539 4464209.58436155\\
0.00253985906878073 4469824.61237973\\
0.00256340740929651 4475558.85311931\\
0.00258717407859592 4481415.13519575\\
0.00261116110091746 4487396.36641768\\
0.00263537051926739 4493505.53649333\\
0.00265980439559376 4499745.71984474\\
0.00268446481096197 4506120.07854701\\
0.00270935386573205 4512631.86538626\\
0.00273447367973758 4519284.42704937\\
0.00275982639246618 4526081.20744729\\
0.00278541416324177 4533025.75118626\\
0.00281123917140846 4540121.70718274\\
0.00283730361651621 4547372.83244367\\
0.00286360971850812 4554782.99600888\\
0.00289015971790951 4562356.18306547\\
0.0029169558760188 4570096.49925559\\
0.00294400047510003 4578008.17516552\\
0.00297129581857733 4586095.57102917\\
0.00299884423123103 4594363.18164197\\
0.00302664805939569 4602815.64150029\\
0.00305470967115997 4611457.73018231\\
0.00308303145656828 4620294.3779782\\
0.00311161582782436 4629330.67178939\\
0.00314046521949675 4638571.86130404\\
0.00316958208872612 4648023.3654727\\
0.00319896891543454 4657690.77929479\\
0.00322862820253673 4667579.88093871\\
0.00325856247615323 4677696.63920404\\
0.00328877428582551 4688047.22136641\\
0.00331926620473319 4698638.00140389\\
0.00335004082991313 4709475.56864122\\
0.00338110078248069 4720566.73683485\\
0.00341244870785289 4731918.55372281\\
0.00344408727597382 4743538.3110681\\
0.00347601918154198 4755433.55522735\\
0.00350824714423979 4767612.09827165\\
0.00354077390896527 4780082.02970589\\
0.00357360224606579 4792851.72880747\\
0.00360673495157403 4805929.87763786\\
0.00364017484744614 4819325.47475703\\
0.00367392478180207 4833047.84970255\\
0.00370798762916817 4847106.67826667\\
0.00374236629072198 4861511.99863026\\
0.00377706369453937 4876274.22840919\\
0.00381208279584389 4891404.18268296\\
0.00384742657725851 4906913.09304953\\
0.00388309804905961 4922812.62780221\\
0.00391910024943341 4939114.91328834\\
0.0039554362447347 4955832.55653434\\
0.00399210912974805 4972978.66923182\\
0.00402912202795135 4990566.89317602\\
0.00406647809178186 5008611.42725898\\
0.00410418050290471 5027127.05613363\\
0.0041422324724839 5046129.1806672\\
0.00418063724145576 5065633.8503181\\
0.00421939808080502 5085657.79757389\\
0.00425851829184341 5106218.4746137\\
0.0042980012064908 5127334.09235792\\
0.00433785018755899 5149023.66208625\\
0.00437806862903817 5171307.03982728\\
0.00441865995638594 5194204.97374162\\
0.0044596276268191 5217739.15472392\\
0.00450097512960805 5241932.27048933\\
0.00454270598637404 5266808.06342721\\
0.0045848237513891 5292391.39252861\\
0.00462733201187869 5318708.2997271\\
0.00467023438832734 5345786.08101826\\
0.00471353453478691 5373653.36277709\\
0.00475723613918789 5402340.18369642\\
0.00480134292365346 5431878.08286399\\
0.0048458586448165 5462300.19449348\\
0.00489078709413959 5493641.34992027\\
0.00493613209823792 5525938.18750383\\
0.00498189751920516 5559229.271174\\
0.00502808725494248 5593555.21840885\\
0.00507470523949047 5628958.83853195\\
0.00512175544336424 5665485.28231679\\
0.0051692418738916 5703182.20397822\\
0.00521716857555435 5742099.93676762\\
0.00526553963033276 5782291.68350763\\
0.00531435915805324 5823813.72357751\\
0.00536363131673924 5866725.63801544\\
0.00541336030296538 5911090.55460834\\
0.00546355035221488 5956975.41506428\\
0.0055142057392403 6004451.26660224\\
0.00556533077842765 6053593.58061524\\
0.00561692982416381 6104482.60133751\\
0.00566900727120743 6157203.72787482\\
0.00572156755506325 6211847.93334475\\
0.00577461515235982 6268512.22539267\\
0.00582815458123085 6327300.15288031\\
0.00588219040169996 6388322.36423306\\
0.00593672721606913 6451697.22363147\\
0.00599176966931062 6517551.49212394\\
0.00604732244946265 6586021.08174714\\
0.00610339028802862 6657251.89183976\\
0.00615997796038017 6731400.73816419\\
0.00621709028616383 6808636.38694615\\
0.00627473212971158 6889140.70781611\\
0.00633290840045512 6973109.96178612\\
0.00639162405334401 7060756.24288429\\
0.00645088408926769 7152309.0950869\\
0.00651069355548146 7248017.32966081\\
0.00657105754603627 7348151.07219067\\
0.00663198120221268 7453004.07350428\\
0.00669346971295866 7562896.32459158\\
0.00675552831533165 7678177.022567\\
0.00681816229494448 7799227.94323037\\
0.00688137698641567 7926467.28576168\\
0.0069451777738237 8060354.06732769\\
0.00700957009116562 8201393.16008058\\
0.00707455942281988 8350141.0808254\\
0.0071401513040134 8507212.66541521\\
0.00720635132129305 8673288.7864045\\
0.00727316511300145 8849125.30497574\\
0.00734059836975721 9035563.48801405\\
0.00740865683493957 9233542.17040243\\
0.00747734630517759 9444112.0032491\\
0.00754667263084389 9668452.20403829\\
0.00761664171655289 9907890.31792928\\
0.00768725952166374 10163925.6154746\\
0.00775853206078784 10438256.8962023\\
0.00783046540430118 10732815.6467036\\
0.00790306567886135 11049805.7240331\\
0.00797633906792928 11391751.0090196\\
0.00805029181229598 11761552.8087603\\
0.00812493021061405 12162559.1898038\\
0.00820026061993413 12598648.8941443\\
0.00827628945624635 13074333.0149831\\
0.00835302319502678 13594878.1406135\\
0.00843046837178897 14166455.0992758\\
0.00850863158264058 14796317.5076193\\
0.00858751948484518 15493013.5325085\\
0.00866713879738923 16266631.6058184\\
0.00874749630155442 17129074.2953672\\
0.00882859884149515 18094340.2760345\\
0.00891045332482151 19178764.973382\\
0.00899306672318762 20401112.0202125\\
0.00907644607288536 21782294.5392958\\
0.00916059847544371 23344292.7155081\\
0.00924553109823357 25107452.3330744\\
0.00933125117507825 27084717.6416697\\
0.00941776600686952 29270482.9460155\\
0.00950508296218951 31621116.1360261\\
0.00959320947793824 34025689.413768\\
0.00968215305996709 36273541.9778319\\
0.009771921283718 38044787.3402319\\
0.00986252179486878 38970152.77602\\
0.00995396230998423 38780441.5427132\\
0.0100462506171734 37462161.6778655\\
0.0101393945767529 35273421.88868\\
0.0102334021219164 32600061.535769\\
0.0103282812594103 29790000.0151362\\
0.0104240400702156 27075296.7236253\\
0.0105206867102362 24576034.3183863\\
0.0106182294109938 22335583.3036874\\
0.0107166764803286 20354373.7203817\\
0.0108160363031071 18612629.2819014\\
0.0109163173419361 17083255.5350001\\
0.0110175281378839 15738331.219806\\
0.011119677311207 14551997.4923468\\
0.0112227735620851 13501489.3406275\\
0.0113268256713615 12567278.8952605\\
0.0114318425012915 11732834.8211624\\
0.0115378329962966 10984246.7950544\\
0.0116448061837269 10309831.1794866\\
0.0117527711746295 9699767.26227\\
0.0118617371645248 9145781.09882992\\
0.0119717134341897 8640879.10297715\\
0.0120827093504478 8179127.27252041\\
0.0121947343669674 7755469.82389026\\
0.0123077980250667 7365580.79505257\\
0.0124219099545262 7005742.75002796\\
0.0125370798744092 6672747.55133296\\
0.0126533175938894 6363815.01714014\\
0.0127706330130864 6076526.04837611\\
0.0128890361239089 5808767.46621761\\
0.0130085370109057 5558686.34153546\\
0.0131291458521247 5324652.03681266\\
0.0132508729199795 5105224.53347804\\
0.013373728582125 4899127.89890369\\
0.0134977233023394 4705227.97131902\\
0.0136228676414165 4522513.51911347\\
0.0137491722580642 4350080.27293039\\
0.0138766479098131 4187117.34217353\\
0.0140053054539322 4032895.61804248\\
0.014135155848354 3886757.83779975\\
0.0142662101526074 3748110.04331736\\
0.0143984795287601 3616414.21405717\\
0.014531975242369 3491181.89274836\\
0.0146667086634397 3371968.65301467\\
0.0148026912673951 3258369.28344695\\
0.0149399346360526 3150013.58328542\\
0.0150784504586105 3046562.68183487\\
0.0152182505326439 2947705.8077152\\
0.015359346765109 2853157.44559906\\
0.0155017511733577 2762654.82766822\\
0.015645475886161 2675955.71499728\\
0.0157905331447418 2592836.4307206\\
0.0159369353038178 2513090.11241601\\
0.0160846948326536 2436525.15581379\\
0.0162338243161228 2362963.82588123\\
0.0163843364557798 2292241.0146609\\
0.0165362440709418 2224203.12806188\\
0.0166895600997802 2158707.08619393\\
0.0168442976004231 2095619.42387317\\
0.0170004697520672 2034815.47966949\\
0.0171580898561 1976178.66334812\\
0.0173171713372335 1919599.79283727\\
0.0174777277446468 1864976.49294606\\
0.0176397727531404 1812212.64899784\\
0.0178033201643011 1761217.90935875\\
0.0179683839076772 1711907.23153578\\
0.018134978041965 1664200.46712908\\
0.0183031167562061 1618021.98143755\\
0.0184728143709963 1573300.303972\\
0.0186440853397049 1529967.80651386\\
0.0188169442497056 1487960.40569305\\
0.0189914058236193 1447217.28733877\\
0.0191674849205681 1407680.65010022\\
0.0193451965374405 1369295.46603607\\
0.0195245558101686 1332009.25603272\\
0.019705578015018 1295771.87804364\\
0.0198882785698881 1260535.32623271\\
0.0200726730356257 1226253.53917768\\
0.0202587771173502 1192882.21535615\\
0.0204466066657912 1160378.63409088\\
0.0206361776786386 1128701.48004649\\
0.0208275063019052 1097810.66936697\\
0.0210206088313016 1067667.17541377\\
0.0212155017136245 1038232.85184989\\
0.0214122015481573 1009470.2505712\\
0.0216107250880838 981342.431666772\\
0.0218110892419152 953812.762170401\\
0.0220133110749303 926844.699835288\\
0.0222174078106288 900401.557488199\\
0.0224233968321985 874446.242673906\\
0.0226312956839953 848940.966243245\\
0.0228411220730382 823846.912221268\\
0.0230528938705171 799123.859662593\\
0.0232666291133146 774729.745206671\\
0.0234823460055428 750620.152634278\\
0.0237000629200933 726747.71287956\\
0.0239197984002024 703061.39473448\\
0.0241415711610302 679505.663119826\\
0.0243654000912547 656019.478904183\\
0.0245913042546805 632535.113208818\\
0.0248193028918626 608976.752932721\\
0.025049415421745 585258.88916103\\
0.025281661443315 561284.519407016\\
0.0255160607372719 536943.286192363\\
0.025752633267712 512109.878140415\\
0.0259913991838293 486643.466551564\\
0.0262323788216312 460389.935066097\\
0.0264755927056707 433190.866356504\\
0.0267210615507944 404908.318738566\\
0.0269688062639069 375486.417674034\\
0.0272188479457518 345099.91245793\\
0.0274712078927081 314510.89407366\\
0.0277259075986048 285915.478474222\\
0.0279829687565512 264797.808913995\\
0.0282424132607843 262844.768904527\\
0.0285042632085343 298333.787863453\\
0.0287685409019059 389015.237510591\\
0.0290352688497781 546897.421691519\\
0.0293044697697214 775423.511098423\\
0.0295761665899325 1038395.80562343\\
0.0298503824511873 1221010.45507893\\
0.0301271407088116 1239247.63447729\\
0.0304064649346707 1154066.8776153\\
0.0306883789191764 1047056.77825747\\
0.0309729066733141 950743.763420778\\
0.031260072430687 870827.105070957\\
0.0315499006495809 805363.325723832\\
0.0318424160150465 751260.320918844\\
0.0321376434410028 705853.678149571\\
0.0324356080723581 667113.689485589\\
0.0327363352871525 633543.060861069\\
0.0330398506987186 604040.217813836\\
0.0333461801578636 577787.829330631\\
0.0336553497550709 554171.685827515\\
0.0339673858227221 532723.654598246\\
0.0342823149373397 513081.827771798\\
0.0346001639218511 494962.542609066\\
0.0349209598478727 478140.539442478\\
0.0352447300380159 462434.718255633\\
0.0355715020682139 447697.781290069\\
0.0359013037700707 433808.603089978\\
0.0362341632332315 420666.537817533\\
0.0365701088077749 408187.119188547\\
0.0369091691066278 396298.773155119\\
0.0372513730080021 384940.275105522\\
0.0375967496578547 374058.759787839\\
0.0379453284723694 363608.145047354\\
0.0382971391404628 353547.867414891\\
0.0386522116263126 343841.853612461\\
0.0390105761719099 334457.670480717\\
0.0393722632996348 325365.808899093\\
0.0397373038148561 316539.066427229\\
0.040105728808555 307951.999733704\\
0.0404775696599732 299580.42211008\\
0.0408528580392856 291400.923193673\\
0.0412316259102975 283390.388673346\\
0.0416139055331671 275525.496957804\\
0.0419997294671531 267782.166880954\\
0.0423891305733878 260134.926105781\\
0.0427821420176762 252556.163840988\\
0.0431787972733202 245015.224405521\\
0.0435791301239703 237477.292464622\\
0.0439831746665023 229902.023410306\\
0.0443909653139217 222241.902465388\\
0.0448025367982949 214440.424594683\\
0.045217924173707 206430.510639962\\
0.0456371628192476 198134.486250687\\
0.046060288442024 189469.50156176\\
0.0464873370802026 180369.541097899\\
0.046918345106078 170856.462201226\\
0.0473533492291712 161255.775947217\\
0.047792386499356 152831.036508456\\
0.0482354943100147 149453.283412441\\
0.0486827104012228 160243.156796139\\
0.0491340728629636 195751.989233008\\
0.0495896201383721 245315.348422445\\
0.0500493910270095 273966.913575418\\
0.0505134246881676 273533.004315296\\
0.0509817606442042 260752.523987705\\
0.0514544387839093 246208.541974252\\
0.0519314993659021 233011.923208487\\
0.0524129830220606 221597.763003703\\
0.0528989307609815 211740.701072506\\
0.0533893839714735 203126.910596533\\
0.0538843844260822 195486.851516131\\
0.054383974284648 188612.132588574\\
0.0548881960978967 182345.975971949\\
0.055397092811064 176570.551457635\\
0.0559107077675529 171196.443933877\\
0.0564290847126254 166154.816552874\\
0.0569522677971283 161391.77374671\\
0.0574803015812535 156864.329614201\\
0.0580132310383338 152537.502624934\\
0.0585511015586724 148382.188821895\\
0.0590939589534097 144373.569020014\\
0.0596418494584246 140489.878518154\\
0.0601948197382727 136711.417346868\\
0.0607529168901607 133019.712200845\\
0.0613161884479577 129396.763644039\\
0.0618846823862439 125824.331184415\\
0.0624584471243962 122283.233293076\\
0.0630375315307128 118752.68863831\\
0.0636219849265749 115209.847429925\\
0.0642118570906476 111629.984586322\\
0.0648071982631197 107988.697733634\\
0.0654080591499826 104269.859391187\\
0.0660144909273489 100489.851610048\\
0.0666265452458115 96767.6205050046\\
0.0672442742348425 93518.2420685063\\
0.067867730507233 91913.3825286184\\
0.0684969671635745 94411.293672909\\
0.0691320377967813 102879.210751372\\
0.0697729964966554 112657.924021761\\
0.0704198978544929 116619.632298306\\
0.0710727969677342 115143.287287389\\
0.0717317494446562 111505.084342603\\
0.0723968114091088 107477.119757464\\
0.073068039505295 103638.737504585\\
0.0737454909025955 100114.046321497\\
0.0744292233004376 96890.2399827632\\
0.0751192949332097 93922.6868851317\\
0.0758157645752211 91165.8913922726\\
0.0765186915457083 88581.0605284733\\
0.0772281357138865 86136.7055930923\\
0.0779441575040495 83807.4362320725\\
0.0786668179007158 81572.5092245494\\
0.0793961784538228 79414.5551982955\\
0.0801323012839689 77318.5599435817\\
0.0808752490877045 75271.0954528928\\
0.0816250851428723 73259.8201494392\\
0.082381873313996 71273.3783613893\\
0.0831456780577206 69302.1013682226\\
0.0839165644283016 67340.5815252907\\
0.0846945980831458 65394.8699058399\\
0.0854798452884041 63501.1730845746\\
0.0862723729246145 61771.7250029103\\
0.0870722484923992 60491.7038978171\\
0.0878795401182132 60228.9285222822\\
0.0886943165601471 61552.3128982649\\
0.089516647213783 63826.2161904098\\
0.0903466021181053 65167.4731717781\\
0.0911842519614657 64860.4593029773\\
0.0920296680876042 63540.5742860982\\
0.092882922501725 61853.1582696203\\
0.09374408787663 60107.9142245105\\
0.0946132375589077 58415.2451416301\\
0.0954904455751808 56803.8020813553\\
0.0963757866384109 55273.2914766178\\
0.0972693361542617 53814.5651342841\\
0.0981711702275219 52416.7356060323\\
0.0990813656685867 51069.5155748128\\
0.1 49763.9203123823\\
};
\addlegendentry{reduced model};
\end{axis}
\end{tikzpicture}%

%% file: figures/abserr_tchain_30.tex
%
%
%
%
\begin{tikzpicture}

\begin{axis}[%
width=\figwidth,
height=\figheight,
scale only axis,
separate axis lines,
every outer x axis line/.append style={darkgray!60!black},
every x tick label/.append style={font=\color{darkgray!60!black}},
xmode=log,
xmin=0.001,
xmax=0.1,
xminorticks=true,
xlabel={$\omega$},
every outer y axis line/.append style={darkgray!60!black},
every y tick label/.append style={font=\color{darkgray!60!black}},
ymode=log,
ymin=1e-08,
ymax=0.01,
yminorticks=true,
ylabel={$\sigma{}_{\text{max}}\text{($\mathcal{T}$(j}\omega\text{)-$\hat{\mathcal{T}}$}\text{(j}\omega\text{))}$},
]
\addplot [
color=red,
dash pattern=on 1pt off 3pt on 3pt off 3pt,
forget plot
]
table[row sep=crcr]{
0.001 6.90385036515173e-05\\
0.00100927151463057 4.74447706144254e-05\\
0.00101862899024469 3.24770482967033e-05\\
0.00102807322383086 2.48842780969029e-05\\
0.00103760501976691 2.28160719992163e-05\\
0.00104722518988843 1.75352563309028e-05\\
0.00105693455355799 4.96304174182332e-05\\
0.00106673393773486 4.91883974429915e-05\\
0.00107662417704549 7.13363738172837e-05\\
0.0010866061138546 3.59888094363642e-05\\
0.00109668059833687 0.000134852981245945\\
0.00110684848854941 6.43378623775086e-05\\
0.00111711065050482 5.20988571817301e-05\\
0.00112746795824495 1.91182154113036e-05\\
0.00113792129391532 6.3543302032384e-05\\
0.00114847154784029 1.63825254526405e-05\\
0.00115911961859889 6.73967884596991e-05\\
0.00116986641310131 6.81557124908032e-05\\
0.00118071284666619 1.48962963071154e-05\\
0.00119165984309856 3.49342162853897e-05\\
0.00120270833476851 6.81044343869777e-05\\
0.00121385926269063 2.32231212762638e-05\\
0.00122511357660412 3.91947659390984e-05\\
0.00123647223505371 8.45599572560023e-06\\
0.00124793620547131 4.71706508811303e-05\\
0.00125950646425836 7.35325223725206e-05\\
0.00127118399686903 1.24801246041238e-05\\
0.00128296979789415 3.76138762656695e-05\\
0.0012948648711459 4.66586650788911e-05\\
0.00130687022974335 0.000110718613482116\\
0.00131898689619867 7.56033135539392e-05\\
0.00133121590250431 4.96807013907356e-05\\
0.00134355829022083 3.77740031521433e-05\\
0.00135601511056563 5.38811488760697e-05\\
0.00136858742450252 4.94729568041653e-05\\
0.00138127630283201 6.899500674529e-06\\
0.00139408282628258 1.75052027425839e-05\\
0.00140700808560269 9.17491840550826e-05\\
0.00142005318165368 5.21277054200292e-05\\
0.00143321922550357 4.44265817887541e-05\\
0.00144650733852165 6.20691991061839e-05\\
0.00145991865247398 3.66103409393106e-06\\
0.00147345430961984 2.02290723164932e-05\\
0.00148711546280895 3.14284412147799e-05\\
0.00150090327557974 0.000111451522554818\\
0.00151481892225835 1.51492110492491e-05\\
0.00152886358805873 0.000135476565871577\\
0.00154303846918356 3.48075769280653e-05\\
0.00155734477292614 8.85678678002713e-05\\
0.00157178371777316 7.53671221405992e-05\\
0.00158635653350859 2.86771983716748e-05\\
0.00160106446131832 2.52181756145277e-05\\
0.00161590875389592 3.80712613135988e-05\\
0.00163089067554933 4.84196416035945e-05\\
0.00164601150230855 3.37440745217939e-05\\
0.00166127252203429 3.22313757473577e-05\\
0.0016766750345277 4.44538929185689e-05\\
0.00169222035164104 6.14539156736161e-06\\
0.00170790979738943 8.2366002581171e-05\\
0.00172374470806362 5.14405474773462e-05\\
0.0017397264323438 1.68333734722072e-06\\
0.00175585633141447 7.71722810571551e-05\\
0.00177213577908036 0.000109132494496059\\
0.00178856616188346 8.2304208536334e-06\\
0.00180514887922111 5.01568110079896e-05\\
0.00182188534346517 4.92768109754205e-05\\
0.00183877698008233 6.62131314336998e-05\\
0.00185582522775552 3.64589812875086e-05\\
0.00187303153850644 4.06544073454623e-05\\
0.00189039737781922 3.07982819960513e-05\\
0.00190792422476527 0.000100187991665013\\
0.0019256135721292 6.99742109297222e-05\\
0.00194346692653602 6.79951369514302e-05\\
0.00196148580857943 4.22672693518714e-05\\
0.00197967175295133 7.22929193796129e-05\\
0.00199802630857255 0.000104580478665808\\
0.00201655103872475 3.99910454449527e-06\\
0.00203524752118358 8.64159329039833e-05\\
0.00205411734835306 6.63332735718389e-05\\
0.00207316212740123 6.42249160647436e-05\\
0.00209238348039698 6.30069348690275e-05\\
0.00211178304444824 9.72096367943096e-05\\
0.00213136247184144 4.4252607049393e-05\\
0.00215112343018217 3.11893868841421e-05\\
0.00217106760253726 5.47111444541864e-05\\
0.00219119668757815 6.345535205793e-05\\
0.00221151239972549 6.56726996638209e-05\\
0.00223201646929523 3.19887183643402e-05\\
0.00225271064264598 1.84786729668506e-05\\
0.00227359668232772 1.01500590185239e-05\\
0.00229467636723194 7.74846571709839e-05\\
0.00231595149274315 1.5962598933652e-06\\
0.00233742387089181 9.94526076040743e-06\\
0.00235909533050864 7.46745664233876e-05\\
0.00238096771738036 2.32997431451672e-05\\
0.00240304289440697 9.03964208796835e-05\\
0.00242532274176035 5.62294350426057e-05\\
0.00244780915704444 2.15239853959181e-05\\
0.00247050405545683 5.61419064406741e-05\\
0.00249340936995188 1.96696530165676e-07\\
0.00251652705140539 3.19722334917572e-05\\
0.00253985906878073 3.68824137491194e-05\\
0.00256340740929651 2.58845942781457e-06\\
0.00258717407859592 0.000146574166036517\\
0.00261116110091746 8.10210671723397e-06\\
0.00263537051926739 1.9969910095573e-05\\
0.00265980439559376 4.00834572243835e-05\\
0.00268446481096197 5.9897312251828e-05\\
0.00270935386573205 4.59541147512459e-05\\
0.00273447367973758 9.27299546882404e-06\\
0.00275982639246618 6.53205740368549e-05\\
0.00278541416324177 2.55798786154304e-05\\
0.00281123917140846 2.92703676124804e-05\\
0.00283730361651621 7.16159898215949e-05\\
0.00286360971850812 5.70324541306169e-05\\
0.00289015971790951 5.66017402774308e-05\\
0.0029169558760188 7.40955996303551e-05\\
0.00294400047510003 1.51435115390916e-05\\
0.00297129581857733 9.91371261381732e-05\\
0.00299884423123103 2.91574425772176e-05\\
0.00302664805939569 7.58256486922549e-05\\
0.00305470967115997 8.25737259077678e-06\\
0.00308303145656828 3.59358270825693e-05\\
0.00311161582782436 4.22885733759528e-05\\
0.00314046521949675 0.000123264579495441\\
0.00316958208872612 8.56204083773039e-05\\
0.00319896891543454 3.23413832785082e-05\\
0.00322862820253673 3.69329392565622e-05\\
0.00325856247615323 3.09487065978867e-05\\
0.00328877428582551 2.52221995972117e-05\\
0.00331926620473319 2.15166008205684e-05\\
0.00335004082991313 4.63148645961675e-05\\
0.00338110078248069 5.09164234967943e-05\\
0.00341244870785289 7.79149595440072e-05\\
0.00344408727597382 2.02006693481314e-05\\
0.00347601918154198 1.53903739610718e-06\\
0.00350824714423979 9.7481347111371e-05\\
0.00354077390896527 6.64417834700164e-05\\
0.00357360224606579 5.80480159518148e-05\\
0.00360673495157403 1.43535879964205e-05\\
0.00364017484744614 8.64165415777465e-05\\
0.00367392478180207 2.30447946974041e-05\\
0.00370798762916817 2.70525533426025e-05\\
0.00374236629072198 2.61513085979026e-05\\
0.00377706369453937 1.94418631467127e-05\\
0.00381208279584389 1.43749843713701e-05\\
0.00384742657725851 4.8153572891606e-05\\
0.00388309804905961 0.000156381345521225\\
0.00391910024943341 1.9926830691345e-05\\
0.0039554362447347 0.00012266080829754\\
0.00399210912974805 3.36037076997395e-05\\
0.00402912202795135 4.0102835180573e-05\\
0.00406647809178186 8.21581603593672e-05\\
0.00410418050290471 3.76088807101414e-05\\
0.0041422324724839 0.000113120807066265\\
0.00418063724145576 3.97964108991184e-06\\
0.00421939808080502 6.9120139358384e-06\\
0.00425851829184341 4.43261678481185e-06\\
0.0042980012064908 4.79993700152091e-05\\
0.00433785018755899 0.000101985562206066\\
0.00437806862903817 0.000103996783578774\\
0.00441865995638594 6.64098151531641e-05\\
0.0044596276268191 0.000106264355571572\\
0.00450097512960805 2.00339561509443e-05\\
0.00454270598637404 0.000251452804547063\\
0.0045848237513891 6.09968210122785e-05\\
0.00462733201187869 7.41256044531827e-05\\
0.00467023438832734 0.000178442555738715\\
0.00471353453478691 4.16647474259313e-06\\
0.00475723613918789 0.000109559149631003\\
0.00480134292365346 0.000148775885798711\\
0.0048458586448165 4.66703187994838e-05\\
0.00489078709413959 9.83164154586871e-05\\
0.00493613209823792 6.78898162894094e-05\\
0.00498189751920516 0.000129750880737646\\
0.00502808725494248 7.66707507595294e-06\\
0.00507470523949047 7.92653488623192e-05\\
0.00512175544336424 2.77387381470763e-05\\
0.0051692418738916 8.41981827636559e-05\\
0.00521716857555435 0.000112523758787315\\
0.00526553963033276 6.8859856107138e-05\\
0.00531435915805324 5.46396687683757e-05\\
0.00536363131673924 3.69480770259827e-05\\
0.00541336030296538 8.04893431268796e-06\\
0.00546355035221488 3.00144050282072e-05\\
0.0055142057392403 0.000160844670646765\\
0.00556533077842765 0.000105572678462004\\
0.00561692982416381 1.74742130524081e-05\\
0.00566900727120743 9.68832379220689e-05\\
0.00572156755506325 9.04175098432491e-05\\
0.00577461515235982 0.000115070145747888\\
0.00582815458123085 0.000134729376028249\\
0.00588219040169996 3.23833296845604e-05\\
0.00593672721606913 0.000243490580212758\\
0.00599176966931062 0.000163286421509648\\
0.00604732244946265 8.64647268370605e-05\\
0.00610339028802862 5.51258901261843e-05\\
0.00615997796038017 2.45267412697136e-06\\
0.00621709028616383 0.000133449912756399\\
0.00627473212971158 9.78316100781575e-05\\
0.00633290840045512 0.000133805001905096\\
0.00639162405334401 6.66978624692589e-05\\
0.00645088408926769 0.000103557743501364\\
0.00651069355548146 4.21816798702924e-05\\
0.00657105754603627 3.35491015655144e-05\\
0.00663198120221268 2.2224558533341e-05\\
0.00669346971295866 6.1494204708299e-05\\
0.00675552831533165 6.61573061057624e-05\\
0.00681816229494448 3.68736272138121e-05\\
0.00688137698641567 0.000292612460642169\\
0.0069451777738237 0.00025148430413956\\
0.00700957009116562 9.31561836087981e-05\\
0.00707455942281988 2.32346642345208e-05\\
0.0071401513040134 8.91099285512125e-05\\
0.00720635132129305 0.000403195664758978\\
0.00727316511300145 5.97071683667637e-05\\
0.00734059836975721 0.000280338540820531\\
0.00740865683493957 0.000246640581511208\\
0.00747734630517759 0.000203237348565265\\
0.00754667263084389 0.000152161231674196\\
0.00761664171655289 0.000150262707881985\\
0.00768725952166374 0.00023628771830477\\
0.00775853206078784 0.000422464286895526\\
0.00783046540430118 0.000275065214417001\\
0.00790306567886135 0.000339818441375824\\
0.00797633906792928 0.000235994642024595\\
0.00805029181229598 6.53251874559714e-05\\
0.00812493021061405 0.000440595783013306\\
0.00820026061993413 0.000260983072899537\\
0.00827628945624635 0.000439929512211997\\
0.00835302319502678 0.000333383982292162\\
0.00843046837178897 0.000446690274521139\\
0.00850863158264058 0.000162171915579288\\
0.00858751948484518 0.00101309680871894\\
0.00866713879738923 0.000781795846222438\\
0.00874749630155442 0.000116119634878574\\
0.00882859884149515 0.00160195796076825\\
0.00891045332482151 0.000795687836199492\\
0.00899306672318762 0.00284900564670163\\
0.00907644607288536 0.00272109883716352\\
0.00916059847544371 0.000407118416163473\\
0.00924553109823357 0.000673665183163315\\
0.00933125117507825 0.00545746574549975\\
0.00941776600686952 0.000622724218331617\\
0.00950508296218951 0.00212730545201725\\
0.00959320947793824 0.00751483637866595\\
0.00968215305996709 0.00200447998610007\\
0.009771921283718 0.004181369415759\\
0.00986252179486878 0.00479971999732526\\
0.00995396230998423 0.00141182438505409\\
0.0100462506171734 0.00740825193400624\\
0.0101393945767529 0.00909809403086843\\
0.0102334021219164 0.00378703227149064\\
0.0103282812594103 0.00136097539738158\\
0.0104240400702156 0.00123207539156894\\
0.0105206867102362 0.00251360954820955\\
0.0106182294109938 0.000635283443434257\\
0.0107166764803286 0.00111319045539153\\
0.0108160363031071 0.000226194243447486\\
0.0109163173419361 0.00135719202773661\\
0.0110175281378839 0.00109576702804042\\
0.011119677311207 0.000633295473505461\\
0.0112227735620851 0.000662065697090412\\
0.0113268256713615 7.40135940943756e-05\\
0.0114318425012915 0.000398536830504105\\
0.0115378329962966 0.000448192106771432\\
0.0116448061837269 0.000338406029783322\\
0.0117527711746295 4.05231325709837e-05\\
0.0118617371645248 0.000331150580677257\\
0.0119717134341897 0.000171787069784344\\
0.0120827093504478 0.000262874494274215\\
0.0121947343669674 7.00828939486118e-06\\
0.0123077980250667 0.000200358028238962\\
0.0124219099545262 0.000137373683176302\\
0.0125370798744092 0.000167070694181732\\
0.0126533175938894 5.26772809457517e-05\\
0.0127706330130864 2.92650282775968e-05\\
0.0128890361239089 9.00754190839456e-06\\
0.0130085370109057 5.86626161814975e-05\\
0.0131291458521247 3.45058450221388e-05\\
0.0132508729199795 4.74537529883049e-06\\
0.013373728582125 2.6162910213367e-05\\
0.0134977233023394 1.82575314205774e-05\\
0.0136228676414165 6.03693051290682e-05\\
0.0137491722580642 7.92878939553211e-05\\
0.0138766479098131 4.50574806986701e-05\\
0.0140053054539322 3.21195425670283e-05\\
0.014135155848354 3.06402166413863e-05\\
0.0142662101526074 5.6720746293415e-05\\
0.0143984795287601 5.03130889404644e-05\\
0.014531975242369 7.31040128437047e-05\\
0.0146667086634397 2.95317956738614e-05\\
0.0148026912673951 3.54819647885836e-05\\
0.0149399346360526 1.82608716248083e-05\\
0.0150784504586105 4.44456238925141e-05\\
0.0152182505326439 1.216455768938e-06\\
0.015359346765109 2.67151646499113e-05\\
0.0155017511733577 1.08332555536988e-05\\
0.015645475886161 6.6731137204762e-06\\
0.0157905331447418 2.17974406012983e-06\\
0.0159369353038178 5.06936582240683e-06\\
0.0160846948326536 7.60861374260042e-06\\
0.0162338243161228 1.02643447911372e-05\\
0.0163843364557798 6.50399620996574e-06\\
0.0165362440709418 8.08306822569037e-06\\
0.0166895600997802 1.21045499553274e-05\\
0.0168442976004231 4.29733212416131e-06\\
0.0170004697520672 1.77522542052659e-05\\
0.0171580898561 2.9490687878718e-05\\
0.0173171713372335 1.47961711853211e-05\\
0.0174777277446468 1.30276669272466e-06\\
0.0176397727531404 4.34790418504444e-06\\
0.0178033201643011 1.62244647973173e-05\\
0.0179683839076772 1.45042013039106e-05\\
0.018134978041965 4.37574487783706e-06\\
0.0183031167562061 4.75949461831675e-06\\
0.0184728143709963 5.84254580197181e-06\\
0.0186440853397049 6.86825698555148e-06\\
0.0188169442497056 8.63544477434147e-06\\
0.0189914058236193 2.25750530616088e-05\\
0.0191674849205681 1.62365718182103e-05\\
0.0193451965374405 2.26142124157756e-05\\
0.0195245558101686 2.5060309727883e-05\\
0.019705578015018 3.98693342372041e-05\\
0.0198882785698881 3.31316966933706e-05\\
0.0200726730356257 3.33908902069019e-05\\
0.0202587771173502 2.43534860750963e-05\\
0.0204466066657912 1.85569221964852e-05\\
0.0206361776786386 1.38844050626118e-05\\
0.0208275063019052 1.06101068932507e-05\\
0.0210206088313016 1.08495034226009e-05\\
0.0212155017136245 4.54195550474433e-06\\
0.0214122015481573 2.64764099404412e-06\\
0.0216107250880838 7.622940593224e-06\\
0.0218110892419152 5.29241825223623e-06\\
0.0220133110749303 2.30333474985194e-06\\
0.0222174078106288 7.08376308086352e-07\\
0.0224233968321985 2.59128883045285e-06\\
0.0226312956839953 2.0213001677121e-06\\
0.0228411220730382 1.13064483031819e-06\\
0.0230528938705171 5.43136423017929e-06\\
0.0232666291133146 2.45278465715248e-06\\
0.0234823460055428 3.73264425313676e-06\\
0.0237000629200933 2.84632051404964e-06\\
0.0239197984002024 5.71699816990908e-07\\
0.0241415711610302 1.11133899177964e-06\\
0.0243654000912547 3.09541727783026e-06\\
0.0245913042546805 2.70890711827182e-06\\
0.0248193028918626 3.79384505332401e-06\\
0.025049415421745 3.5717879674416e-06\\
0.025281661443315 1.48139000024522e-06\\
0.0255160607372719 1.30555709899566e-06\\
0.025752633267712 5.16621815712196e-07\\
0.0259913991838293 2.86662624158222e-06\\
0.0262323788216312 4.17757502166838e-06\\
0.0264755927056707 2.72828675713705e-06\\
0.0267210615507944 4.48767519747531e-07\\
0.0269688062639069 5.39047609570237e-06\\
0.0272188479457518 4.95722127845139e-06\\
0.0274712078927081 1.40862082407553e-07\\
0.0277259075986048 4.132020621126e-07\\
0.0279829687565512 3.01502267406864e-06\\
0.0282424132607843 3.49470797520853e-06\\
0.0285042632085343 4.41364631400448e-07\\
0.0287685409019059 1.82804867155107e-05\\
0.0290352688497781 4.65488215995714e-06\\
0.0293044697697214 1.10692130755494e-05\\
0.0295761665899325 1.27200758608475e-05\\
0.0298503824511873 3.47397438689306e-05\\
0.0301271407088116 2.41778833985894e-06\\
0.0304064649346707 1.46343765677741e-05\\
0.0306883789191764 2.03480413605629e-06\\
0.0309729066733141 6.6659916532056e-06\\
0.031260072430687 4.68753372754215e-06\\
0.0315499006495809 6.25692132744021e-06\\
0.0318424160150465 2.21294679996715e-06\\
0.0321376434410028 4.50109981368996e-07\\
0.0324356080723581 4.51383071533073e-06\\
0.0327363352871525 2.18295780591069e-06\\
0.0330398506987186 1.18454864221312e-06\\
0.0333461801578636 3.5792107944311e-07\\
0.0336553497550709 1.82590284334873e-06\\
0.0339673858227221 1.31893138071148e-07\\
0.0342823149373397 2.31672036405066e-06\\
0.0346001639218511 1.68099881040264e-06\\
0.0349209598478727 1.1247532524003e-06\\
0.0352447300380159 4.34244025457581e-07\\
0.0355715020682139 9.99642303543057e-07\\
0.0359013037700707 2.1616443431554e-07\\
0.0362341632332315 7.5561423883878e-07\\
0.0365701088077749 7.22612043132286e-07\\
0.0369091691066278 7.58302296009515e-07\\
0.0372513730080021 1.20195230845166e-06\\
0.0375967496578547 9.12823192851975e-07\\
0.0379453284723694 1.4124941884781e-06\\
0.0382971391404628 1.4811243823179e-06\\
0.0386522116263126 2.37074005912982e-06\\
0.0390105761719099 3.66280488178884e-06\\
0.0393722632996348 4.3483643950458e-06\\
0.0397373038148561 4.96721066916682e-06\\
0.040105728808555 4.63193643512867e-06\\
0.0404775696599732 3.31244971354203e-06\\
0.0408528580392856 2.30021613615491e-06\\
0.0412316259102975 1.22574876202784e-06\\
0.0416139055331671 7.64156088633155e-07\\
0.0419997294671531 6.32322481092667e-07\\
0.0423891305733878 4.91422915175944e-07\\
0.0427821420176762 7.33128811552133e-07\\
0.0431787972733202 2.22885957635513e-07\\
0.0435791301239703 7.40014606101021e-07\\
0.0439831746665023 3.35509331982242e-07\\
0.0443909653139217 5.02400648157949e-07\\
0.0448025367982949 4.24625553757165e-08\\
0.045217924173707 5.370820896635e-07\\
0.0456371628192476 2.49458805553463e-07\\
0.046060288442024 1.63166895657442e-07\\
0.0464873370802026 1.18486472840771e-07\\
0.046918345106078 2.37912176199918e-07\\
0.0473533492291712 4.76912132331145e-07\\
0.047792386499356 1.85203079792932e-07\\
0.0482354943100147 9.49276153412348e-07\\
0.0486827104012228 1.23646709688948e-06\\
0.0491340728629636 4.11059822589905e-07\\
0.0495896201383721 4.56761543984231e-07\\
0.0500493910270095 5.01532516446463e-07\\
0.0505134246881676 9.87892234894081e-07\\
0.0509817606442042 1.26620304333951e-06\\
0.0514544387839093 2.23049869387183e-07\\
0.0519314993659021 8.36220795022215e-07\\
0.0524129830220606 1.328574369961e-07\\
0.0528989307609815 6.28271581165768e-07\\
0.0533893839714735 4.15164000130539e-07\\
0.0538843844260822 2.34619382759069e-07\\
0.054383974284648 2.21866502217146e-07\\
0.0548881960978967 2.836807391082e-07\\
0.055397092811064 3.91822346533925e-07\\
0.0559107077675529 3.9426943893234e-07\\
0.0564290847126254 6.14485051112353e-07\\
0.0569522677971283 6.73558640850616e-07\\
0.0574803015812535 9.72977933588191e-07\\
0.0580132310383338 1.27290298707284e-06\\
0.0585511015586724 1.76156013873038e-06\\
0.0590939589534097 2.42647509948264e-06\\
0.0596418494584246 2.47371918846209e-06\\
0.0601948197382727 1.95958945417556e-06\\
0.0607529168901607 1.27535189930332e-06\\
0.0613161884479577 8.22562193438492e-07\\
0.0618846823862439 5.91358047579053e-07\\
0.0624584471243962 4.17080412620505e-07\\
0.0630375315307128 2.6633619412405e-07\\
0.0636219849265749 1.29633264739291e-07\\
0.0642118570906476 8.44807727474417e-08\\
0.0648071982631197 1.55321680632768e-07\\
0.0654080591499826 2.23323322901523e-07\\
0.0660144909273489 1.65558684106074e-07\\
0.0666265452458115 1.29602788378046e-07\\
0.0672442742348425 1.54082064525201e-07\\
0.067867730507233 1.9669514043918e-07\\
0.0684969671635745 6.56723972208827e-08\\
0.0691320377967813 2.72692872679481e-07\\
0.0697729964966554 1.64331771104573e-07\\
0.0704198978544929 1.92845749770284e-07\\
0.0710727969677342 3.14574493242459e-07\\
0.0717317494446562 2.23154461322846e-07\\
0.0723968114091088 3.19477336785647e-07\\
0.073068039505295 4.70239783466723e-07\\
0.0737454909025955 4.00435794718761e-07\\
0.0744292233004376 5.35390629016775e-07\\
0.0751192949332097 7.44227558961321e-07\\
0.0758157645752211 9.42260135453655e-07\\
0.0765186915457083 1.29126834526891e-06\\
0.0772281357138865 1.75334388719142e-06\\
0.0779441575040495 2.39619221694629e-06\\
0.0786668179007158 3.07935326205045e-06\\
0.0793961784538228 3.31725913686903e-06\\
0.0801323012839689 2.77723316343554e-06\\
0.0808752490877045 1.99270334782865e-06\\
0.0816250851428723 1.34244431595346e-06\\
0.082381873313996 9.06495684800187e-07\\
0.0831456780577206 6.21640207661689e-07\\
0.0839165644283016 4.33757396468366e-07\\
0.0846945980831458 3.16571443079364e-07\\
0.0854798452884041 2.57660603370085e-07\\
0.0862723729246145 1.85447024791562e-07\\
0.0870722484923992 1.83829787191873e-07\\
0.0878795401182132 2.06648859422627e-07\\
0.0886943165601471 2.71049727638563e-07\\
0.089516647213783 3.60568279704285e-07\\
0.0903466021181053 4.57186469917026e-07\\
0.0911842519614657 5.64501864087006e-07\\
0.0920296680876042 7.17662730269052e-07\\
0.092882922501725 9.61338378618389e-07\\
0.09374408787663 1.24917607621544e-06\\
0.0946132375589077 1.65174886600262e-06\\
0.0954904455751808 2.23727216536797e-06\\
0.0963757866384109 3.0776627379085e-06\\
0.0972693361542617 4.2066149531761e-06\\
0.0981711702275219 5.44047252661712e-06\\
0.0990813656685867 6.24095392937492e-06\\
0.1 5.97073808308578e-06\\
};
\end{axis}
\end{tikzpicture}%

%% file: figures/relatv_err_tchain_30.tex
%
%
%
%
\begin{tikzpicture}

\begin{axis}[%
width=\figwidth,
height=\figheight,
scale only axis,
separate axis lines,
every outer x axis line/.append style={darkgray!60!black},
every x tick label/.append style={font=\color{darkgray!60!black}},
xmode=log,
xmin=0.001,
xmax=0.1,
xminorticks=true,
xlabel={$\omega$},
every outer y axis line/.append style={darkgray!60!black},
every y tick label/.append style={font=\color{darkgray!60!black}},
ymode=log,
ymin=1e-14,
ymax=1e-09,
yminorticks=true,
ylabel={$\frac{\sigma{}_{\text{max}}\text{($\mathcal{T}$(j}\omega\text{)}-\hat{\text{$\mathcal{T}$}}\text{(j}\omega\text{))}}{\sigma{}_{\text{max}}\text{($\mathcal{T}$(j}\omega\text{))}}$},
]
\addplot [
color=red,
dash pattern=on 1pt off 3pt on 3pt off 3pt,
forget plot
]
table[row sep=crcr]{
0.001 1.63428015335392e-11\\
0.00100927151463057 1.12290239874345e-11\\
0.00101862899024469 7.68505884350074e-12\\
0.00102807322383086 5.8872315457313e-12\\
0.00103760501976691 5.39685539436571e-12\\
0.00104722518988843 4.14690637870843e-12\\
0.00105693455355799 1.17346622729759e-11\\
0.00106673393773486 1.16277104973684e-11\\
0.00107662417704549 1.68596951808517e-11\\
0.0010866061138546 8.50377127714693e-12\\
0.00109668059833687 3.18572483345117e-11\\
0.00110684848854941 1.51955399523027e-11\\
0.00111711065050482 1.23020547118745e-11\\
0.00112746795824495 4.51330691940934e-12\\
0.00113792129391532 1.49973122124762e-11\\
0.00114847154784029 3.86561573317098e-12\\
0.00115911961859889 1.58989779988248e-11\\
0.00116986641310131 1.60739424432764e-11\\
0.00118071284666619 3.51225884028614e-12\\
0.00119165984309856 8.23465040978989e-12\\
0.00120270833476851 1.6049198421068e-11\\
0.00121385926269063 5.47116916541392e-12\\
0.00122511357660412 9.23138759521778e-12\\
0.00123647223505371 1.99104346443393e-12\\
0.00124793620547131 1.11035692137256e-11\\
0.00125950646425836 1.73038434482138e-11\\
0.00127118399686903 2.93597259069016e-12\\
0.00128296979789415 8.84603525351197e-12\\
0.0012948648711459 1.09697788413829e-11\\
0.00130687022974335 2.60224768480017e-11\\
0.00131898689619867 1.77635042420805e-11\\
0.00133121590250431 1.16689746238505e-11\\
0.00134355829022083 8.86936325508143e-12\\
0.00135601511056563 1.26470131039866e-11\\
0.00136858742450252 1.16082797042876e-11\\
0.00138127630283201 1.61831725148455e-12\\
0.00139408282628258 4.1044618883962e-12\\
0.00140700808560269 2.15046010883672e-11\\
0.00142005318165368 1.22133512166179e-11\\
0.00143321922550357 1.04050242342518e-11\\
0.00144650733852165 1.45313892895493e-11\\
0.00145991865247398 8.56766245465208e-13\\
0.00147345430961984 4.73215364696646e-12\\
0.00148711546280895 7.34897314032284e-12\\
0.00150090327557974 2.60499779782637e-11\\
0.00151481892225835 3.53936561862121e-12\\
0.00152886358805873 3.16380786780641e-11\\
0.00154303846918356 8.12506059310044e-12\\
0.00155734477292614 2.06648479476278e-11\\
0.00157178371777316 1.75767026528954e-11\\
0.00158635653350859 6.68479064629256e-12\\
0.00160106446131832 5.87565738678266e-12\\
0.00161590875389592 8.86600127464108e-12\\
0.00163089067554933 1.12703072795076e-11\\
0.00164601150230855 7.85038992125623e-12\\
0.00166127252203429 7.49458943210828e-12\\
0.0016766750345277 1.03311767537183e-11\\
0.00169222035164104 1.42743437370809e-12\\
0.00170790979738943 1.91212684568517e-11\\
0.00172374470806362 1.19352588964928e-11\\
0.0017397264323438 3.90346523465477e-13\\
0.00175585633141447 1.7884987437133e-11\\
0.00177213577908036 2.52769509618886e-11\\
0.00178856616188346 1.90515804243792e-12\\
0.00180514887922111 1.16030511108985e-11\\
0.00182188534346517 1.13923434351819e-11\\
0.00183877698008233 1.52981018261451e-11\\
0.00185582522775552 8.41812847183492e-12\\
0.00187303153850644 9.3806039872616e-12\\
0.00189039737781922 7.10160064319061e-12\\
0.00190792422476527 2.30858750727415e-11\\
0.0019256135721292 1.61125326090928e-11\\
0.00194346692653602 1.56456236652458e-11\\
0.00196148580857943 9.71857129344062e-12\\
0.00197967175295133 1.66100546497154e-11\\
0.00199802630857255 2.4010252247779e-11\\
0.00201655103872475 9.1743068555605e-13\\
0.00203524752118358 1.98089845240887e-11\\
0.00205411734835306 1.51932628875725e-11\\
0.00207316212740123 1.46983170487776e-11\\
0.00209238348039698 1.44075433767134e-11\\
0.00211178304444824 2.22096315348502e-11\\
0.00213136247184144 1.01016922404362e-11\\
0.00215112343018217 7.11341243101383e-12\\
0.00217106760253726 1.24668094838515e-11\\
0.00219119668757815 1.44460292390695e-11\\
0.00221151239972549 1.49368145664887e-11\\
0.00223201646929523 7.26866858551543e-12\\
0.00225271064264598 4.19474503855319e-12\\
0.00227359668232772 2.30182236822807e-12\\
0.00229467636723194 1.75541141517654e-11\\
0.00231595149274315 3.61258527306522e-13\\
0.00233742387089181 2.24839774989003e-12\\
0.00235909533050864 1.68640978990009e-11\\
0.00238096771738036 5.25612594987381e-12\\
0.00240304289440697 2.03695148620591e-11\\
0.00242532274176035 1.26560560145428e-11\\
0.00244780915704444 4.83896824364985e-12\\
0.00247050405545683 1.26067356193892e-11\\
0.00249340936995188 4.41151107155958e-14\\
0.00251652705140539 7.16190243476234e-12\\
0.00253985906878073 8.25142303048418e-12\\
0.00256340740929651 5.7835446092034e-13\\
0.00258717407859592 3.27071163036366e-11\\
0.00261116110091746 1.80552508752788e-12\\
0.00263537051926739 4.44417169033821e-12\\
0.00265980439559376 8.90793829687996e-12\\
0.00268446481096197 1.32924358890892e-11\\
0.00270935386573205 1.01834397581156e-11\\
0.00273447367973758 2.0518725073689e-12\\
0.00275982639246618 1.44320375711406e-11\\
0.00278541416324177 5.64300315498551e-12\\
0.00281123917140846 6.44704470507068e-12\\
0.00283730361651621 1.57488713728883e-11\\
0.00286360971850812 1.25214426639656e-11\\
0.00289015971790951 1.2406251946796e-11\\
0.0029169558760188 1.62131367775765e-11\\
0.00294400047510003 3.30788215302697e-12\\
0.00297129581857733 2.16168905782457e-11\\
0.00299884423123103 6.34635126232708e-12\\
0.00302664805939569 1.64737531539045e-11\\
0.00305470967115997 1.79062089991998e-12\\
0.00308303145656828 7.77782196176293e-12\\
0.00311161582782436 9.13492173572056e-12\\
0.00314046521949675 2.65738212488346e-11\\
0.00316958208872612 1.84208214216179e-11\\
0.00319896891543454 6.94365186758503e-12\\
0.00322862820253673 7.91265285187632e-12\\
0.00325856247615323 6.61622781150628e-12\\
0.00328877428582551 5.38010783725099e-12\\
0.00331926620473319 4.5793272038063e-12\\
0.00335004082991313 9.83439958899401e-12\\
0.00338110078248069 1.07860827597123e-11\\
0.00341244870785289 1.64658285345993e-11\\
0.00344408727597382 4.25856565785147e-12\\
0.00347601918154198 3.23637661683848e-13\\
0.00350824714423979 2.04465768405742e-11\\
0.00354077390896527 1.3899716167252e-11\\
0.00357360224606579 1.21113731940035e-11\\
0.00360673495157403 2.98664116244504e-12\\
0.00364017484744614 1.79312524187953e-11\\
0.00367392478180207 4.76817019283851e-12\\
0.00370798762916817 5.5811755627436e-12\\
0.00374236629072198 5.37925415085087e-12\\
0.00377706369453937 3.98703236038732e-12\\
0.00381208279584389 2.93882571025713e-12\\
0.00384742657725851 9.81341466172325e-12\\
0.00388309804905961 3.17666662007567e-11\\
0.00391910024943341 4.03449424466994e-12\\
0.0039554362447347 2.47507975491986e-11\\
0.00399210912974805 6.75725956917425e-12\\
0.00402912202795135 8.03572741119453e-12\\
0.00406647809178186 1.64033807675481e-11\\
0.00410418050290471 7.48118762265489e-12\\
0.0041422324724839 2.24173426826096e-11\\
0.00418063724145576 7.85615622348786e-13\\
0.00421939808080502 1.35911895981802e-12\\
0.00425851829184341 8.68082085959578e-13\\
0.0042980012064908 9.36146721671625e-12\\
0.00433785018755899 1.98067767602296e-11\\
0.00437806862903817 2.01103478822173e-11\\
0.00441865995638594 1.27853666710756e-11\\
0.0044596276268191 2.03659769916871e-11\\
0.00450097512960805 3.82186474702346e-12\\
0.00454270598637404 4.77429216198584e-11\\
0.0045848237513891 1.15253798309571e-11\\
0.00462733201187869 1.39367681544064e-11\\
0.00467023438832734 3.33800404730506e-11\\
0.00471353453478691 7.75352346218909e-13\\
0.00475723613918789 2.02799427476673e-11\\
0.00480134292365346 2.73894007805737e-11\\
0.0048458586448165 8.54407797779578e-12\\
0.00489078709413959 1.78964022575589e-11\\
0.00493613209823792 1.22856633546435e-11\\
0.00498189751920516 2.33397247013103e-11\\
0.00502808725494248 1.3706980223793e-12\\
0.00507470523949047 1.40817069616565e-11\\
0.00512175544336424 4.89609217298993e-12\\
0.0051692418738916 1.47633688966986e-11\\
0.00521716857555435 1.95962731447092e-11\\
0.00526553963033276 1.1908748274101e-11\\
0.00531435915805324 9.38211133825548e-12\\
0.00536363131673924 6.29790436880105e-12\\
0.00541336030296538 1.36166655515433e-12\\
0.00546355035221488 5.03853095523205e-12\\
0.0055142057392403 2.67875720044065e-11\\
0.00556533077842765 1.74396706773038e-11\\
0.00561692982416381 2.86252155891634e-12\\
0.00566900727120743 1.57349410877829e-11\\
0.00572156755506325 1.4555654100382e-11\\
0.00577461515235982 1.83568511334893e-11\\
0.00582815458123085 2.12933435708393e-11\\
0.00588219040169996 5.0691445794446e-12\\
0.00593672721606913 3.77405466753381e-11\\
0.00599176966931062 2.5053338159721e-11\\
0.00604732244946265 1.31285226334248e-11\\
0.00610339028802862 8.280577484728e-12\\
0.00615997796038017 3.64363112875842e-13\\
0.00621709028616383 1.96000939358216e-11\\
0.00627473212971158 1.42008436507242e-11\\
0.00633290840045512 1.91887124455111e-11\\
0.00639162405334401 9.4462774489774e-12\\
0.00645088408926769 1.44789245155819e-11\\
0.00651069355548146 5.81975427920525e-12\\
0.00657105754603627 4.5656521260607e-12\\
0.00663198120221268 2.98195979959488e-12\\
0.00669346971295866 8.13103896556786e-12\\
0.00675552831533165 8.61627778452557e-12\\
0.00681816229494448 4.72785607529119e-12\\
0.00688137698641567 3.69158731252805e-11\\
0.0069451777738237 3.12001559771648e-11\\
0.00700957009116562 1.13585804011049e-11\\
0.00707455942281988 2.78254750544722e-12\\
0.0071401513040134 1.04746327683089e-11\\
0.00720635132129305 4.64870563715667e-11\\
0.00727316511300145 6.74723956429767e-12\\
0.00734059836975721 3.10261270590361e-11\\
0.00740865683493957 2.67113721869664e-11\\
0.00747734630517759 2.15200061688851e-11\\
0.00754667263084389 1.57379101081236e-11\\
0.00761664171655289 1.51659640006614e-11\\
0.00768725952166374 2.32476827605272e-11\\
0.00775853206078784 4.0472685345188e-11\\
0.00783046540430118 2.5628429992493e-11\\
0.00790306567886135 3.07533408145386e-11\\
0.00797633906792928 2.07162745957578e-11\\
0.00805029181229598 5.5541295029481e-12\\
0.00812493021061405 3.62255818165376e-11\\
0.00820026061993413 2.0715163593933e-11\\
0.00827628945624635 3.36483330884739e-11\\
0.00835302319502678 2.45227635611056e-11\\
0.00843046837178897 3.15315490998969e-11\\
0.00850863158264058 1.0960288970377e-11\\
0.00858751948484518 6.53905585672741e-11\\
0.00866713879738923 4.80613236429199e-11\\
0.00874749630155442 6.77909575704621e-12\\
0.00882859884149515 8.85336484480305e-11\\
0.00891045332482151 4.14879601097263e-11\\
0.00899306672318762 1.39649527154621e-10\\
0.00907644607288536 1.24922506788565e-10\\
0.00916059847544371 1.74397408878796e-11\\
0.00924553109823357 2.68312839634865e-11\\
0.00933125117507825 2.01496128439194e-10\\
0.00941776600686952 2.12748187131116e-11\\
0.00950508296218951 6.72748375786277e-11\\
0.00959320947793824 2.20857725654824e-10\\
0.00968215305996709 5.52601118269313e-11\\
0.009771921283718 1.09906499896877e-10\\
0.00986252179486878 1.2316400258553e-10\\
0.00995396230998423 3.64055778865571e-11\\
0.0100462506171734 1.9775292194939e-10\\
0.0101393945767529 2.57930576133931e-10\\
0.0102334021219164 1.16166414812175e-10\\
0.0103282812594103 4.56856460776132e-11\\
0.0104240400702156 4.55055176008824e-11\\
0.0105206867102362 1.02278891534528e-10\\
0.0106182294109938 2.84426618629971e-11\\
0.0107166764803286 5.46904793354827e-11\\
0.0108160363031071 1.2152729204415e-11\\
0.0109163173419361 7.94457487837295e-11\\
0.0110175281378839 6.9624092466e-11\\
0.011119677311207 4.35194875386434e-11\\
0.0112227735620851 4.90364937050642e-11\\
0.0113268256713615 5.88938900065944e-12\\
0.0114318425012915 3.39676503220296e-11\\
0.0115378329962966 4.08031715889654e-11\\
0.0116448061837269 3.28236247424872e-11\\
0.0117527711746295 4.17774277208122e-12\\
0.0118617371645248 3.62080151613426e-11\\
0.0119717134341897 1.98807398807897e-11\\
0.0120827093504478 3.21396752376281e-11\\
0.0121947343669674 9.03657618944793e-13\\
0.0123077980250667 2.72019320415622e-11\\
0.0124219099545262 1.96087250246773e-11\\
0.0125370798744092 2.50377663612165e-11\\
0.0126533175938894 8.27762604731558e-12\\
0.0127706330130864 4.81607879975336e-12\\
0.0128890361239089 1.55068040867424e-12\\
0.0130085370109057 1.05533236771723e-11\\
0.0131291458521247 6.48039435887945e-12\\
0.0132508729199795 9.29513534166396e-13\\
0.013373728582125 5.34031990042211e-12\\
0.0134977233023394 3.88026500135574e-12\\
0.0136228676414165 1.33486179473277e-11\\
0.0137491722580642 1.82267657105414e-11\\
0.0138766479098131 1.07609787393262e-11\\
0.0140053054539322 7.96438728133309e-12\\
0.014135155848354 7.88323274045105e-12\\
0.0142662101526074 1.51331592826083e-11\\
0.0143984795287601 1.39124242863518e-11\\
0.014531975242369 2.09396173238732e-11\\
0.0146667086634397 8.75802793946951e-12\\
0.0148026912673951 1.0889485414822e-11\\
0.0149399346360526 5.79707710525737e-12\\
0.0150784504586105 1.4588777102196e-11\\
0.0152182505326439 4.12678824920245e-13\\
0.015359346765109 9.36336853434784e-12\\
0.0155017511733577 3.92132069676065e-12\\
0.015645475886161 2.49373099975477e-12\\
0.0157905331447418 8.40679355744098e-13\\
0.0159369353038178 2.01718426146825e-12\\
0.0160846948326536 3.12273145404192e-12\\
0.0162338243161228 4.34384338798119e-12\\
0.0163843364557798 2.83739631582604e-12\\
0.0165362440709418 3.634141200361e-12\\
0.0166895600997802 5.60731468980213e-12\\
0.0168442976004231 2.05062621352673e-12\\
0.0170004697520672 8.72425749793491e-12\\
0.0171580898561 1.49230878893895e-11\\
0.0173171713372335 7.70794581267297e-12\\
0.0174777277446468 6.98543224352986e-13\\
0.0176397727531404 2.39922405765062e-12\\
0.0178033201643011 9.21207121007185e-12\\
0.0179683839076772 8.47253930402596e-12\\
0.018134978041965 2.62933760941475e-12\\
0.0183031167562061 2.94155127242398e-12\\
0.0184728143709963 3.71356046092243e-12\\
0.0186440853397049 4.48915131176514e-12\\
0.0188169442497056 5.80354473229406e-12\\
0.0189914058236193 1.55989382240171e-11\\
0.0191674849205681 1.15342722207549e-11\\
0.0193451965374405 1.65152174797363e-11\\
0.0195245558101686 1.88139156047041e-11\\
0.019705578015018 3.07687911056591e-11\\
0.0198882785698881 2.62838303719092e-11\\
0.0200726730356257 2.7230005165584e-11\\
0.0202587771173502 2.04156669965793e-11\\
0.0204466066657912 1.59921267517193e-11\\
0.0206361776786386 1.23012198602939e-11\\
0.0208275063019052 9.66478755342649e-12\\
0.0210206088313016 1.01618778515801e-11\\
0.0212155017136245 4.3746983122945e-12\\
0.0214122015481573 2.62280239813771e-12\\
0.0216107250880838 7.76787016162558e-12\\
0.0218110892419152 5.54869725188188e-12\\
0.0220133110749303 2.48513558987395e-12\\
0.0222174078106288 7.86733765834723e-13\\
0.0224233968321985 2.96334835006157e-12\\
0.0226312956839953 2.38096669626101e-12\\
0.0228411220730382 1.37239675666348e-12\\
0.0230528938705171 6.79664881049782e-12\\
0.0232666291133146 3.16598745862071e-12\\
0.0234823460055428 4.97274718780014e-12\\
0.0237000629200933 3.9165180207901e-12\\
0.0239197984002024 8.13157743082109e-13\\
0.0241415711610302 1.6355110076277e-12\\
0.0243654000912547 4.71848379108412e-12\\
0.0245913042546805 4.28261935454957e-12\\
0.0248193028918626 6.22986843927461e-12\\
0.025049415421745 6.10291963712566e-12\\
0.025281661443315 2.6392853339589e-12\\
0.0255160607372719 2.43146181835503e-12\\
0.025752633267712 1.0088104872897e-12\\
0.0259913991838293 5.89060870766874e-12\\
0.0262323788216312 9.07399294275375e-12\\
0.0264755927056707 6.29811699417462e-12\\
0.0267210615507944 1.10831884399429e-12\\
0.0269688062639069 1.43559815799141e-11\\
0.0272188479457518 1.43645973222267e-11\\
0.0274712078927081 4.47876639766272e-13\\
0.0277259075986048 1.44518955153308e-12\\
0.0279829687565512 1.13861315032172e-11\\
0.0282424132607843 1.32957105814908e-11\\
0.0285042632085343 1.47943226464784e-12\\
0.0287685409019059 4.69916984016732e-11\\
0.0290352688497781 8.51143555498999e-12\\
0.0293044697697214 1.42750547500434e-11\\
0.0295761665899325 1.22497373273776e-11\\
0.0298503824511873 2.8451635056977e-11\\
0.0301271407088116 1.95101307647927e-12\\
0.0304064649346707 1.26807006175596e-11\\
0.0306883789191764 1.94335606082433e-12\\
0.0309729066733141 7.01134407570018e-12\\
0.031260072430687 5.38285234833135e-12\\
0.0315499006495809 7.76906661570896e-12\\
0.0318424160150465 2.94564578793984e-12\\
0.0321376434410028 6.37681711242449e-13\\
0.0324356080723581 6.76620909815555e-12\\
0.0327363352871525 3.4456344655577e-12\\
0.0330398506987186 1.96104267113633e-12\\
0.0333461801578636 6.19468014508961e-13\\
0.0336553497550709 3.29483243196782e-12\\
0.0339673858227221 2.47582657410987e-13\\
0.0342823149373397 4.51530387287608e-12\\
0.0346001639218511 3.39621418935533e-12\\
0.0349209598478727 2.35234864985339e-12\\
0.0352447300380159 9.39038546015785e-13\\
0.0355715020682139 2.23285069821055e-12\\
0.0359013037700707 4.98294484654624e-13\\
0.0362341632332315 1.79623091192456e-12\\
0.0365701088077749 1.77029604601451e-12\\
0.0369091691066278 1.91346112422632e-12\\
0.0372513730080021 3.12243843054818e-12\\
0.0375967496578547 2.44032032125515e-12\\
0.0379453284723694 3.88466047232927e-12\\
0.0382971391404628 4.18931782320878e-12\\
0.0386522116263126 6.89485597585444e-12\\
0.0390105761719099 1.09514751941816e-11\\
0.0393722632996348 1.33645400840254e-11\\
0.0397373038148561 1.56922515920556e-11\\
0.040105728808555 1.50410987397273e-11\\
0.0404775696599732 1.10569632361494e-11\\
0.0408528580392856 7.89364738777672e-12\\
0.0412316259102975 4.32530110766149e-12\\
0.0416139055331671 2.77344963377762e-12\\
0.0419997294671531 2.36133155713419e-12\\
0.0423891305733878 1.88910778930516e-12\\
0.0427821420176762 2.90283476120512e-12\\
0.0431787972733202 9.09682074558781e-13\\
0.0435791301239703 3.11614891016819e-12\\
0.0439831746665023 1.45935789083486e-12\\
0.0443909653139217 2.2606027152657e-12\\
0.0448025367982949 1.98015628144589e-13\\
0.045217924173707 2.60175730806368e-12\\
0.0456371628192476 1.25903778930294e-12\\
0.046060288442024 8.61177626544539e-13\\
0.0464873370802026 6.56909543150082e-13\\
0.046918345106078 1.39246811700569e-12\\
0.0473533492291712 2.95748868237749e-12\\
0.047792386499356 1.21181589828756e-12\\
0.0482354943100147 6.35165806823727e-12\\
0.0486827104012228 7.71619282605332e-12\\
0.0491340728629636 2.09990112591609e-12\\
0.0495896201383721 1.86193626661432e-12\\
0.0500493910270095 1.83063169892014e-12\\
0.0505134246881676 3.61160159581332e-12\\
0.0509817606442042 4.85595699699603e-12\\
0.0514544387839093 9.05938793181902e-13\\
0.0519314993659021 3.58874680535688e-12\\
0.0524129830220606 5.99543222798325e-13\\
0.0528989307609815 2.9671743693354e-12\\
0.0533893839714735 2.04386508371588e-12\\
0.0538843844260822 1.20017986345224e-12\\
0.054383974284648 1.17631087233062e-12\\
0.0548881960978967 1.5557279923319e-12\\
0.055397092811064 2.21906962004817e-12\\
0.0559107077675529 2.30302353175908e-12\\
0.0564290847126254 3.69826806025164e-12\\
0.0569522677971283 4.17343849203768e-12\\
0.0574803015812535 6.20267167166146e-12\\
0.0580132310383338 8.3448526766957e-12\\
0.0585511015586724 1.18717762065214e-11\\
0.0590939589534097 1.68069205182179e-11\\
0.0596418494584246 1.76078107156957e-11\\
0.0601948197382727 1.43337659149923e-11\\
0.0607529168901607 9.58769101361499e-12\\
0.0613161884479577 6.35689927843024e-12\\
0.0618846823862439 4.69987038289271e-12\\
0.0624584471243962 3.41077350827048e-12\\
0.0630375315307128 2.24278032925876e-12\\
0.0636219849265749 1.12519257364935e-12\\
0.0642118570906476 7.56792837162645e-13\\
0.0648071982631197 1.43831423003259e-12\\
0.0654080591499826 2.14178214304557e-12\\
0.0660144909273489 1.64751645518079e-12\\
0.0666265452458115 1.3393197817794e-12\\
0.0672442742348425 1.647615065436e-12\\
0.067867730507233 2.14000545979006e-12\\
0.0684969671635745 6.95598955018717e-13\\
0.0691320377967813 2.65061202052974e-12\\
0.0697729964966554 1.45867920549106e-12\\
0.0704198978544929 1.65363023334446e-12\\
0.0710727969677342 2.73202633564834e-12\\
0.0717317494446562 2.00129404536552e-12\\
0.0723968114091088 2.9725148711386e-12\\
0.073068039505295 4.53729748924972e-12\\
0.0737454909025955 3.99979632662542e-12\\
0.0744292233004376 5.5257436571007e-12\\
0.0751192949332097 7.92383164971738e-12\\
0.0758157645752211 1.03356652479506e-11\\
0.0765186915457083 1.45772509109393e-11\\
0.0772281357138865 2.03553627356605e-11\\
0.0779441575040495 2.8591642038777e-11\\
0.0786668179007158 3.77498901436698e-11\\
0.0793961784538228 4.17714250044972e-11\\
0.0801323012839689 3.59193596644942e-11\\
0.0808752490877045 2.64736860254233e-11\\
0.0816250851428723 1.83244282229558e-11\\
0.082381873313996 1.2718573268803e-11\\
0.0831456780577206 8.97000517144353e-12\\
0.0839165644283016 6.44124815446442e-12\\
0.0846945980831458 4.84092167377125e-12\\
0.0854798452884041 4.05757233851271e-12\\
0.0862723729246145 3.00213446821089e-12\\
0.0870722484923992 3.03892559387907e-12\\
0.0878795401182132 3.43105654530783e-12\\
0.0886943165601471 4.40356689904702e-12\\
0.089516647213783 5.64921910186152e-12\\
0.0903466021181053 7.01556232983675e-12\\
0.0911842519614657 8.70332819337826e-12\\
0.0920296680876042 1.12945584508416e-11\\
0.092882922501725 1.55422682612595e-11\\
0.09374408787663 2.07822229794129e-11\\
0.0946132375589077 2.82759896331584e-11\\
0.0954904455751808 3.93859580423496e-11\\
0.0963757866384109 5.56808298498585e-11\\
0.0972693361542617 7.81687066058283e-11\\
0.0981711702275219 1.03792662082256e-10\\
0.0990813656685867 1.22205073973603e-10\\
0.1 1.19981264438356e-10\\
};
\end{axis}
\end{tikzpicture}%

%% file: figures/abser_bt_irka30.tex
%
%
%
%
\begin{tikzpicture}

\begin{axis}[%
width=\figwidth,
height=\figheight,
scale only axis,
separate axis lines,
every outer x axis line/.append style={darkgray!60!black},
every x tick label/.append style={font=\color{darkgray!60!black}},
xmode=log,
xmin=0.001,
xmax=0.1,
xminorticks=true,
xlabel={$\omega$},
every outer y axis line/.append style={darkgray!60!black},
every y tick label/.append style={font=\color{darkgray!60!black}},
ymode=log,
ymin=1e-08,
ymax=0.1,
yminorticks=true,
ylabel={$\sigma{}_{\text{max}}\text{($\mathcal{T}$(j}\omega\text{)-$\hat{\mathcal{T}}$}\text{(j}\omega\text{))}$},
legend columns=2,
legend style={nodes=right,anchor=south east,at={(.7,1.05)}}
]
\addplot [
color=red,
dash pattern=on 1pt off 3pt on 3pt off 3pt,
]
table[row sep=crcr]{
0.001 6.90385036515173e-05\\
0.00100927151463057 4.74447706144254e-05\\
0.00101862899024469 3.24770482967033e-05\\
0.00102807322383086 2.48842780969029e-05\\
0.00103760501976691 2.28160719992163e-05\\
0.00104722518988843 1.75352563309028e-05\\
0.00105693455355799 4.96304174182332e-05\\
0.00106673393773486 4.91883974429915e-05\\
0.00107662417704549 7.13363738172837e-05\\
0.0010866061138546 3.59888094363642e-05\\
0.00109668059833687 0.000134852981245945\\
0.00110684848854941 6.43378623775086e-05\\
0.00111711065050482 5.20988571817301e-05\\
0.00112746795824495 1.91182154113036e-05\\
0.00113792129391532 6.3543302032384e-05\\
0.00114847154784029 1.63825254526405e-05\\
0.00115911961859889 6.73967884596991e-05\\
0.00116986641310131 6.81557124908032e-05\\
0.00118071284666619 1.48962963071154e-05\\
0.00119165984309856 3.49342162853897e-05\\
0.00120270833476851 6.81044343869777e-05\\
0.00121385926269063 2.32231212762638e-05\\
0.00122511357660412 3.91947659390984e-05\\
0.00123647223505371 8.45599572560023e-06\\
0.00124793620547131 4.71706508811303e-05\\
0.00125950646425836 7.35325223725206e-05\\
0.00127118399686903 1.24801246041238e-05\\
0.00128296979789415 3.76138762656695e-05\\
0.0012948648711459 4.66586650788911e-05\\
0.00130687022974335 0.000110718613482116\\
0.00131898689619867 7.56033135539392e-05\\
0.00133121590250431 4.96807013907356e-05\\
0.00134355829022083 3.77740031521433e-05\\
0.00135601511056563 5.38811488760697e-05\\
0.00136858742450252 4.94729568041653e-05\\
0.00138127630283201 6.899500674529e-06\\
0.00139408282628258 1.75052027425839e-05\\
0.00140700808560269 9.17491840550826e-05\\
0.00142005318165368 5.21277054200292e-05\\
0.00143321922550357 4.44265817887541e-05\\
0.00144650733852165 6.20691991061839e-05\\
0.00145991865247398 3.66103409393106e-06\\
0.00147345430961984 2.02290723164932e-05\\
0.00148711546280895 3.14284412147799e-05\\
0.00150090327557974 0.000111451522554818\\
0.00151481892225835 1.51492110492491e-05\\
0.00152886358805873 0.000135476565871577\\
0.00154303846918356 3.48075769280653e-05\\
0.00155734477292614 8.85678678002713e-05\\
0.00157178371777316 7.53671221405992e-05\\
0.00158635653350859 2.86771983716748e-05\\
0.00160106446131832 2.52181756145277e-05\\
0.00161590875389592 3.80712613135988e-05\\
0.00163089067554933 4.84196416035945e-05\\
0.00164601150230855 3.37440745217939e-05\\
0.00166127252203429 3.22313757473577e-05\\
0.0016766750345277 4.44538929185689e-05\\
0.00169222035164104 6.14539156736161e-06\\
0.00170790979738943 8.2366002581171e-05\\
0.00172374470806362 5.14405474773462e-05\\
0.0017397264323438 1.68333734722072e-06\\
0.00175585633141447 7.71722810571551e-05\\
0.00177213577908036 0.000109132494496059\\
0.00178856616188346 8.2304208536334e-06\\
0.00180514887922111 5.01568110079896e-05\\
0.00182188534346517 4.92768109754205e-05\\
0.00183877698008233 6.62131314336998e-05\\
0.00185582522775552 3.64589812875086e-05\\
0.00187303153850644 4.06544073454623e-05\\
0.00189039737781922 3.07982819960513e-05\\
0.00190792422476527 0.000100187991665013\\
0.0019256135721292 6.99742109297222e-05\\
0.00194346692653602 6.79951369514302e-05\\
0.00196148580857943 4.22672693518714e-05\\
0.00197967175295133 7.22929193796129e-05\\
0.00199802630857255 0.000104580478665808\\
0.00201655103872475 3.99910454449527e-06\\
0.00203524752118358 8.64159329039833e-05\\
0.00205411734835306 6.63332735718389e-05\\
0.00207316212740123 6.42249160647436e-05\\
0.00209238348039698 6.30069348690275e-05\\
0.00211178304444824 9.72096367943096e-05\\
0.00213136247184144 4.4252607049393e-05\\
0.00215112343018217 3.11893868841421e-05\\
0.00217106760253726 5.47111444541864e-05\\
0.00219119668757815 6.345535205793e-05\\
0.00221151239972549 6.56726996638209e-05\\
0.00223201646929523 3.19887183643402e-05\\
0.00225271064264598 1.84786729668506e-05\\
0.00227359668232772 1.01500590185239e-05\\
0.00229467636723194 7.74846571709839e-05\\
0.00231595149274315 1.5962598933652e-06\\
0.00233742387089181 9.94526076040743e-06\\
0.00235909533050864 7.46745664233876e-05\\
0.00238096771738036 2.32997431451672e-05\\
0.00240304289440697 9.03964208796835e-05\\
0.00242532274176035 5.62294350426057e-05\\
0.00244780915704444 2.15239853959181e-05\\
0.00247050405545683 5.61419064406741e-05\\
0.00249340936995188 1.96696530165676e-07\\
0.00251652705140539 3.19722334917572e-05\\
0.00253985906878073 3.68824137491194e-05\\
0.00256340740929651 2.58845942781457e-06\\
0.00258717407859592 0.000146574166036517\\
0.00261116110091746 8.10210671723397e-06\\
0.00263537051926739 1.9969910095573e-05\\
0.00265980439559376 4.00834572243835e-05\\
0.00268446481096197 5.9897312251828e-05\\
0.00270935386573205 4.59541147512459e-05\\
0.00273447367973758 9.27299546882404e-06\\
0.00275982639246618 6.53205740368549e-05\\
0.00278541416324177 2.55798786154304e-05\\
0.00281123917140846 2.92703676124804e-05\\
0.00283730361651621 7.16159898215949e-05\\
0.00286360971850812 5.70324541306169e-05\\
0.00289015971790951 5.66017402774308e-05\\
0.0029169558760188 7.40955996303551e-05\\
0.00294400047510003 1.51435115390916e-05\\
0.00297129581857733 9.91371261381732e-05\\
0.00299884423123103 2.91574425772176e-05\\
0.00302664805939569 7.58256486922549e-05\\
0.00305470967115997 8.25737259077678e-06\\
0.00308303145656828 3.59358270825693e-05\\
0.00311161582782436 4.22885733759528e-05\\
0.00314046521949675 0.000123264579495441\\
0.00316958208872612 8.56204083773039e-05\\
0.00319896891543454 3.23413832785082e-05\\
0.00322862820253673 3.69329392565622e-05\\
0.00325856247615323 3.09487065978867e-05\\
0.00328877428582551 2.52221995972117e-05\\
0.00331926620473319 2.15166008205684e-05\\
0.00335004082991313 4.63148645961675e-05\\
0.00338110078248069 5.09164234967943e-05\\
0.00341244870785289 7.79149595440072e-05\\
0.00344408727597382 2.02006693481314e-05\\
0.00347601918154198 1.53903739610718e-06\\
0.00350824714423979 9.7481347111371e-05\\
0.00354077390896527 6.64417834700164e-05\\
0.00357360224606579 5.80480159518148e-05\\
0.00360673495157403 1.43535879964205e-05\\
0.00364017484744614 8.64165415777465e-05\\
0.00367392478180207 2.30447946974041e-05\\
0.00370798762916817 2.70525533426025e-05\\
0.00374236629072198 2.61513085979026e-05\\
0.00377706369453937 1.94418631467127e-05\\
0.00381208279584389 1.43749843713701e-05\\
0.00384742657725851 4.8153572891606e-05\\
0.00388309804905961 0.000156381345521225\\
0.00391910024943341 1.9926830691345e-05\\
0.0039554362447347 0.00012266080829754\\
0.00399210912974805 3.36037076997395e-05\\
0.00402912202795135 4.0102835180573e-05\\
0.00406647809178186 8.21581603593672e-05\\
0.00410418050290471 3.76088807101414e-05\\
0.0041422324724839 0.000113120807066265\\
0.00418063724145576 3.97964108991184e-06\\
0.00421939808080502 6.9120139358384e-06\\
0.00425851829184341 4.43261678481185e-06\\
0.0042980012064908 4.79993700152091e-05\\
0.00433785018755899 0.000101985562206066\\
0.00437806862903817 0.000103996783578774\\
0.00441865995638594 6.64098151531641e-05\\
0.0044596276268191 0.000106264355571572\\
0.00450097512960805 2.00339561509443e-05\\
0.00454270598637404 0.000251452804547063\\
0.0045848237513891 6.09968210122785e-05\\
0.00462733201187869 7.41256044531827e-05\\
0.00467023438832734 0.000178442555738715\\
0.00471353453478691 4.16647474259313e-06\\
0.00475723613918789 0.000109559149631003\\
0.00480134292365346 0.000148775885798711\\
0.0048458586448165 4.66703187994838e-05\\
0.00489078709413959 9.83164154586871e-05\\
0.00493613209823792 6.78898162894094e-05\\
0.00498189751920516 0.000129750880737646\\
0.00502808725494248 7.66707507595294e-06\\
0.00507470523949047 7.92653488623192e-05\\
0.00512175544336424 2.77387381470763e-05\\
0.0051692418738916 8.41981827636559e-05\\
0.00521716857555435 0.000112523758787315\\
0.00526553963033276 6.8859856107138e-05\\
0.00531435915805324 5.46396687683757e-05\\
0.00536363131673924 3.69480770259827e-05\\
0.00541336030296538 8.04893431268796e-06\\
0.00546355035221488 3.00144050282072e-05\\
0.0055142057392403 0.000160844670646765\\
0.00556533077842765 0.000105572678462004\\
0.00561692982416381 1.74742130524081e-05\\
0.00566900727120743 9.68832379220689e-05\\
0.00572156755506325 9.04175098432491e-05\\
0.00577461515235982 0.000115070145747888\\
0.00582815458123085 0.000134729376028249\\
0.00588219040169996 3.23833296845604e-05\\
0.00593672721606913 0.000243490580212758\\
0.00599176966931062 0.000163286421509648\\
0.00604732244946265 8.64647268370605e-05\\
0.00610339028802862 5.51258901261843e-05\\
0.00615997796038017 2.45267412697136e-06\\
0.00621709028616383 0.000133449912756399\\
0.00627473212971158 9.78316100781575e-05\\
0.00633290840045512 0.000133805001905096\\
0.00639162405334401 6.66978624692589e-05\\
0.00645088408926769 0.000103557743501364\\
0.00651069355548146 4.21816798702924e-05\\
0.00657105754603627 3.35491015655144e-05\\
0.00663198120221268 2.2224558533341e-05\\
0.00669346971295866 6.1494204708299e-05\\
0.00675552831533165 6.61573061057624e-05\\
0.00681816229494448 3.68736272138121e-05\\
0.00688137698641567 0.000292612460642169\\
0.0069451777738237 0.00025148430413956\\
0.00700957009116562 9.31561836087981e-05\\
0.00707455942281988 2.32346642345208e-05\\
0.0071401513040134 8.91099285512125e-05\\
0.00720635132129305 0.000403195664758978\\
0.00727316511300145 5.97071683667637e-05\\
0.00734059836975721 0.000280338540820531\\
0.00740865683493957 0.000246640581511208\\
0.00747734630517759 0.000203237348565265\\
0.00754667263084389 0.000152161231674196\\
0.00761664171655289 0.000150262707881985\\
0.00768725952166374 0.00023628771830477\\
0.00775853206078784 0.000422464286895526\\
0.00783046540430118 0.000275065214417001\\
0.00790306567886135 0.000339818441375824\\
0.00797633906792928 0.000235994642024595\\
0.00805029181229598 6.53251874559714e-05\\
0.00812493021061405 0.000440595783013306\\
0.00820026061993413 0.000260983072899537\\
0.00827628945624635 0.000439929512211997\\
0.00835302319502678 0.000333383982292162\\
0.00843046837178897 0.000446690274521139\\
0.00850863158264058 0.000162171915579288\\
0.00858751948484518 0.00101309680871894\\
0.00866713879738923 0.000781795846222438\\
0.00874749630155442 0.000116119634878574\\
0.00882859884149515 0.00160195796076825\\
0.00891045332482151 0.000795687836199492\\
0.00899306672318762 0.00284900564670163\\
0.00907644607288536 0.00272109883716352\\
0.00916059847544371 0.000407118416163473\\
0.00924553109823357 0.000673665183163315\\
0.00933125117507825 0.00545746574549975\\
0.00941776600686952 0.000622724218331617\\
0.00950508296218951 0.00212730545201725\\
0.00959320947793824 0.00751483637866595\\
0.00968215305996709 0.00200447998610007\\
0.009771921283718 0.004181369415759\\
0.00986252179486878 0.00479971999732526\\
0.00995396230998423 0.00141182438505409\\
0.0100462506171734 0.00740825193400624\\
0.0101393945767529 0.00909809403086843\\
0.0102334021219164 0.00378703227149064\\
0.0103282812594103 0.00136097539738158\\
0.0104240400702156 0.00123207539156894\\
0.0105206867102362 0.00251360954820955\\
0.0106182294109938 0.000635283443434257\\
0.0107166764803286 0.00111319045539153\\
0.0108160363031071 0.000226194243447486\\
0.0109163173419361 0.00135719202773661\\
0.0110175281378839 0.00109576702804042\\
0.011119677311207 0.000633295473505461\\
0.0112227735620851 0.000662065697090412\\
0.0113268256713615 7.40135940943756e-05\\
0.0114318425012915 0.000398536830504105\\
0.0115378329962966 0.000448192106771432\\
0.0116448061837269 0.000338406029783322\\
0.0117527711746295 4.05231325709837e-05\\
0.0118617371645248 0.0003311
64
50580677257\\
0.0119717134341897 0.000171787069784344\\
0.0120827093504478 0.000262874494274215\\
0.0121947343669674 7.00828939486118e-06\\
0.0123077980250667 0.000200358028238962\\
0.0124219099545262 0.000137373683176302\\
0.0125370798744092 0.000167070694181732\\
0.0126533175938894 5.26772809457517e-05\\
0.0127706330130864 2.92650282775968e-05\\
0.0128890361239089 9.00754190839456e-06\\
0.0130085370109057 5.86626161814975e-05\\
0.0131291458521247 3.45058450221388e-05\\
0.0132508729199795 4.74537529883049e-06\\
0.013373728582125 2.6162910213367e-05\\
0.0134977233023394 1.82575314205774e-05\\
0.0136228676414165 6.03693051290682e-05\\
0.0137491722580642 7.92878939553211e-05\\
0.0138766479098131 4.50574806986701e-05\\
0.0140053054539322 3.21195425670283e-05\\
0.014135155848354 3.06402166413863e-05\\
0.0142662101526074 5.6720746293415e-05\\
0.0143984795287601 5.03130889404644e-05\\
0.014531975242369 7.31040128437047e-05\\
0.0146667086634397 2.95317956738614e-05\\
0.0148026912673951 3.54819647885836e-05\\
0.0149399346360526 1.82608716248083e-05\\
0.0150784504586105 4.44456238925141e-05\\
0.0152182505326439 1.216455768938e-06\\
0.015359346765109 2.67151646499113e-05\\
0.0155017511733577 1.08332555536988e-05\\
0.015645475886161 6.6731137204762e-06\\
0.0157905331447418 2.17974406012983e-06\\
0.0159369353038178 5.06936582240683e-06\\
0.0160846948326536 7.60861374260042e-06\\
0.0162338243161228 1.02643447911372e-05\\
0.0163843364557798 6.50399620996574e-06\\
0.0165362440709418 8.08306822569037e-06\\
0.0166895600997802 1.21045499553274e-05\\
0.0168442976004231 4.29733212416131e-06\\
0.0170004697520672 1.77522542052659e-05\\
0.0171580898561 2.9490687878718e-05\\
0.0173171713372335 1.47961711853211e-05\\
0.0174777277446468 1.30276669272466e-06\\
0.0176397727531404 4.34790418504444e-06\\
0.0178033201643011 1.62244647973173e-05\\
0.0179683839076772 1.45042013039106e-05\\
0.018134978041965 4.37574487783706e-06\\
0.0183031167562061 4.75949461831675e-06\\
0.0184728143709963 5.84254580197181e-06\\
0.0186440853397049 6.86825698555148e-06\\
0.0188169442497056 8.63544477434147e-06\\
0.0189914058236193 2.25750530616088e-05\\
0.0191674849205681 1.62365718182103e-05\\
0.0193451965374405 2.26142124157756e-05\\
0.0195245558101686 2.5060309727883e-05\\
0.019705578015018 3.98693342372041e-05\\
0.0198882785698881 3.31316966933706e-05\\
0.0200726730356257 3.33908902069019e-05\\
0.0202587771173502 2.43534860750963e-05\\
0.0204466066657912 1.85569221964852e-05\\
0.0206361776786386 1.38844050626118e-05\\
0.0208275063019052 1.06101068932507e-05\\
0.0210206088313016 1.08495034226009e-05\\
0.0212155017136245 4.54195550474433e-06\\
0.0214122015481573 2.64764099404412e-06\\
0.0216107250880838 7.622940593224e-06\\
0.0218110892419152 5.29241825223623e-06\\
0.0220133110749303 2.30333474985194e-06\\
0.0222174078106288 7.08376308086352e-07\\
0.0224233968321985 2.59128883045285e-06\\
0.0226312956839953 2.0213001677121e-06\\
0.0228411220730382 1.13064483031819e-06\\
0.0230528938705171 5.43136423017929e-06\\
0.0232666291133146 2.45278465715248e-06\\
0.0234823460055428 3.73264425313676e-06\\
0.0237000629200933 2.84632051404964e-06\\
0.0239197984002024 5.71699816990908e-07\\
0.0241415711610302 1.11133899177964e-06\\
0.0243654000912547 3.09541727783026e-06\\
0.0245913042546805 2.70890711827182e-06\\
0.0248193028918626 3.79384505332401e-06\\
0.025049415421745 3.5717879674416e-06\\
0.025281661443315 1.48139000024522e-06\\
0.0255160607372719 1.30555709899566e-06\\
0.025752633267712 5.16621815712196e-07\\
0.0259913991838293 2.86662624158222e-06\\
0.0262323788216312 4.1775750
64
2166838e-06\\
0.0264755927056707 2.72828675713705e-06\\
0.0267210615507944 4.48767519747531e-07\\
0.0269688062639069 5.39047609570237e-06\\
0.0272188479457518 4.95722127845139e-06\\
0.0274712078927081 1.40862082407553e-07\\
0.0277259075986048 4.132020621126e-07\\
0.0279829687565512 3.01502267406864e-06\\
0.0282424132607843 3.49470797520853e-06\\
0.0285042632085343 4.41364631400448e-07\\
0.0287685409019059 1.82804867155107e-05\\
0.0290352688497781 4.65488215995714e-06\\
0.0293044697697214 1.10692130755494e-05\\
0.0295761665899325 1.27200758608475e-05\\
0.0298503824511873 3.47397438689306e-05\\
0.0301271407088116 2.41778833985894e-06\\
0.0304064649346707 1.46343765677741e-05\\
0.0306883789191764 2.03480413605629e-06\\
0.0309729066733141 6.6659916532056e-06\\
0.031260072430687 4.68753372754215e-06\\
0.0315499006495809 6.25692132744021e-06\\
0.0318424160150465 2.21294679996715e-06\\
0.0321376434410028 4.50109981368996e-07\\
0.0324356080723581 4.51383071533073e-06\\
0.0327363352871525 2.18295780591069e-06\\
0.0330398506987186 1.18454864221312e-06\\
0.0333461801578636 3.5792107944311e-07\\
0.0336553497550709 1.82590284334873e-06\\
0.0339673858227221 1.31893138071148e-07\\
0.0342823149373397 2.31672036405066e-06\\
0.0346001639218511 1.68099881040264e-06\\
0.0349209598478727 1.1247532524003e-06\\
0.0352447300380159 4.34244025457581e-07\\
0.0355715020682139 9.99642303543057e-07\\
0.0359013037700707 2.1616443431554e-07\\
0.0362341632332315 7.5561423883878e-07\\
0.0365701088077749 7.22612043132286e-07\\
0.0369091691066278 7.58302296009515e-07\\
0.0372513730080021 1.20195230845166e-06\\
0.0375967496578547 9.12823192851975e-07\\
0.0379453284723694 1.4124941884781e-06\\
0.0382971391404628 1.4811243823179e-06\\
0.0386522116263126 2.37074005912982e-06\\
0.0390105761719099 3.66280488178884e-06\\
0.0393722632996348 4.3483643950458e-06\\
0.0397373038148561 4.96721066916682e-06\\
0.040105728808555 4.63193643512867e-06\\
0.0404775696599732 3.31244971354203e-06\\
0.0408528580392856 2.30021613615491e-06\\
0.0412316259102975 1.22574876202784e-06\\
0.0416139055331671 7.64156088633155e-07\\
0.0419997294671531 6.32322481092667e-07\\
0.0423891305733878 4.91422915175944e-07\\
0.0427821420176762 7.33128811552133e-07\\
0.0431787972733202 2.22885957635513e-07\\
0.0435791301239703 7.40014606101021e-07\\
0.0439831746665023 3.35509331982242e-07\\
0.0443909653139217 5.02400648157949e-07\\
0.0448025367982949 4.24625553757165e-08\\
0.045217924173707 5.370820896635e-07\\
0.0456371628192476 2.49458805553463e-07\\
0.046060288442024 1.63166895657442e-07\\
0.0464873370802026 1.18486472840771e-07\\
0.046918345106078 2.37912176199918e-07\\
0.0473533492291712 4.76912132331145e-07\\
0.047792386499356 1.85203079792932e-07\\
0.0482354943100147 9.49276153412348e-07\\
0.0486827104012228 1.23646709688948e-06\\
0.0491340728629636 4.11059822589905e-07\\
0.0495896201383721 4.56761543984231e-07\\
0.0500493910270095 5.01532516446463e-07\\
0.0505134246881676 9.87892234894081e-07\\
0.0509817606442042 1.26620304333951e-06\\
0.0514544387839093 2.23049869387183e-07\\
0.0519314993659021 8.36220795022215e-07\\
0.0524129830220606 1.328574369961e-07\\
0.0528989307609815 6.28271581165768e-07\\
0.0533893839714735 4.15164000130539e-07\\
0.0538843844260822 2.34619382759069e-07\\
0.054383974284648 2.21866502217146e-07\\
0.0548881960978967 2.836807391082e-07\\
0.055397092811064 3.91822346533925e-07\\
0.0559107077675529 3.9426943893234e-07\\
0.0564290847126254 6.14485051112353e-07\\
0.0569522677971283 6.73558640850616e-07\\
0.0574803015812535 9.72977933588191e-07\\
0.0580132310383338 1.27290298707284e-06\\
0.0585511015586724 1.76156013873038e-06\\
0.0590939589534097 2.42647509948264e-06\\
0.0596418494584246 2.47371918846209e-06\\
0.0601948197382727 1.95958945417556e-06\\
0.0607529168901607 1.27535189930332e-06\\
0.0613161884479577 8.22562193438492e-07\\
0.0618846823862439 5.91358047579053e-07\\
0.0624584471243962 4.17080412620505e-07\\
0.0630375315307128 2.6633619412405e-07\\
0.0636219849265749 1.29633264739291e-07\\
0.0642118570906476 8.44807727474417e-08\\
0.0648071982631197 1.55321680632768e-07\\
0.0654080591499826 2.23323322901523e-07\\
0.0660144909273489 1.65558684106074e-07\\
0.0666265452458115 1.29602788378046e-07\\
0.0672442742348425 1.54082064525201e-07\\
0.067867730507233 1.9669514043918e-07\\
0.0684969671635745 6.56723972208827e-08\\
0.0691320377967813 2.72692872679481e-07\\
0.0697729964966554 1.64331771104573e-07\\
0.0704198978544929 1.92845749770284e-07\\
0.0710727969677342 3.14574493242459e-07\\
0.0717317494446562 2.23154461322846e-07\\
0.0723968114091088 3.19477336785647e-07\\
0.073068039505295 4.70239783466723e-07\\
0.0737454909025955 4.00435794718761e-07\\
0.0744292233004376 5.35390629016775e-07\\
0.0751192949332097 7.44227558961321e-07\\
0.0758157645752211 9.42260135453655e-07\\
0.0765186915457083 1.29126834526891e-06\\
0.0772281357138865 1.75334388719142e-06\\
0.0779441575040495 2.39619221694629e-06\\
0.0786668179007158 3.07935326205045e-06\\
0.0793961784538228 3.31725913686903e-06\\
0.0801323012839689 2.77723316343554e-06\\
0.0808752490877045 1.99270334782865e-06\\
0.0816250851428723 1.34244431595346e-06\\
0.082381873313996 9.06495684800187e-07\\
0.0831456780577206 6.21640207661689e-07\\
0.0839165644283016 4.33757396468366e-07\\
0.0846945980831458 3.16571443079364e-07\\
0.0854798452884041 2.57660603370085e-07\\
0.0862723729246145 1.85447024791562e-07\\
0.0870722484923992 1.83829787191873e-07\\
0.0878795401182132 2.06648859422627e-07\\
0.0886943165601471 2.71049727638563e-07\\
0.089516647213783 3.60568279704285e-07\\
0.0903466021181053 4.57186469917026e-07\\
0.0911842519614657 5.64501864087006e-07\\
0.0920296680876042 7.17662730269052e-07\\
0.092882922501725 9.61338378618389e-07\\
0.09374408787663 1.24917607621544e-06\\
0.0946132375589077 1.65174886600262e-06\\
0.0954904455751808 2.23727216536797e-06\\
0.0963757866384109 3.0776627379085e-06\\
0.0972693361542617 4.2066149531761e-06\\
0.0981711702275219 5.44047252661712e-06\\
0.0990813656685867 6.24095392937492e-06\\
0.1 5.97073808308578e-06\\
};
\addlegendentry{IRKA};
\addplot [
color=blue,
dashed,
]
table[row sep=crcr]{
0.001 0.000461458117094382\\
0.00101393945767529 0.000454711753511606\\
0.00102807322383086 0.000370431536604348\\
0.00104240400702156 0.000492054989985234\\
0.00105693455355799 0.000347248830474604\\
0.00107166764803286 0.000402762560609194\\
0.0010866061138546 0.000430413077367777\\
0.00110175281378839 0.000442676176248256\\
0.00111711065050482 0.000345804705998803\\
0.00113268256713615 0.000398008629049281\\
0.00114847154784029 0.000411952868023056\\
0.00116448061837269 0.000446805192500599\\
0.00118071284666619 0.000411031148558455\\
0.00119717134341897 0.000380767967060846\\
0.00121385926269063 0.000375044284785786\\
0.00123077980250667 0.000398673649221659\\
0.00124793620547131 0.000443062477263449\\
0.00126533175938894 0.000404505785440981\\
0.00128296979789415 0.000362816134420502\\
0.00130085370109057 0.000357352608671689\\
0.00131898689619867 0.000327960977835081\\
0.0013373728582125 0.000437419578585986\\
0.00135601511056563 0.000450688446800457\\
0.00137491722580642 0.000367590692623658\\
0.00139408282628258 0.000416674662118183\\
0.0014135155848354 0.000442833793712807\\
0.00143321922550357 0.000442875305982618\\
0.0014531975242369 0.000364815774764033\\
0.00147345430961984 0.000420769060208893\\
0.00149399346360526 0.000428445473966898\\
0.00151481892225835 0.0004166423731743\\
0.0015359346765109 0.000373242045836789\\
0.00155734477292614 0.000323127024909482\\
0.00157905331447418 0.000365475236685655\\
0.00160106446131832 0.000427682765850909\\
0.00162338243161228 0.000499178900663579\\
0.00164601150230855 0.000374629151877115\\
0.00166895600997802 0.000425295753968728\\
0.00169222035164104 0.000411733080464371\\
0.00171580898561 0.000382138634130383\\
0.0017397264323438 0.000405622772485742\\
0.00176397727531405 0.000475863018190186\\
0.00178856616188346 0.000415769960034352\\
0.0018134978041965 0.000482181956704733\\
0.00183877698008233 0.00035158442326389\\
0.00186440853397049 0.000385375584656335\\
0.00189039737781922 0.000438638693369989\\
0.00191674849205682 0.000373196149690419\\
0.00194346692653602 0.000353818146719272\\
0.0019705578015018 0.000426048144558183\\
0.00199802630857255 0.000326310445150414\\
0.00202587771173502 0.000365119531102467\\
0.00205411734835306 0.000474216943313655\\
0.00208275063019051 0.00048787533738999\\
0.00211178304444824 0.000337389290433168\\
0.00214122015481573 0.000471173433690845\\
0.00217106760253726 0.000373412140964073\\
0.00220133110749303 0.000390240037861244\\
0.00223201646929523 0.000394500643572692\\
0.00226312956839953 0.00042957901194989\\
0.00229467636723194 0.00048913857490613\\
0.00232666291133146 0.000413925812164533\\
0.00235909533050864 0.000365951175727341\\
0.00239197984002024 0.000393577100103103\\
0.00242532274176035 0.000383118311726919\\
0.00245913042546804 0.000484645986692285\\
0.00249340936995188 0.000430459227110132\\
0.0025281661443315 0.000490538095586106\\
0.00256340740929651 0.000434763328820424\\
0.00259913991838293 0.000405364869768594\\
0.00263537051926739 0.00042022692633681\\
0.00267210615507944 0.000490566024680626\\
0.00270935386573205 0.000404264467170697\\
0.00274712078927081 0.000452019236902643\\
0.00278541416324177 0.000423110316426607\\
0.00282424132607844 0.000391421211745029\\
0.00286360971850812 0.000405391365913925\\
0.00290352688497781 0.000436058976236351\\
0.00294400047510003 0.00046060765900384\\
0.00298503824511873 0.000463092781135149\\
0.00302664805939569 0.000509954618063846\\
0.00306883789191764 0.000401873567564499\\
0.00311161582782436 0.000488369266093132\\
0.00315499006495809 0.000416709270907988\\
0.00319896891543454 0.000485233875884007\\
0.00324356080723581 0.000510469360932041\\
0.00328877428582551 0.00048515235543149\\
0.00333461801578636 0.000445766391807876\\
0.00338110078248069 0.000508083892014061\\
0.00342823149373397 0.000521468748770624\\
0.00347601918154198 0.00048102425095209\\
0.00352447300380159 0.000464083987364464\\
0.00357360224606579 0.000522492091497225\\
0.00362341632332315 0.000498980683343918\\
0.00367392478180207 0.000481154801326391\\
0.00372513730080021 0.000499295387310227\\
0.00377706369453937 0.000491339517368362\\
0.00382971391404628 0.000474391047315785\\
0.00388309804905961 0.000606464792106902\\
0.00393722632996348 0.00046818188779004\\
0.00399210912974805 0.000503728958654735\\
0.00404775696599732 0.000494999406827575\\
0.00410418050290471 0.000548601498044701\\
0.00416139055331671 0.000621805417512993\\
0.00421939808080502 0.000544432997751322\\
0.00427821420176762 0.000601220593591155\\
0.00433785018755899 0.000520858800413392\\
0.00439831746665022 0.000569550875049277\\
0.0044596276268191 0.000537057952604069\\
0.0045217924173707 0.000571246208276344\\
0.0045848237513891 0.000565698949875537\\
0.00464873370802026 0.000576688717239732\\
0.00471353453478691 0.000598673189377744\\
0.0047792386499356 0.00065470525462046\\
0.0048458586448165 0.00060810995296251\\
0.00491340728629636 0.000622442675619872\\
0.00498189751920516 0.00062479096537827\\
0.00505134246881677 0.000671056576913966\\
0.00512175544336424 0.000669455246320816\\
0.00519314993659021 0.000674894483117517\\
0.00526553963033276 0.000702697610892767\\
0.00533893839714735 0.000717309774562603\\
0.00541336030296538 0.000720509570131401\\
0.00548881960978967 0.00073769285637439\\
0.00556533077842765 0.000767495582138491\\
0.00564290847126254 0.000790220406833846\\
0.00572156755506325 0.000800604913968631\\
0.00580132310383338 0.000818274236836767\\
0.00588219040169996 0.000840379893517848\\
0.00596418494584246 0.000887884270799365\\
0.00604732244946265 0.000889517724740827\\
0.00613161884479578 0.000993132140220951\\
0.00621709028616383 0.000989716160452152\\
0.00630375315307128 0.00098375036799703\\
0.00639162405334401 0.00104955548332554\\
0.00648071982631197 0.00104696603794913\\
0.00657105754603627 0.00111249000579384\\
0.00666265452458115 0.00119161811715323\\
0.00675552831533165 0.00120830340188987\\
0.00684969671635745 0.00131725081240265\\
0.0069451777738237 0.0012797290824024\\
0.00704198978544929 0.0013175539916544\\
0.0071401513040134 0.00156171581232472\\
0.00723968114091088 0.00177887191838301\\
0.00734059836975721 0.00188799575283679\\
0.00744292233004376 0.00170759557331827\\
0.00754667263084389 0.00190476306367962\\
0.00765186915457083 0.00198610775847093\\
0.00775853206078784 0.00262309345780082\\
0.00786668179007158 0.00284457297004893\\
0.00797633906792928 0.00297165004882243\\
0.00808752490877046 0.00280330030691792\\
0.00820026061993413 0.00366045300670066\\
0.00831456780577206 0.00393144053947156\\
0.00843046837178897 0.00477681488325216\\
0.00854798452884041 0.0053186783601029\\
0.00866713879738923 0.0052676559421032\\
0.00878795401182132 0.00751293859906027\\
0.00891045332482151 0.00762788546576124\\
0.00903466021181053 0.00948654341458908\\
0.00916059847544371 0.0128922980645416\\
0.0092882922501725 0.0135649181578134\\
0.00941776600686952 0.0194075984526994\\
0.00954904455751808 0.026337781071356\\
0.00968215305996709 0.0293045266736871\\
0.00981711702275219 0.0304876562262272\\
0.00995396230998423 0.035066187420237\\
0.0100927151463057 0.0255619525147443\\
0.0102334021219164 0.0300242411325183\\
0.0103760501976691 0.0198394273869323\\
0.0105206867102362 0.0127556255977772\\
0.0106673393773486 0.0118144569917013\\
0.0108160363031071 0.00918966979225282\\
0.0109668059833687 0.00596069742959304\\
0.011119677311207 0.00501507242289439\\
0.0112746795824495 0.00526759241358084\\
0.0114318425012915 0.00413046319456512\\
0.0115911961859889 0.00302595647270353\\
0.0117527711746295 0.00268589295786321\\
0.0119165984309856 0.00257374960414745\\
0.0120827093504478 0.00216792935397089\\
0.0122511357660412 0.00168942683864544\\
0.0124219099545262 0.00139379310933675\\
0.0125950646425836 0.00156190160776287\\
0.0127706330130864 0.00115323831728194\\
0.012948648711459 0.000974804074530046\\
0.0131291458521247 0.000928075410577478\\
0.0133121590250431 0.000905883586055135\\
0.0134977233023394 0.000784748366015157\\
0.0136858742450252 0.00073266775481663\\
0.0138766479098131 0.000669098900118283\\
0.0140700808560269 0.000605182282288687\\
0.0142662101526074 0.000567782738317099\\
0.0144650733852165 0.000533895410368996\\
0.0146667086634397 0.00050466604486687\\
0.0148711546280895 0.000479065832844273\\
0.0150784504586105 0.000463554207675142\\
0.0152886358805873 0.000433558368222827\\
0.0155017511733577 0.00041588155147384\\
0.0157178371777316 0.000382662296911373\\
0.0159369353038178 0.000386162765069929\\
0.0161590875389592 0.000379423498303496\\
0.0163843364557798 0.000367532408333612\\
0.0166127252203429 0.000341344915976761\\
0.0168442976004231 0.000353770987293712\\
0.0170790979738943 0.000348329058646507\\
0.0173171713372335 0.000330245184786211\\
0.0175585633141447 0.00035634640177508\\
0.0178033201643011 0.00035696921941997\\
0.0180514887922111 0.000351252317524666\\
0.0183031167562061 0.000344479787004315\\
0.0185582522775552 0.000347793513019637\\
0.0188169442497056 0.000360067306282406\\
0.0190792422476527 0.000385025267486755\\
0.0193451965374405 0.000383698360729067\\
0.0196148580857943 0.000388751913850494\\
0.0198882785698881 0.000341486214599721\\
0.0201655103872475 0.000279700580507501\\
0.0204466066657912 0.000243592411078504\\
0.0207316212740123 0.000247939632353309\\
0.0210206088313016 0.00025892274326212\\
0.0213136247184144 0.00026768672983845\\
0.0216107250880838 0.000283228972519729\\
0.0219119668757815 0.000296908008286586\\
0.0222174078106288 0.000306888886747297\\
0.0225271064264598 0.000322419818774109\\
0.0228411220730382 0.000340705276391261\\
0.0231595149274315 0.000360275486152831\\
0.0234823460055428 0.000381425805287292\\
0.0238096771738036 0.000408767475316898\\
0.0241415711610302 0.00043975877816968\\
0.0244780915704444 0.000477554122255209\\
0.0248193028918626 0.000523358291635434\\
0.0251652705140539 0.000580730971818197\\
0.0255160607372719 0.000654286830692346\\
0.0258717407859592 0.000750150108101986\\
0.0262323788216312 0.000873642478182742\\
0.0265980439559376 0.00104482991425126\\
0.0269688062639069 0.00129450874801269\\
0.0273447367973758 0.00164677105437823\\
0.0277259075986048 0.00219878742683218\\
0.0281123917140846 0.0031016780644801\\
0.0285042632085343 0.00471183377152916\\
0.0289015971790951 0.00775356094193112\\
0.0293044697697214 0.0131962912179997\\
0.0297129581857733 0.0181262112053409\\
0.0301271407088116 0.0143062240240829\\
0.0305470967115997 0.00814722790599893\\
0.0309729066733141 0.00465515409395954\\
0.0314046521949675 0.00286242652174039\\
0.0318424160150465 0.00190459236412522\\
0.0322862820253673 0.0013454989616447\\
0.0327363352871525 0.000998732857616899\\
0.0331926620473319 0.000775847587177889\\
0.0336553497550709 0.000625819473532459\\
0.0341244870785289 0.000521521766145581\\
0.0346001639218511 0.000449430652623007\\
0.0350824714423979 0.000399860037224799\\
0.0355715020682139 0.000365353807593986\\
0.0360673495157403 0.000344117324900867\\
0.0365701088077749 0.000333381673532289\\
0.0370798762916817 0.000330354230637107\\
0.0375967496578547 0.000335865525954672\\
0.0381208279584389 0.000350818953900528\\
0.0386522116263126 0.000373489582113289\\
0.039191002494334 0.000394880865099446\\
0.0397373038148561 0.000354277455499114\\
0.0402912202795135 0.000262278037772926\\
0.0408528580392856 0.000225748935315358\\
0.041422324724839 0.000232943379013702\\
0.0419997294671531 0.000254453352267917\\
0.0425851829184341 0.000282941991620822\\
0.0431787972733202 0.000317708852151413\\
0.0437806862903817 0.000361799370808203\\
0.0443909653139217 0.000420176471277052\\
0.0450097512960805 0.000501870160205274\\
0.0456371628192476 0.00062167323329856\\
0.0462733201187869 0.000809430985153882\\
0.046918345106078 0.0011241944689122\\
0.0475723613918789 0.00169775107535568\\
0.0482354943100147 0.00282596749230322\\
0.0489078709413959 0.00495434634890173\\
0.0495896201383721 0.00669995760277632\\
0.0502808725494248 0.00478917053121315\\
0.0509817606442042 0.00255357421739622\\
0.051692418738916 0.00141922264274206\\
0.0524129830220606 0.000869005872107111\\
0.0531435915805324 0.000583461163868828\\
0.0538843844260822 0.000426183317152718\\
0.0546355035221488 0.000338769728636102\\
0.055397092811064 0.000293204357623636\\
0.0561692982416381 0.000274923337928382\\
0.0569522677971283 0.0002766351744848\\
0.0577461515235982 0.000294175592482161\\
0.0585511015586724 0.00032092191520007\\
0.0593672721606913 0.000321351367649464\\
0.0601948197382727 0.000262318738259176\\
0.0610339028802862 0.000221877179196882\\
0.0618846823862439 0.000226446809473931\\
0.0627473212971158 0.000255787539526114\\
0.0636219849265749 0.000302758509372189\\
0.0645088408926769 0.000372713964230442\\
0.0654080591499826 0.00048260673289184\\
0.0663198120221267 0.00067005985005971\\
0.0672442742348425 0.00101318981466307\\
0.0681816229494448 0.00163333712047464\\
0.0691320377967813 0.0023148431764657\\
0.0700957009116562 0.00199126124963476\\
0.0710727969677342 0.00114974543406644\\
0.0720635132129305 0.000634455251044724\\
0.073068039505295 0.000379717365894179\\
0.0740865683493956 0.000256529077274368\\
0.0751192949332097 0.000204313988370073\\
0.0761664171655289 0.000194764131828585\\
0.0772281357138865 0.00021121994834486\\
0.0783046540430119 0.000238257078989721\\
0.0793961784538228 0.000245254300254887\\
0.0805029181229598 0.00022051313237713\\
0.0816250851428723 0.000209711670080881\\
0.0827628945624635 0.000228685427263577\\
0.0839165644283016 0.000273080028967763\\
0.0850863158264058 0.000347350919450777\\
0.0862723729246145 0.000468991047913109\\
0.0874749630155442 0.00065759856465053\\
0.0886943165601471 0.000837557515599978\\
0.0899306672318762 0.000742320040819561\\
0.0911842519614657 0.000456168464848442\\
0.0924553109823358 0.000252306297013495\\
0.09374408787663 0.00015032925995189\\
0.0950508296218951 0.000117504441615747\\
0.0963757866384109 0.00012833919777238\\
0.09771921283718 0.000156904933801141\\
0.0990813656685867 0.000179396298478875\\
0.100462506171734 0.000182311921067046\\
0.101862899024469 0.000185303994871936\\
0.103282812594103 0.000205077017006381\\
0.104722518988843 0.0002426656500072\\
0.106182294109938 0.000293506268669501\\
0.107662417704549 0.000333279012333298\\
0.109163173419361 0.000299274635753091\\
0.110684848854941 0.000189363222475029\\
0.112227735620851 9.69093671425975e-05\\
0.113792129391532 5.82057681042911e-05\\
0.115378329962966 6.7941046149208e-05\\
0.116986641310131 9.52214122498451e-05\\
0.118617371645248 0.000120724781273206\\
0.120270833476851 0.000135174285962079\\
0.121947343669674 0.000144323664110189\\
0.123647223505372 0.000157387878840355\\
0.125370798744092 0.000171295346155386\\
0.127118399686903 0.000167994090565192\\
0.128890361239089 0.000123446774944884\\
0.130687022974335 5.36047669050406e-05\\
0.132508729199795 6.64671301747685e-06\\
0.134355829022083 3.32277535179891e-05\\
0.136228676414165 5.69523038316731e-05\\
0.138127630283201 7.67960652780773e-05\\
0.140053054539322 9.06376308255955e-05\\
0.142005318165368 9.90956363025499e-05\\
0.143984795287601 0.000104112115872449\\
0.145991865247398 0.000100458755404381\\
0.148026912673951 7.8148124459259e-05\\
0.150090327557974 4.20137893946463e-05\\
0.152182505326439 2.11097315129312e-05\\
0.154303846918356 2.91061707605159e-05\\
0.15645475886161 4.1518744302392e-05\\
0.158635653350859 5.24694813328099e-05\\
0.160846948326536 6.00172844297197e-05\\
0.163089067554933 6.31594607146982e-05\\
0.165362440709418 6.00373397655639e-05\\
0.16766750345277 4.81370079942801e-05\\
0.170004697520672 3.19551941228469e-05\\
0.172374470806362 2.32200497608051e-05\\
0.174777277446468 2.45308367558908e-05\\
0.177213577908036 2.98314716474201e-05\\
0.179683839076772 3.51117952544308e-05\\
0.182188534346517 3.77852802146267e-05\\
0.184728143709963 3.62295693971157e-05\\
0.187303153850644 3.00630428428408e-05\\
0.189914058236194 2.21962669215241e-05\\
0.19256135721292 1.75514717187073e-05\\
0.195245558101686 1.76762252894458e-05\\
0.197967175295133 2.0610880056263e-05\\
0.200726730356257 2.34988574814377e-05\\
0.203524752118358 2.42939812882649e-05\\
0.206361776786386 2.25068978597037e-05\\
0.209238348039698 1.95593242373782e-05\\
0.212155017136245 1.81852692887489e-05\\
0.215112343018217 2.02618361129537e-05\\
0.218110892419152 2.54055030527905e-05\\
0.221151239972549 3.21283199269127e-05\\
0.224233968321985 3.94978310450605e-05\\
0.227359668232772 4.75453632211033e-05\\
0.230528938705171 5.75501146258456e-05\\
0.233742387089181 7.19752120149795e-05\\
0.237000629200933 9.32225639038276e-05\\
0.240304289440697 0.000122993819903991\\
0.243654000912547 0.000162530020472418\\
0.247050405545683 0.000212589360356386\\
0.25049415421745 0.000274242227708157\\
0.253985906878073 0.000351006728220915\\
0.25752633267712 0.000449720206707739\\
0.261116110091746 0.000578484620953028\\
0.264755927056707 0.000743213424111499\\
0.268446481096197 0.000946233646527087\\
0.272188479457518 0.00119063625696668\\
0.275982639246618 0.00148711086120596\\
0.279829687565512 0.00185337777173281\\
0.283730361651621 0.00230480801477412\\
0.287685409019059 0.00284718057332477\\
0.29169558760188 0.00348249649698175\\
0.295761665899325 0.00422249249443773\\
0.299884423123103 0.00508943877263557\\
0.304064649346707 0.00609954559070783\\
0.308303145656828 0.00725114658795381\\
0.31260072430687 0.00853539424956522\\
0.316958208872612 0.00995340181994853\\
0.321376434410028 0.011510250971497\\
0.325856247615323 0.0131919918241357\\
0.330398506987186 0.0149608688910289\\
0.335004082991313 0.0167759835758502\\
0.339673858227221 0.0186049307393451\\
0.344408727597382 0.0204073773813121\\
0.349209598478727 0.0221200911519955\\
0.354077390896527 0.0236709838448875\\
0.359013037700707 0.0250026307284794\\
0.364017484744614 0.0260740330200055\\
0.369091691066278 0.0268484562086005\\
0.374236629072198 0.027295630632473\\
0.379453284723694 0.0274105828695931\\
0.384742657725851 0.0272271855865513\\
0.390105761719099 0.0268130484752675\\
0.39554362447347 0.0262520350511987\\
0.40105728808555 0.0256268087367856\\
0.406647809178186 0.0250073097485077\\
0.412316259102975 0.0244432055376877\\
0.418063724145576 0.0239573341871855\\
0.423891305733878 0.0235430927999705\\
0.42980012064908 0.023171767706353\\
0.435791301239703 0.0228094566556219\\
0.441865995638594 0.0224344975480588\\
0.448025367982949 0.0220450092268044\\
0.454270598637405 0.0216533243264673\\
0.46060288442024 0.0212733057724616\\
0.467023438832734 0.0209104381140188\\
0.473533492291712 0.0205606518291364\\
0.480134292365345 0.0202162539372272\\
0.486827104012228 0.0198727112402075\\
0.493613209823792 0.0195311952283889\\
0.500493910270095 0.0191963161441571\\
0.507470523949047 0.0188720183368767\\
0.514544387839093 0.0185588906394226\\
0.521716857555435 0.0182541625230457\\
0.528989307609815 0.0179536893316461\\
0.536363131673924 0.0176544149095609\\
0.54383974284648 0.0173559213734268\\
0.55142057392403 0.0170603441422197\\
0.559107077675529 0.0167708816538464\\
0.566900727120743 0.0164898946377789\\
0.574803015812536 0.0162176875257723\\
0.582815458123085 0.015952510416797\\
0.590939589534097 0.015691591107898\\
0.599176966931062 0.0154325376653738\\
0.607529168901607 0.015174392962171\\
0.615997796038017 0.0149178973372072\\
0.624584471243962 0.0146649502652785\\
0.633290840045512 0.0144176289872039\\
0.642118570906476 0.0141772539144433\\
0.651069355548146 0.013943876144401\\
0.660144909273489 0.0137163196175453\\
0.669346971295866 0.0134926778939741\\
0.678677305072329 0.0132710214935801\\
0.688137698641567 0.0130500379314097\\
0.697729964966554 0.0128293912267854\\
0.707455942281988 0.0126097156044648\\
0.717317494446561 0.0123922973975654\\
0.727316511300146 0.0121785955918626\\
0.737454909025955 0.0119697765833971\\
0.747734630517759 0.0117663987685272\\
0.758157645752211 0.0115683104744972\\
0.768725952166374 0.0113747526407149\\
0.779441575040495 0.011184606568528\\
0.790306567886135 0.0109967022296622\\
0.801323012839689 0.0108101021944698\\
0.812493021061405 0.0106242950106115\\
0.823818733139961 0.0104392636679165\\
0.835302319502678 0.0102554303478398\\
0.846945980831459 0.0100735086566586\\
0.858751948484517 0.0098943116023856\\
0.870722484923992 0.00971856450817265\\
0.882859884149515 0.00954676136872489\\
0.89516647213783 0.00937908548364354\\
0.907644607288536 0.00921539824232081\\
0.920296680876042 0.00905528654242157\\
0.933125117507825 0.00889815122405328\\
0.946132375589077 0.00874331568093363\\
0.959320947793824 0.00859013434450734\\
0.972693361542618 0.00843808388801653\\
0.986252179486878 0.00828682531373137\\
1 0.00813623114027455\\
};
\addlegendentry{BT};
\end{axis}
\end{tikzpicture}%

%% file: figures/relatverr_bt_irka30.tex
%
%
%
%
\begin{tikzpicture}

\begin{axis}[%
width=\figwidth,
height=\figheight,
scale only axis,
separate axis lines,
every outer x axis line/.append style={darkgray!60!black},
every x tick label/.append style={font=\color{darkgray!60!black}},
xmode=log,
xmin=0.001,
xmax=0.1,
xminorticks=true,
xlabel={$\omega$},
every outer y axis line/.append style={darkgray!60!black},
every y tick label/.append style={font=\color{darkgray!60!black}},
ymode=log,
ymin=1e-14,
ymax=1e-05,
yminorticks=true,
ylabel={$\frac{\sigma{}_{\text{max}}\text{($\mathcal{T}$(j}\omega\text{)}-\hat{\text{$\mathcal{T}$}}\text{(j}\omega\text{))}}{\sigma{}_{\text{max}}\text{($\mathcal{T}$(j}\omega\text{))}}$},
]
\addplot [
color=red,
dash pattern=on 1pt off 3pt on 3pt off 3pt,
]
table[row sep=crcr]{
0.001 1.63428015335392e-11\\
0.00100927151463057 1.12290239874345e-11\\
0.00101862899024469 7.68505884350074e-12\\
0.00102807322383086 5.8872315457313e-12\\
0.00103760501976691 5.39685539436571e-12\\
0.00104722518988843 4.14690637870843e-12\\
0.00105693455355799 1.17346622729759e-11\\
0.00106673393773486 1.16277104973684e-11\\
0.00107662417704549 1.68596951808517e-11\\
0.0010866061138546 8.50377127714693e-12\\
0.00109668059833687 3.18572483345117e-11\\
0.00110684848854941 1.51955399523027e-11\\
0.00111711065050482 1.23020547118745e-11\\
0.00112746795824495 4.51330691940934e-12\\
0.00113792129391532 1.49973122124762e-11\\
0.00114847154784029 3.86561573317098e-12\\
0.00115911961859889 1.58989779988248e-11\\
0.00116986641310131 1.60739424432764e-11\\
0.00118071284666619 3.51225884028614e-12\\
0.00119165984309856 8.23465040978989e-12\\
0.00120270833476851 1.6049198421068e-11\\
0.00121385926269063 5.47116916541392e-12\\
0.00122511357660412 9.23138759521778e-12\\
0.00123647223505371 1.99104346443393e-12\\
0.00124793620547131 1.11035692137256e-11\\
0.00125950646425836 1.73038434482138e-11\\
0.00127118399686903 2.93597259069016e-12\\
0.00128296979789415 8.84603525351197e-12\\
0.0012948648711459 1.09697788413829e-11\\
0.00130687022974335 2.60224768480017e-11\\
0.00131898689619867 1.77635042420805e-11\\
0.00133121590250431 1.16689746238505e-11\\
0.00134355829022083 8.86936325508143e-12\\
0.00135601511056563 1.26470131039866e-11\\
0.00136858742450252 1.16082797042876e-11\\
0.00138127630283201 1.61831725148455e-12\\
0.00139408282628258 4.1044618883962e-12\\
0.00140700808560269 2.15046010883672e-11\\
0.00142005318165368 1.22133512166179e-11\\
0.00143321922550357 1.04050242342518e-11\\
0.00144650733852165 1.45313892895493e-11\\
0.00145991865247398 8.56766245465208e-13\\
0.00147345430961984 4.73215364696646e-12\\
0.00148711546280895 7.34897314032284e-12\\
0.00150090327557974 2.60499779782637e-11\\
0.00151481892225835 3.53936561862121e-12\\
0.00152886358805873 3.16380786780641e-11\\
0.00154303846918356 8.12506059310044e-12\\
0.00155734477292614 2.06648479476278e-11\\
0.00157178371777316 1.75767026528954e-11\\
0.00158635653350859 6.68479064629256e-12\\
0.00160106446131832 5.87565738678266e-12\\
0.00161590875389592 8.86600127464108e-12\\
0.00163089067554933 1.12703072795076e-11\\
0.00164601150230855 7.85038992125623e-12\\
0.00166127252203429 7.49458943210828e-12\\
0.0016766750345277 1.03311767537183e-11\\
0.00169222035164104 1.42743437370809e-12\\
0.00170790979738943 1.91212684568517e-11\\
0.00172374470806362 1.19352588964928e-11\\
0.0017397264323438 3.90346523465477e-13\\
0.00175585633141447 1.7884987437133e-11\\
0.00177213577908036 2.52769509618886e-11\\
0.00178856616188346 1.90515804243792e-12\\
0.00180514887922111 1.16030511108985e-11\\
0.00182188534346517 1.13923434351819e-11\\
0.00183877698008233 1.52981018261451e-11\\
0.00185582522775552 8.41812847183492e-12\\
0.00187303153850644 9.3806039872616e-12\\
0.00189039737781922 7.10160064319061e-12\\
0.00190792422476527 2.30858750727415e-11\\
0.0019256135721292 1.61125326090928e-11\\
0.00194346692653602 1.56456236652458e-11\\
0.00196148580857943 9.71857129344062e-12\\
0.00197967175295133 1.66100546497154e-11\\
0.00199802630857255 2.4010252247779e-11\\
0.00201655103872475 9.1743068555605e-13\\
0.00203524752118358 1.98089845240887e-11\\
0.00205411734835306 1.51932628875725e-11\\
0.00207316212740123 1.46983170487776e-11\\
0.00209238348039698 1.44075433767134e-11\\
0.00211178304444824 2.22096315348502e-11\\
0.00213136247184144 1.01016922404362e-11\\
0.00215112343018217 7.11341243101383e-12\\
0.00217106760253726 1.24668094838515e-11\\
0.00219119668757815 1.44460292390695e-11\\
0.00221151239972549 1.49368145664887e-11\\
0.00223201646929523 7.26866858551543e-12\\
0.00225271064264598 4.19474503855319e-12\\
0.00227359668232772 2.30182236822807e-12\\
0.00229467636723194 1.75541141517654e-11\\
0.00231595149274315 3.61258527306522e-13\\
0.00233742387089181 2.24839774989003e-12\\
0.00235909533050864 1.68640978990009e-11\\
0.00238096771738036 5.25612594987381e-12\\
0.00240304289440697 2.03695148620591e-11\\
0.00242532274176035 1.26560560145428e-11\\
0.00244780915704444 4.83896824364985e-12\\
0.00247050405545683 1.26067356193892e-11\\
0.00249340936995188 4.41151107155958e-14\\
0.00251652705140539 7.16190243476234e-12\\
0.00253985906878073 8.25142303048418e-12\\
0.00256340740929651 5.7835446092034e-13\\
0.00258717407859592 3.27071163036366e-11\\
0.00261116110091746 1.80552508752788e-12\\
0.00263537051926739 4.44417169033821e-12\\
0.00265980439559376 8.90793829687996e-12\\
0.00268446481096197 1.32924358890892e-11\\
0.00270935386573205 1.01834397581156e-11\\
0.00273447367973758 2.0518725073689e-12\\
0.00275982639246618 1.44320375711406e-11\\
0.00278541416324177 5.64300315498551e-12\\
0.00281123917140846 6.44704470507068e-12\\
0.00283730361651621 1.57488713728883e-11\\
0.00286360971850812 1.25214426639656e-11\\
0.00289015971790951 1.2406251946796e-11\\
0.0029169558760188 1.62131367775765e-11\\
0.00294400047510003 3.30788215302697e-12\\
0.00297129581857733 2.16168905782457e-11\\
0.00299884423123103 6.34635126232708e-12\\
0.00302664805939569 1.64737531539045e-11\\
0.00305470967115997 1.79062089991998e-12\\
0.00308303145656828 7.77782196176293e-12\\
0.00311161582782436 9.13492173572056e-12\\
0.00314046521949675 2.65738212488346e-11\\
0.00316958208872612 1.84208214216179e-11\\
0.00319896891543454 6.94365186758503e-12\\
0.00322862820253673 7.91265285187632e-12\\
0.00325856247615323 6.61622781150628e-12\\
0.00328877428582551 5.38010783725099e-12\\
0.00331926620473319 4.5793272038063e-12\\
0.00335004082991313 9.83439958899401e-12\\
0.00338110078248069 1.07860827597123e-11\\
0.00341244870785289 1.64658285345993e-11\\
0.00344408727597382 4.25856565785147e-12\\
0.00347601918154198 3.23637661683848e-13\\
0.00350824714423979 2.04465768405742e-11\\
0.00354077390896527 1.3899716167252e-11\\
0.00357360224606579 1.21113731940035e-11\\
0.00360673495157403 2.98664116244504e-12\\
0.00364017484744614 1.79312524187953e-11\\
0.00367392478180207 4.76817019283851e-12\\
0.00370798762916817 5.5811755627436e-12\\
0.00374236629072198 5.37925415085087e-12\\
0.00377706369453937 3.98703236038732e-12\\
0.00381208279584389 2.93882571025713e-12\\
0.00384742657725851 9.81341466172325e-12\\
0.00388309804905961 3.17666662007567e-11\\
0.00391910024943341 4.03449424466994e-12\\
0.0039554362447347 2.47507975491986e-11\\
0.00399210912974805 6.75725956917425e-12\\
0.00402912202795135 8.03572741119453e-12\\
0.00406647809178186 1.64033807675481e-11\\
0.00410418050290471 7.48118762265489e-12\\
0.0041422324724839 2.24173426826096e-11\\
0.00418063724145576 7.85615622348786e-13\\
0.00421939808080502 1.35911895981802e-12\\
0.00425851829184341 8.68082085959578e-13\\
0.0042980012064908 9.36146721671625e-12\\
0.00433785018755899 1.98067767602296e-11\\
0.00437806862903817 2.01103478822173e-11\\
0.00441865995638594 1.27853666710756e-11\\
0.0044596276268191 2.03659769916871e-11\\
0.00450097512960805 3.82186474702346e-12\\
0.00454270598637404 4.77429216198584e-11\\
0.0045848237513891 1.15253798309571e-11\\
0.00462733201187869 1.39367681544064e-11\\
0.00467023438832734 3.33800404730506e-11\\
0.00471353453478691 7.75352346218909e-13\\
0.00475723613918789 2.02799427476673e-11\\
0.00480134292365346 2.73894007805737e-11\\
0.0048458586448165 8.54407797779578e-12\\
0.00489078709413959 1.78964022575589e-11\\
0.00493613209823792 1.22856633546435e-11\\
0.00498189751920516 2.33397247013103e-11\\
0.00502808725494248 1.3706980223793e-12\\
0.00507470523949047 1.40817069616565e-11\\
0.00512175544336424 4.89609217298993e-12\\
0.0051692418738916 1.47633688966986e-11\\
0.00521716857555435 1.95962731447092e-11\\
0.00526553963033276 1.1908748274101e-11\\
0.00531435915805324 9.38211133825548e-12\\
0.00536363131673924 6.29790436880105e-12\\
0.00541336030296538 1.36166655515433e-12\\
0.00546355035221488 5.03853095523205e-12\\
0.0055142057392403 2.67875720044065e-11\\
0.00556533077842765 1.74396706773038e-11\\
0.00561692982416381 2.86252155891634e-12\\
0.00566900727120743 1.57349410877829e-11\\
0.00572156755506325 1.4555654100382e-11\\
0.00577461515235982 1.83568511334893e-11\\
0.00582815458123085 2.12933435708393e-11\\
0.00588219040169996 5.0691445794446e-12\\
0.00593672721606913 3.77405466753381e-11\\
0.00599176966931062 2.5053338159721e-11\\
0.00604732244946265 1.31285226334248e-11\\
0.00610339028802862 8.280577484728e-12\\
0.00615997796038017 3.64363112875842e-13\\
0.00621709028616383 1.96000939358216e-11\\
0.00627473212971158 1.42008436507242e-11\\
0.00633290840045512 1.91887124455111e-11\\
0.00639162405334401 9.4462774489774e-12\\
0.00645088408926769 1.44789245155819e-11\\
0.00651069355548146 5.81975427920525e-12\\
0.00657105754603627 4.5656521260607e-12\\
0.00663198120221268 2.98195979959488e-12\\
0.00669346971295866 8.13103896556786e-12\\
0.00675552831533165 8.61627778452557e-12\\
0.00681816229494448 4.72785607529119e-12\\
0.00688137698641567 3.69158731252805e-11\\
0.0069451777738237 3.12001559771648e-11\\
0.00700957009116562 1.13585804011049e-11\\
0.00707455942281988 2.78254750544722e-12\\
0.0071401513040134 1.04746327683089e-11\\
0.00720635132129305 4.64870563715667e-11\\
0.00727316511300145 6.74723956429767e-12\\
0.00734059836975721 3.10261270590361e-11\\
0.00740865683493957 2.67113721869664e-11\\
0.00747734630517759 2.15200061688851e-11\\
0.00754667263084389 1.57379101081236e-11\\
0.00761664171655289 1.51659640006614e-11\\
0.00768725952166374 2.32476827605272e-11\\
0.00775853206078784 4.0472685345188e-11\\
0.00783046540430118 2.5628429992493e-11\\
0.00790306567886135 3.07533408145386e-11\\
0.00797633906792928 2.07162745957578e-11\\
0.00805029181229598 5.5541295029481e-12\\
0.00812493021061405 3.62255818165376e-11\\
0.00820026061993413 2.0715163593933e-11\\
0.00827628945624635 3.36483330884739e-11\\
0.00835302319502678 2.45227635611056e-11\\
0.00843046837178897 3.15315490998969e-11\\
0.00850863158264058 1.0960288970377e-11\\
0.00858751948484518 6.53905585672741e-11\\
0.00866713879738923 4.80613236429199e-11\\
0.00874749630155442 6.77909575704621e-12\\
0.00882859884149515 8.85336484480305e-11\\
0.00891045332482151 4.14879601097263e-11\\
0.00899306672318762 1.39649527154621e-10\\
0.00907644607288536 1.24922506788565e-10\\
0.00916059847544371 1.74397408878796e-11\\
0.00924553109823357 2.68312839634865e-11\\
0.00933125117507825 2.01496128439194e-10\\
0.00941776600686952 2.12748187131116e-11\\
0.00950508296218951 6.72748375786277e-11\\
0.00959320947793824 2.20857725654824e-10\\
0.00968215305996709 5.52601118269313e-11\\
0.009771921283718 1.09906499896877e-10\\
0.00986252179486878 1.2316400258553e-10\\
0.00995396230998423 3.64055778865571e-11\\
0.0100462506171734 1.9775292194939e-10\\
0.0101393945767529 2.57930576133931e-10\\
0.0102334021219164 1.16166414812175e-10\\
0.0103282812594103 4.56856460776132e-11\\
0.0104240400702156 4.55055176008824e-11\\
0.0105206867102362 1.02278891534528e-10\\
0.0106182294109938 2.84426618629971e-11\\
0.0107166764803286 5.46904793354827e-11\\
0.0108160363031071 1.2152729204415e-11\\
0.0109163173419361 7.94457487837295e-11\\
0.0110175281378839 6.9624092466e-11\\
0.011119677311207 4.35194875386434e-11\\
0.0112227735620851 4.90364937050642e-11\\
0.0113268256713615 5.88938900065944e-12\\
0.0114318425012915 3.39676503220296e-11\\
0.0115378329962966 4.08031715889654e-11\\
0.0116448061837269 3.28236247424872e-11\\
0.0117527711746295 4.17774277208122e-12\\
0.0118617371645248 3.62080151613426e-11\\
0.0119717134341897 1.98807398807897e-11\\
0.0120827093504478 3.21396752376281e-11\\
0.0121947343669674 9.03657618944793e-13\\
0.0123077980250667 2.72019320415622e-11\\
0.0124219099545262 1.96087250246773e-11\\
0.0125370798744092 2.50377663612165e-11\\
0.0126533175938894 8.27762604731558e-12\\
0.0127706330130864 4.81607879975336e-12\\
0.0128890361239089 1.55068040867424e-12\\
0.0130085370109057 1.05533236771723e-11\\
0.0131291458521247 6.48039435887945e-12\\
0.0132508729199795 9.29513534166396e-13\\
0.013373728582125 5.34031990042211e-12\\
0.0134977233023394 3.88026500135574e-12\\
0.0136228676414165 1.33486179473277e-11\\
0.0137491722580642 1.82267657105414e-11\\
0.0138766479098131 1.07609787393262e-11\\
0.0140053054539322 7.96438728133309e-12\\
0.014135155848354 7.88323274045105e-12\\
0.0142662101526074 1.51331592826083e-11\\
0.0143984795287601 1.39124242863518e-11\\
0.014531975242369 2.09396173238732e-11\\
0.0146667086634397 8.75802793946951e-12\\
0.0148026912673951 1.0889485414822e-11\\
0.0149399346360526 5.79707710525737e-12\\
0.0150784504586105 1.4588777102196e-11\\
0.0152182505326439 4.12678824920245e-13\\
0.015359346765109 9.36336853434784e-12\\
0.0155017511733577 3.92132069676065e-12\\
0.015645475886161 2.49373099975477e-12\\
0.0157905331447418 8.40679355744098e-13\\
0.0159369353038178 2.01718426146825e-12\\
0.0160846948326536 3.12273145404192e-12\\
0.0162338243161228 4.34384338798119e-12\\
0.0163843364557798 2.83739631582604e-12\\
0.0165362440709418 3.634141200361e-12\\
0.0166895600997802 5.60731468980213e-12\\
0.0168442976004231 2.05062621352673e-12\\
0.0170004697520672 8.72425749793491e-12\\
0.0171580898561 1.49230878893895e-11\\
0.0173171713372335 7.70794581267297e-12\\
0.0174777277446468 6.98543224352986e-13\\
0.0176397727531404 2.39922405765062e-12\\
0.0178033201643011 9.21207121007185e-12\\
0.0179683839076772 8.47253930402596e-12\\
0.018134978041965 2.62933760941475e-12\\
0.0183031167562061 2.94155127242398e-12\\
0.0184728143709963 3.71356046092243e-12\\
0.0186440853397049 4.48915131176514e-12\\
0.0188169442497056 5.80354473229406e-12\\
0.0189914058236193 1.55989382240171e-11\\
0.0191674849205681 1.15342722207549e-11\\
0.0193451965374405 1.65152174797363e-11\\
0.0195245558101686 1.88139156047041e-11\\
0.019705578015018 3.07687911056591e-11\\
0.0198882785698881 2.62838303719092e-11\\
0.0200726730356257 2.7230005165584e-11\\
0.0202587771173502 2.04156669965793e-11\\
0.0204466066657912 1.59921267517193e-11\\
0.0206361776786386 1.23012198602939e-11\\
0.0208275063019052 9.66478755342649e-12\\
0.0210206088313016 1.01618778515801e-11\\
0.0212155017136245 4.3746983122945e-12\\
0.0214122015481573 2.62280239813771e-12\\
0.0216107250880838 7.76787016162558e-12\\
0.0218110892419152 5.54869725188188e-12\\
0.0220133110749303 2.48513558987395e-12\\
0.0222174078106288 7.86733765834723e-13\\
0.0224233968321985 2.96334835006157e-12\\
0.0226312956839953 2.38096669626101e-12\\
0.0228411220730382 1.37239675666348e-12\\
0.0230528938705171 6.79664881049782e-12\\
0.0232666291133146 3.16598745862071e-12\\
0.0234823460055428 4.97274718780014e-12\\
0.0237000629200933 3.9165180207901e-12\\
0.0239197984002024 8.13157743082109e-13\\
0.0241415711610302 1.6355110076277e-12\\
0.0243654000912547 4.71848379108412e-12\\
0.0245913042546805 4.28261935454957e-12\\
0.0248193028918626 6.22986843927461e-12\\
0.025049415421745 6.10291963712566e-12\\
0.025281661443315 2.6392853339589e-12\\
0.0255160607372719 2.43146181835503e-12\\
0.025752633267712 1.0088104872897e-12\\
0.0259913991838293 5.89060870766874e-12\\
0.0262323788216312 9.07399294275375e-12\\
0.0264755927056707 6.29811699417462e-12\\
0.0267210615507944 1.10831884399429e-12\\
0.0269688062639069 1.43559815799141e-11\\
0.0272188479457518 1.43645973222267e-11\\
0.0274712078927081 4.47876639766272e-13\\
0.0277259075986048 1.44518955153308e-12\\
0.0279829687565512 1.13861315032172e-11\\
0.0282424132607843 1.32957105814908e-11\\
0.0285042632085343 1.47943226464784e-12\\
0.0287685409019059 4.69916984016732e-11\\
0.0290352688497781 8.51143555498999e-12\\
0.0293044697697214 1.42750547500434e-11\\
0.0295761665899325 1.22497373273776e-11\\
0.0298503824511873 2.8451635056977e-11\\
0.0301271407088116 1.95101307647927e-12\\
0.0304064649346707 1.26807006175596e-11\\
0.0306883789191764 1.94335606082433e-12\\
0.0309729066733141 7.01134407570018e-12\\
0.031260072430687 5.38285234833135e-12\\
0.0315499006495809 7.76906661570896e-12\\
0.0318424160150465 2.94564578793984e-12\\
0.0321376434410028 6.37681711242449e-13\\
0.0324356080723581 6.76620909815555e-12\\
0.0327363352871525 3.4456344655577e-12\\
0.0330398506987186 1.96104267113633e-12\\
0.0333461801578636 6.19468014508961e-13\\
0.0336553497550709 3.29483243196782e-12\\
0.0339673858227221 2.47582657410987e-13\\
0.0342823149373397 4.51530387287608e-12\\
0.0346001639218511 3.39621418935533e-12\\
0.0349209598478727 2.35234864985339e-12\\
0.0352447300380159 9.39038546015785e-13\\
0.0355715020682139 2.23285069821055e-12\\
0.0359013037700707 4.98294484654624e-13\\
0.0362341632332315 1.79623091192456e-12\\
0.0365701088077749 1.77029604601451e-12\\
0.0369091691066278 1.91346112422632e-12\\
0.0372513730080021 3.12243843054818e-12\\
0.0375967496578547 2.44032032125515e-12\\
0.0379453284723694 3.88466047232927e-12\\
0.0382971391404628 4.18931782320878e-12\\
0.0386522116263126 6.89485597585444e-12\\
0.0390105761719099 1.09514751941816e-11\\
0.0393722632996348 1.33645400840254e-11\\
0.0397373038148561 1.56922515920556e-11\\
0.040105728808555 1.50410987397273e-11\\
0.0404775696599732 1.10569632361494e-11\\
0.0408528580392856 7.89364738777672e-12\\
0.0412316259102975 4.32530110766149e-12\\
0.0416139055331671 2.77344963377762e-12\\
0.0419997294671531 2.36133155713419e-12\\
0.0423891305733878 1.88910778930516e-12\\
0.0427821420176762 2.90283476120512e-12\\
0.0431787972733202 9.09682074558781e-13\\
0.0435791301239703 3.11614891016819e-12\\
0.0439831746665023 1.45935789083486e-12\\
0.0443909653139217 2.2606027152657e-12\\
0.0448025367982949 1.98015628144589e-13\\
0.045217924173707 2.60175730806368e-12\\
0.0456371628192476 1.25903778930294e-12\\
0.046060288442024 8.61177626544539e-13\\
0.0464873370802026 6.56909543150082e-13\\
0.046918345106078 1.39246811700569e-12\\
0.0473533492291712 2.95748868237749e-12\\
0.047792386499356 1.21181589828756e-12\\
0.0482354943100147 6.35165806823727e-12\\
0.0486827104012228 7.71619282605332e-12\\
0.0491340728629636 2.09990112591609e-12\\
0.0495896201383721 1.86193626661432e-12\\
0.0500493910270095 1.83063169892014e-12\\
0.0505134246881676 3.61160159581332e-12\\
0.0509817606442042 4.85595699699603e-12\\
0.0514544387839093 9.05938793181902e-13\\
0.0519314993659021 3.58874680535688e-12\\
0.0524129830220606 5.99543222798325e-13\\
0.0528989307609815 2.9671743693354e-12\\
0.0533893839714735 2.04386508371588e-12\\
0.0538843844260822 1.20017986345224e-12\\
0.054383974284648 1.17631087233062e-12\\
0.0548881960978967 1.5557279923319e-12\\
0.055397092811064 2.21906962004817e-12\\
0.0559107077675529 2.30302353175908e-12\\
0.0564290847126254 3.69826806025164e-12\\
0.0569522677971283 4.17343849203768e-12\\
0.0574803015812535 6.20267167166146e-12\\
0.0580132310383338 8.3448526766957e-12\\
0.0585511015586724 1.18717762065214e-11\\
0.0590939589534097 1.68069205182179e-11\\
0.0596418494584246 1.76078107156957e-11\\
0.0601948197382727 1.43337659149923e-11\\
0.0607529168901607 9.58769101361499e-12\\
0.0613161884479577 6.35689927843024e-12\\
0.0618846823862439 4.69987038289271e-12\\
0.0624584471243962 3.41077350827048e-12\\
0.0630375315307128 2.24278032925876e-12\\
0.0636219849265749 1.12519257364935e-12\\
0.0642118570906476 7.56792837162645e-13\\
0.0648071982631197 1.43831423003259e-12\\
0.0654080591499826 2.14178214304557e-12\\
0.0660144909273489 1.64751645518079e-12\\
0.0666265452458115 1.3393197817794e-12\\
0.0672442742348425 1.647615065436e-12\\
0.067867730507233 2.14000545979006e-12\\
0.0684969671635745 6.95598955018717e-13\\
0.0691320377967813 2.65061202052974e-12\\
0.0697729964966554 1.45867920549106e-12\\
0.0704198978544929 1.65363023334446e-12\\
0.0710727969677342 2.73202633564834e-12\\
0.0717317494446562 2.00129404536552e-12\\
0.0723968114091088 2.9725148711386e-12\\
0.073068039505295 4.53729748924972e-12\\
0.0737454909025955 3.99979632662542e-12\\
0.0744292233004376 5.5257436571007e-12\\
0.0751192949332097 7.92383164971738e-12\\
0.0758157645752211 1.03356652479506e-11\\
0.0765186915457083 1.45772509109393e-11\\
0.0772281357138865 2.03553627356605e-11\\
0.0779441575040495 2.8591642038777e-11\\
0.0786668179007158 3.77498901436698e-11\\
0.0793961784538228 4.17714250044972e-11\\
0.0801323012839689 3.59193596644942e-11\\
0.0808752490877045 2.64736860254233e-11\\
0.0816250851428723 1.83244282229558e-11\\
0.082381873313996 1.2718573268803e-11\\
0.0831456780577206 8.97000517144353e-12\\
0.0839165644283016 6.44124815446442e-12\\
0.0846945980831458 4.84092167377125e-12\\
0.0854798452884041 4.05757233851271e-12\\
0.0862723729246145 3.00213446821089e-12\\
0.0870722484923992 3.03892559387907e-12\\
0.0878795401182132 3.43105654530783e-12\\
0.0886943165601471 4.40356689904702e-12\\
0.089516647213783 5.64921910186152e-12\\
0.0903466021181053 7.01556232983675e-12\\
0.0911842519614657 8.70332819337826e-12\\
0.0920296680876042 1.12945584508416e-11\\
0.092882922501725 1.55422682612595e-11\\
0.09374408787663 2.07822229794129e-11\\
0.0946132375589077 2.82759896331584e-11\\
0.0954904455751808 3.93859580423496e-11\\
0.0963757866384109 5.56808298498585e-11\\
0.0972693361542617 7.81687066058283e-11\\
0.0981711702275219 1.03792662082256e-10\\
0.0990813656685867 1.22205073973603e-10\\
0.1 1.19981264438356e-10\\
};
\addplot [
color=blue,
dashed,
]
table[row sep=crcr]{
0.001 1.09236411927193e-10\\
0.00101393945767529 1.07608982718674e-10\\
0.00102807322383086 8.76383160218041e-11\\
0.00104240400702156 1.16377722171107e-10\\
0.00105693455355799 8.21038379743371e-11\\
0.00107166764803286 9.51994464038327e-11\\
0.0010866061138546 1.01702012985465e-10\\
0.00110175281378839 1.04564674158603e-10\\
0.00111711065050482 8.16545437452084e-11\\
0.00113268256713615 9.3948168074099e-11\\
0.00114847154784029 9.72042736975068e-11\\
0.00116448061837269 1.05388578058878e-10\\
0.00118071284666619 9.69132028118842e-11\\
0.00119717134341897 8.97422050199195e-11\\
0.00121385926269063 8.83572325259299e-11\\
0.00123077980250667 9.38847961772324e-11\\
0.00124793620547131 1.04293131224683e-10\\
0.00126533175938894 9.5175055093158e-11\\
0.00128296979789415 8.53271354687797e-11\\
0.00130085370109057 8.40028570364428e-11\\
0.00131898689619867 7.70566255254668e-11\\
0.0013373728582125 1.02723725881999e-10\\
0.00135601511056563 1.05785841827739e-10\\
0.00137491722580642 8.62358539469654e-11\\
0.00139408282628258 9.76981127081768e-11\\
0.0014135155848354 1.03774069554689e-10\\
0.00143321922550357 1.0372457447688e-10\\
0.0014531975242369 8.53923176210628e-11\\
0.00147345430961984 9.84298148548681e-11\\
0.00149399346360526 1.00163288986128e-10\\
0.00151481892225835 9.73416824202841e-11\\
0.0015359346765109 8.71445812183692e-11\\
0.00155734477292614 7.53927017028555e-11\\
0.00157905331447418 8.52141510239874e-11\\
0.00160106446131832 9.96470736338236e-11\\
0.00162338243161228 1.16219522904749e-10\\
0.00164601150230855 8.71555957537201e-11\\
0.00166895600997802 9.88658008205256e-11\\
0.00169222035164104 9.56362089226296e-11\\
0.00171580898561 8.86888470858346e-11\\
0.0017397264323438 9.40592444762508e-11\\
0.00176397727531405 1.1025074152585e-10\\
0.00178856616188346 9.62414313010319e-11\\
0.0018134978041965 1.11511070648039e-10\\
0.00183877698008233 8.1231233006449e-11\\
0.00186440853397049 8.89512342500367e-11\\
0.00189039737781922 1.01143201018939e-10\\
0.00191674849205682 8.59639018667984e-11\\
0.00194346692653602 8.14132571489439e-11\\
0.0019705578015018 9.79255951246249e-11\\
0.00199802630857255 7.49164299025915e-11\\
0.00202587771173502 8.37288903451326e-11\\
0.00205411734835306 1.08616721255322e-10\\
0.00208275063019051 1.11607290148341e-10\\
0.00211178304444824 7.70838372761399e-11\\
0.00214122015481573 1.07509028285859e-10\\
0.00217106760253726 8.5087929832182e-11\\
0.00220133110749303 8.8799288968365e-11\\
0.00223201646929523 8.96408040560642e-11\\
0.00226312956839953 9.74682388000255e-11\\
0.00229467636723194 1.10814123639815e-10\\
0.00232666291133146 9.36287795404394e-11\\
0.00235909533050864 8.26444229850596e-11\\
0.00239197984002024 8.8736654135199e-11\\
0.00242532274176035 8.62318251950958e-11\\
0.00245913042546804 1.08892676139461e-10\\
0.00249340936995188 9.65434237529174e-11\\
0.0025281661443315 1.09813708363471e-10\\
0.00256340740929651 9.71416851142863e-11\\
0.00259913991838293 9.03946522357523e-11\\
0.00263537051926739 9.35187289580187e-11\\
0.00267210615507944 1.089440781751e-10\\
0.00270935386573205 8.9585075679602e-11\\
0.00274712078927081 9.99453374947331e-11\\
0.00278541416324177 9.33394910272168e-11\\
0.00282424132607844 8.61453949788375e-11\\
0.00286360971850812 8.9003442375679e-11\\
0.00290352688497781 9.54970173761582e-11\\
0.00294400047510003 1.00613114127011e-10\\
0.00298503824511873 1.00887133239053e-10\\
0.00302664805939569 1.10791884310454e-10\\
0.00306883789191764 8.70638005873256e-11\\
0.00311161582782436 1.05494573776019e-10\\
0.00315499006495809 8.97447667467002e-11\\
0.00319896891543454 1.04179066166799e-10\\
0.00324356080723581 1.09247171075434e-10\\
0.00328877428582551 1.03487088017743e-10\\
0.00333461801578636 9.47627495392936e-11\\
0.00338110078248069 1.07631968857464e-10\\
0.00342823149373397 1.10068040814699e-10\\
0.00347601918154198 1.0115255430773e-10\\
0.00352447300380159 9.72145886293205e-11\\
0.00357360224606579 1.09014866525175e-10\\
0.00362341632332315 1.03682414426187e-10\\
0.00367392478180207 9.95551495229447e-11\\
0.00372513730080021 1.02857043798156e-10\\
0.00377706369453937 1.00761256310766e-10\\
0.00382971391404628 9.6832076460121e-11\\
0.00388309804905961 1.2319477460153e-10\\
0.00393722632996348 9.46315033751216e-11\\
0.00399210912974805 1.01293207182798e-10\\
0.00404775696599732 9.90091642386674e-11\\
0.00410418050290471 1.09128234061888e-10\\
0.00416139055331671 1.22988101693141e-10\\
0.00421939808080502 1.07052621198836e-10\\
0.00427821420176762 1.1750149868123e-10\\
0.00433785018755899 1.0115680847593e-10\\
0.00439831746665022 1.098950949797e-10\\
0.0044596276268191 1.02929245155683e-10\\
0.0045217924173707 1.08720100355831e-10\\
0.0045848237513891 1.06889099449571e-10\\
0.00464873370802026 1.08153124386161e-10\\
0.00471353453478691 1.11408970575805e-10\\
0.0047792386499356 1.20861208796404e-10\\
0.0048458586448165 1.11328548654415e-10\\
0.00491340728629636 1.12972811040793e-10\\
0.00498189751920516 1.12388055055134e-10\\
0.00505134246881677 1.19594061475404e-10\\
0.00512175544336424 1.1816379585471e-10\\
0.00519314993659021 1.17937304851324e-10\\
0.00526553963033276 1.21525797961501e-10\\
0.00533893839714735 1.22720012594656e-10\\
0.00541336030296538 1.21891140640818e-10\\
0.00548881960978967 1.23349492592937e-10\\
0.00556533077842765 1.2678346702739e-10\\
0.00564290847126254 1.28897546432228e-10\\
0.00572156755506325 1.28883533941557e-10\\
0.00580132310383338 1.29933596936653e-10\\
0.00588219040169996 1.31549387397655e-10\\
0.00596418494584246 1.36928184567315e-10\\
0.00604732244946265 1.35061475462698e-10\\
0.00613161884479578 1.48362617970074e-10\\
0.00621709028616383 1.45361876332382e-10\\
0.00630375315307128 1.4194140546825e-10\\
0.00639162405334401 1.486463272816e-10\\
0.00648071982631197 1.45419474524109e-10\\
0.00657105754603627 1.51397269171432e-10\\
0.00666265452458115 1.58727937573553e-10\\
0.00675552831533165 1.573685261916e-10\\
0.00684969671635745 1.67545540667749e-10\\
0.0069451777738237 1.58768345865876e-10\\
0.00704198978544929 1.59225518326227e-10\\
0.0071401513040134 1.83575499257205e-10\\
0.00723968114091088 2.0306904141746e-10\\
0.00734059836975721 2.08951633774591e-10\\
0.00744292233004376 1.82881242056362e-10\\
0.00754667263084389 1.97008065350378e-10\\
0.00765186915457083 1.97943070711069e-10\\
0.00775853206078784 2.51296120031201e-10\\
0.00786668179007158 2.61249050849015e-10\\
0.00797633906792928 2.60859813959195e-10\\
0.00808752490877046 2.34430300070728e-10\\
0.00820026061993413 2.90543298533759e-10\\
0.00831456780577206 2.94961968713408e-10\\
0.00843046837178897 3.37191968627147e-10\\
0.00854798452884041 3.51399098515681e-10\\
0.00866713879738923 3.23832005882696e-10\\
0.00878795401182132 4.26922684867947e-10\\
0.00891045332482151 3.97725582229109e-10\\
0.00903466021181053 4.50229128347694e-10\\
0.00916059847544371 5.52267666527362e-10\\
0.0092882922501725 5.203455630396e-10\\
0.00941776600686952 6.63043328946249e-10\\
0.00954904455751808 8.02316633859653e-10\\
0.00968215305996709 8.07876073721195e-10\\
0.00981711702275219 7.89129353409941e-10\\
0.00995396230998423 9.04223521584194e-10\\
0.0100927151463057 7.01225781270074e-10\\
0.0102334021219164 9.2098725328475e-10\\
0.0103760501976691 6.98301646299016e-10\\
0.0105206867102362 5.1902700954466e-10\\
0.0106673393773486 5.54317328105095e-10\\
0.0108160363031071 4.93733026805207e-10\\
0.0109668059833687 3.63691788937295e-10\\
0.011119677311207 3.44631204460464e-10\\
0.0112746795824495 4.04547589964328e-10\\
0.0114318425012915 3.52043070356973e-10\\
0.0115911961859889 2.84436617964699e-10\\
0.0117527711746295 2.76902825112084e-10\\
0.0119165984309856 2.89589212247587e-10\\
0.0120827093504478 2.65056317339266e-10\\
0.0122511357660412 2.23571196930897e-10\\
0.0124219099545262 1.98950084108897e-10\\
0.0125950646425836 2.39723399673543e-10\\
0.0127706330130864 1.89785793413247e-10\\
0.012948648711459 1.71571382324176e-10\\
0.0131291458521247 1.74297851609264e-10\\
0.0133121590250431 1.81155564421728e-10\\
0.0134977233023394 1.6678221986176e-10\\
0.0136858742450252 1.65198896161217e-10\\
0.0138766479098131 1.59799414592922e-10\\
0.0140700808560269 1.52868018808761e-10\\
0.0142662101526074 1.51485076949097e-10\\
0.0144650733852165 1.50265290114912e-10\\
0.0146667086634397 1.49665105700215e-10\\
0.0148711546280895 1.49541808046319e-10\\
0.0150784504586105 1.52156464872951e-10\\
0.0152886358805873 1.49507525922919e-10\\
0.0155017511733577 1.50536920975569e-10\\
0.0157178371777316 1.45280118581048e-10\\
0.0159369353038178 1.53660532570184e-10\\
0.0161590875389592 1.58133999358142e-10\\
0.0163843364557798 1.60337593640442e-10\\
0.0166127252203429 1.55783743450251e-10\\
0.0168442976004231 1.68814520071869e-10\\
0.0170790979738943 1.73710002276255e-10\\
0.0173171713372335 1.72038560337395e-10\\
0.0175585633141447 1.93838457890942e-10\\
0.0178033201643011 2.02683164602401e-10\\
0.0180514887922111 2.08105349040126e-10\\
0.0183031167562061 2.12901796734304e-10\\
0.0185582522775552 2.24171107140161e-10\\
0.0188169442497056 2.41987155642012e-10\\
0.0190792422476527 2.6975745789733e-10\\
0.0193451965374405 2.80215899521589e-10\\
0.0196148580857943 2.95907297199487e-10\\
0.0198882785698881 2.70905707665747e-10\\
0.0201655103872475 2.31261328579276e-10\\
0.0204466066657912 2.09924936499555e-10\\
0.0207316212740123 2.22734897121519e-10\\
0.0210206088313016 2.42512600580842e-10\\
0.0213136247184144 2.61471575516305e-10\\
0.0216107250880838 2.88613804297455e-10\\
0.0219119668757815 3.15771967736873e-10\\
0.0222174078106288 3.40835579631067e-10\\
0.0225271064264598 3.74193297998197e-10\\
0.0228411220730382 4.13554109795834e-10\\
0.0231595149274315 4.5784826207434e-10\\
0.0234823460055428 5.08147568309745e-10\\
0.0238096771738036 5.71795015123316e-10\\
0.0241415711610302 6.47174559443544e-10\\
0.0244780915704444 7.41219557477259e-10\\
0.0248193028918626 8.59406026778978e-10\\
0.0251652705140539 1.01294349813708e-09\\
0.0255160607372719 1.21853992315218e-09\\
0.0258717407859592 1.50190672594686e-09\\
0.0262323788216312 1.89761419972158e-09\\
0.0265980439559376 2.49248787893338e-09\\
0.0269688062639069 3.44755146142362e-09\\
0.0273447367973758 4.99411455259328e-09\\
0.0277259075986048 7.69034065089985e-09\\
0.0281123917140846 1.19116516333749e-08\\
0.0285042632085343 1.57938321544685e-08\\
0.0289015971790951 1.68924235610303e-08\\
0.0293044697697214 1.70181726875027e-08\\
0.0297129581857733 1.57843171995588e-08\\
0.0301271407088116 1.1544281890141e-08\\
0.0305470967115997 7.40468287218047e-09\\
0.0309729066733141 4.89632882490319e-09\\
0.0314046521949675 3.42197122711944e-09\\
0.0318424160150465 2.53519627096825e-09\\
0.0322862820253673 1.96207025038039e-09\\
0.0327363352871525 1.57642458629845e-09\\
0.0331926620473319 1.31377099126016e-09\\
0.0336553497550709 1.12928806998848e-09\\
0.0341244870785289 9.97753002343483e-10\\
0.0346001639218511 9.08009422804939e-10\\
0.0350824714423979 8.50480123391431e-10\\
0.0355715020682139 8.16072410589989e-10\\
0.0360673495157403 8.05612755383653e-10\\
0.0365701088077749 8.1673736838049e-10\\
0.0370798762916817 8.45854568506844e-10\\
0.0375967496578547 8.97895095802135e-10\\
0.0381208279584389 9.78488339414675e-10\\
0.0386522116263126 1.08622489725778e-09\\
0.039191002494334 1.19705504936794e-09\\
0.0397373038148561 1.119221900451e-09\\
0.0402912202795135 8.63493085111135e-10\\
0.0408528580392856 7.74702196691937e-10\\
0.041422324724839 8.33604222192542e-10\\
0.0419997294671531 9.5022515962193e-10\\
0.0425851829184341 1.1037812922418e-09\\
0.0431787972733202 1.29669024821838e-09\\
0.0437806862903817 1.54815512001378e-09\\
0.0443909653139217 1.89062668478294e-09\\
0.0450097512960805 2.38456473990303e-09\\
0.0456371628192476 3.13763265074756e-09\\
0.0462733201187869 4.37586972628139e-09\\
0.046918345106078 6.57975972595436e-09\\
0.0475723613918789 1.0831236250613e-08\\
0.0482354943100147 1.8908701286281e-08\\
0.0489078709413959 2.83554783601942e-08\\
0.0495896201383721 2.73116119552705e-08\\
0.0502808725494248 1.73279309543654e-08\\
0.0509817606442042 9.79309491754947e-09\\
0.051692418738916 5.92858216989929e-09\\
0.0524129830220606 3.92154623010724e-09\\
0.0531435915805324 2.81461852795372e-09\\
0.0538843844260822 2.18011244156767e-09\\
0.0546355035221488 1.82712731256533e-09\\
0.055397092811064 1.66055072719548e-09\\
0.0561692982416381 1.6302599043537e-09\\
0.0569522677971283 1.71405994285578e-09\\
0.0577461515235982 1.90186077217322e-09\\
0.0585511015586724 2.16280618143948e-09\\
0.0593672721606913 2.25640560662403e-09\\
0.0601948197382727 1.91877710982329e-09\\
0.0610339028802862 1.69112711840413e-09\\
0.0618846823862439 1.79970604526999e-09\\
0.0627473212971158 2.12240080381327e-09\\
0.0636219849265749 2.62788742565304e-09\\
0.0645088408926769 3.39391229921128e-09\\
0.0654080591499826 4.62843947148817e-09\\
0.0663198120221267 6.79531097906407e-09\\
0.0672442742348425 1.08341409360604e-08\\
0.0681816229494448 1.76658110272213e-08\\
0.0691320377967813 2.25005922923153e-08\\
0.0700957009116562 1.72382394645839e-08\\
0.0710727969677342 9.98534487899471e-09\\
0.0720635132129305 5.79481469869916e-09\\
0.073068039505295 3.6638555721394e-09\\
0.0740865683493956 2.60522573984345e-09\\
0.0751192949332097 2.17534224315242e-09\\
0.0761664171655289 2.16755955716756e-09\\
0.0772281357138865 2.45214797677492e-09\\
0.0783046540430119 2.88170060062914e-09\\
0.0793961784538228 3.08827896387882e-09\\
0.0805029181229598 2.89047608018979e-09\\
0.0816250851428723 2.86257418668716e-09\\
0.0827628945624635 3.25362579147149e-09\\
0.0839165644283016 4.05520746604255e-09\\
0.0850863158264058 5.39054093368025e-09\\
0.0862723729246145 7.5923255808758e-09\\
0.0874749630155442 1.0925774612807e-08\\
0.0886943165601471 1.36072468468306e-08\\
0.0899306672318762 1.14720486283628e-08\\
0.0911842519614657 7.03307484638104e-09\\
0.0924553109823358 4.02294289803719e-09\\
0.09374408787663 2.50098946028125e-09\\
0.0950508296218951 2.04004165210429e-09\\
0.0963757866384109 2.32190257438263e-09\\
0.09771921283718 2.95441070772015e-09\\
0.0990813656685867 3.51278637437357e-09\\
0.100462506171734 3.71123276572044e-09\\
0.101862899024469 3.92174308544885e-09\\
0.103282812594103 4.51301548898175e-09\\
0.104722518988843 5.54834343218874e-09\\
0.106182294109938 6.93602065483389e-09\\
0.107662417704549 8.00210986644185e-09\\
0.109163173419361 7.12410803995582e-09\\
0.110684848854941 4.48922669895246e-09\\
0.112227735620851 2.34528236641202e-09\\
0.113792129391532 1.45721331853566e-09\\
0.115378329962966 1.76589838239545e-09\\
0.116986641310131 2.57031775813811e-09\\
0.118617371645248 3.38304483122794e-09\\
0.120270833476851 3.93093521260758e-09\\
0.121947343669674 4.35348532215908e-09\\
0.123647223505372 4.91889156725734e-09\\
0.125370798744092 5.52596955128113e-09\\
0.127118399686903 5.54091653482027e-09\\
0.128890361239089 4.11062506984912e-09\\
0.130687022974335 1.80222070073246e-09\\
0.132508729199795 2.28406635609977e-10\\
0.134355829022083 1.17761136773381e-09\\
0.136228676414165 2.08880228913506e-09\\
0.138127630283201 2.91722420547879e-09\\
0.140053054539322 3.56554920452006e-09\\
0.142005318165368 4.03377952535093e-09\\
0.143984795287601 4.37600244924286e-09\\
0.145991865247398 4.33902399171187e-09\\
0.148026912673951 3.44375906827728e-09\\
0.150090327557974 1.88294168725718e-09\\
0.152182505326439 9.66905962954435e-10\\
0.154303846918356 1.37128130023084e-09\\
0.15645475886161 2.01904300259136e-09\\
0.158635653350859 2.63672925573399e-09\\
0.160846948326536 3.11602521253617e-09\\
0.163089067554933 3.38300809059341e-09\\
0.165362440709418 3.30738307390056e-09\\
0.16766750345277 2.7159426652584e-09\\
0.170004697520672 1.84284655064822e-09\\
0.172374470806362 1.37193822183101e-09\\
0.174777277446468 1.49050948402611e-09\\
0.177213577908036 1.86838891758713e-09\\
0.179683839076772 2.26833918642575e-09\\
0.182188534346517 2.51621120186671e-09\\
0.184728143709963 2.48218637607708e-09\\
0.187303153850644 2.11379570398985e-09\\
0.189914058236194 1.59960331479018e-09\\
0.19256135721292 1.29789072805075e-09\\
0.195245558101686 1.34411242822132e-09\\
0.197967175295133 1.6139196608611e-09\\
0.200726730356257 1.8952575479906e-09\\
0.203524752118358 2.01645034502092e-09\\
0.206361776786386 1.91971798150195e-09\\
0.209238348039698 1.71260323994451e-09\\
0.212155017136245 1.63501879019941e-09\\
0.215112343018217 1.87274427085193e-09\\
0.218110892419152 2.41613184244222e-09\\
0.221151239972549 3.14435016922243e-09\\
0.224233968321985 3.97567457629137e-09\\
0.227359668232772 4.91799841604822e-09\\
0.230528938705171 6.1154141548657e-09\\
0.233742387089181 7.86013653216385e-09\\
0.237000629200933 1.04688630185462e-08\\
0.240304289440697 1.42071732141713e-08\\
0.243654000912547 1.93067533221876e-08\\
0.247050405545683 2.59586757607866e-08\\
0.25049415421745 3.44151765720901e-08\\
0.253985906878073 4.52773725137001e-08\\
0.25752633267712 5.96469662080674e-08\\
0.261116110091746 7.88976081096346e-08\\
0.264755927056707 1.04219533739273e-07\\
0.268446481096197 1.36399582738214e-07\\
0.272188479457518 1.76424717184532e-07\\
0.275982639246618 2.26539652134536e-07\\
0.279829687565512 2.90289523444123e-07\\
0.283730361651621 3.71157392257484e-07\\
0.287685409019059 4.71361216847266e-07\\
0.29169558760188 5.92691477826731e-07\\
0.295761665899325 7.38798835495071e-07\\
0.299884423123103 9.15521623575768e-07\\
0.304064649346707 1.12806574280056e-06\\
0.308303145656828 1.37868434477869e-06\\
0.31260072430687 1.66838998349632e-06\\
0.316958208872612 2.00018736527789e-06\\
0.321376434410028 2.37802378732553e-06\\
0.325856247615323 2.80201366389447e-06\\
0.330398506987186 3.26692521600807e-06\\
0.335004082991313 3.76611455741952e-06\\
0.339673858227221 4.29398659452914e-06\\
0.344408727597382 4.84225133517911e-06\\
0.349209598478727 5.39600198154482e-06\\
0.354077390896527 5.93643637288745e-06\\
0.359013037700707 6.4464534938643e-06\\
0.364017484744614 6.91145539336484e-06\\
0.369091691066278 7.31654351860014e-06\\
0.374236629072198 7.64723990449129e-06\\
0.379453284723694 7.89505541816254e-06\\
0.384742657725851 8.06241497731171e-06\\
0.390105761719099 8.1626990577078e-06\\
0.39554362447347 8.21628765248484e-06\\
0.40105728808555 8.24579168266737e-06\\
0.406647809178186 8.27237026942564e-06\\
0.412316259102975 8.31277827982971e-06\\
0.418063724145576 8.37628679668433e-06\\
0.423891305733878 8.4625563843229e-06\\
0.42980012064908 8.56292608304228e-06\\
0.435791301239703 8.66568478494532e-06\\
0.441865995638594 8.76252336369395e-06\\
0.448025367982949 8.85213425186464e-06\\
0.454270598637405 8.93896249543432e-06\\
0.46060288442024 9.02863950213435e-06\\
0.467023438832734 9.12378892605907e-06\\
0.473533492291712 9.22303191643497e-06\\
0.480134292365345 9.32314033804743e-06\\
0.486827104012228 9.42200510616987e-06\\
0.493613209823792 9.52005989505053e-06\\
0.500493910270095 9.61951931187413e-06\\
0.507470523949047 9.72251039909713e-06\\
0.514544387839093 9.82961799924149e-06\\
0.521716857555435 9.93964959924305e-06\\
0.528989307609815 1.0050493203729e-05\\
0.536363131673924 1.01604157044791e-05\\
0.54383974284648 1.02690501598791e-05\\
0.55142057392403 1.03775495481977e-05\\
0.559107077675529 1.04878706694343e-05\\
0.566900727120743 1.0601655426928e-05\\
0.574803015812536 1.07193655634676e-05\\
0.582815458123085 1.08401056417996e-05\\
0.590939589534097 1.09621504811708e-05\\
0.599176966931062 1.10838437876276e-05\\
0.607529168901607 1.12044005558192e-05\\
0.615997796038017 1.13242419115338e-05\\
0.624584471243962 1.14447502123299e-05\\
0.633290840045512 1.15676128390378e-05\\
0.642118570906476 1.16940828768243e-05\\
0.651069355548146 1.18244702363375e-05\\
0.660144909273489 1.19580364841777e-05\\
0.669346971295866 1.20932911178888e-05\\
0.678677305072329 1.22285442733822e-05\\
0.688137698641567 1.23624973065512e-05\\
0.697729964966554 1.24946613201714e-05\\
0.707455942281988 1.26254745410612e-05\\
0.717317494446561 1.27561083765504e-05\\
0.727316511300146 1.28880588882628e-05\\
0.737454909025955 1.30226764692411e-05\\
0.747734630517759 1.31607805903959e-05\\
0.758157645752211 1.33024573709417e-05\\
0.768725952166374 1.34470719752125e-05\\
0.779441575040495 1.35934684249607e-05\\
0.790306567886135 1.37402872711972e-05\\
0.801323012839689 1.3886311309864e-05\\
0.812493021061405 1.40307518204883e-05\\
0.823818733139961 1.41734111302769e-05\\
0.835302319502678 1.43146942762517e-05\\
0.846945980831459 1.44554825507032e-05\\
0.858751948484517 1.45969132676932e-05\\
0.870722484923992 1.47401244961934e-05\\
0.882859884149515 1.48860219218868e-05\\
0.89516647213783 1.50351099867789e-05\\
0.907644607288536 1.51874101568624e-05\\
0.920296680876042 1.53424699022675e-05\\
0.933125117507825 1.5499450549598e-05\\
0.946132375589077 1.56572716547597e-05\\
0.959320947793824 1.58147838014488e-05\\
0.972693361542618 1.59709402674527e-05\\
0.986252179486878 1.61249412755463e-05\\
1 1.6276331002404e-05\\
};
\end{axis}
\end{tikzpicture}%